\documentclass[pdflatex,sn-mathphys-num]{sn-jnl}
\usepackage[utf8]{inputenc}
\usepackage{textcomp}
\DeclareUnicodeCharacter{2212}{\textminus}

\usepackage{mathrsfs}
\usepackage{graphicx}
\usepackage{multirow}
\usepackage{amsmath,amssymb,amsfonts}%
\usepackage{amsthm}
\usepackage[title]{appendix}%
\usepackage{xcolor}
\usepackage{textcomp}
\usepackage{manyfoot}
\usepackage{booktabs}
\usepackage{algorithm}
\usepackage{algorithmicx}%
\usepackage{algpseudocode}%
\usepackage{mathtools}
\usepackage{listings}%

\theoremstyle{thmstyleone}%
\newtheorem{theorem}{Theorem}%
\newtheorem{proposition}[theorem]{Proposition}%
\newtheorem{lemma}[theorem]{Lemma}%
\newtheorem{corollary}[theorem]{Corollary}%
\newtheorem{boundarylemma}{Lemma}
 
\newtheorem{starproposition}{Proposition}

\theoremstyle{thmstyletwo}%
\newtheorem{remark}{Remark}%
\newtheorem{notation}{Notation}
\theoremstyle{thmstylethree}%
\newtheorem{definition}{Definition}%

\begin{document}

\title[Integration by Parts Formula and Sobolev Embeddings on Vector Bundles: An Intrinsic Geometrical Global Approach]{Integration by Parts Formula and Sobolev Embeddings on Vector Bundles: An Intrinsic Geometrical Global Approach}

\author[1]{\fnm{Carlos D.} \sur{Velázquez-Mendoza}}\email{danielc21@ciencias.unam.mx}
\equalcont{These authors contributed equally to this work.}

\author*[1]{\fnm{María de los Ángeles} \sur{Sandoval-Romero}}\email{selegna@ciencias.unam.mx}
\equalcont{These authors contributed equally to this work.}

\affil*[1]{\orgdiv{Department of Mathematics}, \orgname{Facultad de Ciencias, UNAM}, \orgaddress{\street{Circuito Exterior s/n}, \city{México City}, \postcode{04510}, \state{CDMX}, \country{México}}}


\abstract{
This article develops a unified and intrinsic framework for the theory of Sobolev spaces on vector bundles over Riemannian manifolds. The analytical core of our approach is an explicit higher-order geometric integration by parts formula, which characterizes the formal adjoint of the covariant derivative as a global differential operator. This identity is established on arbitrary Riemannian manifolds with boundary, without assuming completeness or compactness.

While first-order integration by parts identities are classical, explicit higher-order formulas with precise boundary terms are rarely stated in the literature. As applications of this framework, we recover the classical Meyers--Serrin theorem on arbitrary manifolds and, in the compact case, the Sobolev embedding and Rellich--Kondrashov compactness theorems, providing direct and self-contained proofs. At the end of this work, we also establish Green's formula and use it to establish norm equivalence in Sobolev spaces on vector bundles over closed manifolds for the Bochner Laplacian. As a corollary, we recover the classical norm equivalence for the Laplace--Beltrami operator on closed manifolds and provide a complete proof which, to the best of our knowledge, does not appear explicitly in the standard literature.

By emphasizing intrinsic global arguments and sharp local-to-global norm equivalence estimates, rather than ad hoc coordinate patching, this work offers a transparent and accessible foundation for the study of Sobolev spaces on vector bundles, suitable for researchers in global analysis, differential geometry, and partial differential equations.
}

\maketitle

\section{Introduction}\label{sec:intro}

Sobolev spaces of sections provide a natural setting for weak formulations and variational problems on manifolds. The unknown may be a function, a differential form, a vector field, or a section of a more general vector bundle. In each case, a connection supplies the derivatives entering the energy and the differential equation. Completion with respect to the corresponding Sobolev norm provides a Banach space in which to take limits, while the definition through weak covariant derivatives makes it possible to interpret the equation by testing and integration by parts. The identity $H^{m,p}(E)=W^{m,p}(E)$ reconciles these constructions. Density then permits the passage from identities for smooth sections to Sobolev sections; continuous embeddings control integrability, and compact embeddings provide strongly convergent subsequences in variational arguments.

These constructions have a substantial history. Aubin~\cite[Definitions~2.2--2.3, p.~32]{Aubin} and Hebey~\cite[Definition~2.1, pp.~21--22]{Hebey1} develop scalar Sobolev spaces using the completion of smooth functions with finite covariant Sobolev norm. Nicolaescu~\cite[Definition~10.2.33 and Theorem~10.2.36]{Nicolaescu2020} treats sections using weak covariant derivatives and records density, invariance, and embedding results on compact manifolds. However, the proof is left as an exercise to the reader without providing further details. In the Euclidean setting, Meyers and Serrin~\cite{MeyersSerrin1964} proved density of smooth functions in $W^{m,p}(\Omega)$ for arbitrary open sets $\Omega\subseteq\mathbb R^n$ and $1\leq p<\infty$. Their local regularization and partition-of-unity argument does not require regularity of $\partial\Omega$. Approximation by functions smooth up to a boundary is a separate issue, for which extension and localization provide the standard tools; see Leoni~\cite[Theorem~11.35, p.~330, and Theorem~13.17, pp.~424--425]{Leoni2017}.

A considerably more general density theorem was established by Guidetti, G\"uneysu, and Pallara~\cite{Guidetti2017}. On a smooth manifold without boundary, they consider Hermitian vector bundles, an arbitrary smooth positive measure $\mu$, and a finite family $\mathfrak P=\{P_1,\dots,P_N\}$ of linear differential operators with possibly different target bundles. Their spaces consist of sections $u$ for which $u$ and every $P_i u$ belong to the corresponding $L^p_\mu$ spaces, equipped with the graph norm. If $k$ is the largest order of the family, their Theorem~2.9 proves smooth density provided that, when $k\geq2$, this graph space is locally contained in $W^{k-1,p}$. The proof uses a higher-order version of Friedrichs' mollifier approximation (Proposition~2.10). Their $L^1$ elliptic regularity theorem on the Besov scale (Theorem~3.1) supplies this local regularity in additional situations, including elliptic systems for which the usual $L^p$ estimates do not extend to $p=1$.

The covariant application in \cite[Lemma~3.5 and Corollary~3.6]{Guidetti2017} is especially relevant here: local comparison with partial derivatives verifies the required regularity for the family of iterated covariant derivatives, without an ellipticity assumption on that family. It applies on arbitrary Riemannian manifolds without boundary, throughout $1\leq p<\infty$, and even to connections that need not preserve the fiber metric. Thus, the covariant density theorem recovered below is already contained in their broader theory when the boundary is empty. Chan and Czubak~\cite{ChanCzubak2024Survey} give a further detailed treatment for functions and differential forms, emphasizing the relation between weak partial and weak covariant derivatives. Our treatment uses compatible connections and derives this relation from an explicit higher-order integration by parts formula.

For compactness and embeddings of bundle-valued sections, Palais remains a foundational reference. His construction starts with Hilbert norms on jet bundles and completion \cite[Chapter~IX, Section~3, pp.~149--150]{Palais1965}, followed by interpolation to obtain a continuous Sobolev scale. Fourier series on the torus identify this scale with weighted sequence spaces, and compactness follows from control of the high-frequency tails \cite[Chapter~X, Section~2, Theorem~1 and its corollary]{Palais1965}. Local trivializations, cutoffs, and extension operators then transfer the torus results to sections of arbitrary bundles on compact manifolds, including manifolds with boundary \cite[Chapter~X, Section~4, Theorem~3]{Palais1965}. This is an $L^2$ construction, distinct from the covariant $W^{m,p}$ definition used here. Its continuing role is apparent in Bleecker and Boo\ss-Bavnbek~\cite[Theorems~7.13--7.15, pp.~200--202]{BleeckerBooss2013}, who explicitly refer to Palais for the passage to general bundles. Their detailed compactness argument in Section~7.3 is for scalar functions with fixed compact support in Euclidean space; their Proposition~16.22, p.~498, also records the general integer-order $L^p$ compact embedding for bundle-valued sections.

Other standard approaches begin with Euclidean inequalities and use coordinate localization. Aubin~\cite[Theorems~2.20 and~2.34]{Aubin} and Hebey~\cite[Theorems~2.6 and~2.9]{Hebey1} carry this out for scalar functions on compact manifolds. Nicolaescu~\cite[Theorem~10.2.36]{Nicolaescu2020} indicates how a partition of unity and the Euclidean results yield the bundle statements. On noncompact manifolds, uniform control of these local constructions is essential; Gro\ss e and Schneider~\cite[Theorems~14 and~44]{GrosseSchneider2013} study equivalent local descriptions for manifolds and bundles of bounded geometry, including fractional orders. These results address broader scales and geometric settings than the integer-order compact embeddings considered below.

The central result of this paper is a global integration by parts formula for iterated covariant derivatives, with every boundary term explicitly identified (Theorem~\ref{thm:green-higher-order}). Its proof uses metric compatibility and the divergence theorem, and requires neither completeness nor compactness of the manifold. In particular, it gives the differential expression
\[
 (\nabla^s)^*=(-1)^s(\operatorname{tr}_g\circ\nabla)^s
\]
for the formal adjoint. Testing against sections supported in the interior makes this expression applicable to weak derivatives on manifolds with or without boundary. The boundary terms retain their separate role when smooth sections meet $\partial M$.

The contribution is the explicit formula and a unified derivation of its consequences. Section~\ref{sec:structural_theorems} obtains the local weak-derivative characterization and both local norm estimates from that formula. With the classical Euclidean results as analytic input, Section~\ref{sec:global_approximation} proves smooth density and gives a detailed proof of geometric invariance on compact manifolds. Section~\ref{sec:embeddings_rellich} derives continuous Sobolev embeddings and Rellich--Kondrashov compactness, including smooth boundary. The localization and approximation needed at the boundary are supplied in Appendix~\ref{app:boundary}. Finally, Section~\ref{sec:applications} derives Green's formula and proves the classical equivalence with norms defined by powers of the Bochner Laplacian on closed manifolds. The latter proof keeps track of both Riemannian and bundle curvature. Appendix~\ref{app:star-properties} contains the tensor-contraction estimates and the Leibniz rule used in that proof.

\section{Notation and Preliminaries}\label{sec:preliminaries}

\subsection{Notation and Conventions}

Throughout this paper, $(M,g)$ denotes a smooth $n$-dimensional Riemannian manifold, possibly with smooth boundary and not necessarily complete or compact. All geometric data are smooth up to the boundary. We write $\operatorname{int}(M)=M\setminus\partial M$ and use ``closed'' to mean compact with empty boundary. The measure $d\lambda_g$ is the Riemannian-Lebesgue measure. We denote by $\mathfrak X(M)=\Gamma(TM)$ the space of smooth vector fields on $M$.

Let $E\to M$ be a smooth real vector bundle of finite rank, equipped with a smooth fiber metric $h_E$ and a compatible connection $\nabla^E$. Following \cite[Chapter~10, pp.~255--256]{LeeS}, $\Gamma(E)$ denotes its smooth sections, smooth up to $\partial M$ when present, and $\Gamma_c(E)$ denotes sections with compact support in $M$. Such support may meet $\partial M$. For weak derivatives we instead use
\[
 \Gamma_c(\operatorname{int}(M);E)
 :=\{u\in\Gamma(E)\mid\operatorname{supp}u\Subset\operatorname{int}(M)\}.
\]
These spaces coincide when $\partial M=\varnothing$. For a Euclidean open set $\Omega$, $C_c^\infty(\Omega)$ has its usual meaning. Sections of the trivial line bundle $M\times\mathbb R$ are identified with scalar functions on $M$.

The Levi--Civita connection on $TM$ and $\nabla^E$ induce a connection on each tensor bundle $T^{(k,l)}(TM)\otimes E$, also denoted by $\nabla$. The induced fiber inner product, denoted by $\langle\cdot,\cdot\rangle_{g,h_E}$, is defined on decomposable tensors by taking the product of the inner products on all factors, using $g^{-1}$ on covectors, and extending bilinearly. In coordinates and a local frame of $E$, it is
\[
\langle F,G\rangle_{g,h_E}
=g_{i_1r_1}\cdots g_{i_kr_k}g^{j_1s_1}\cdots g^{j_ls_l}
h_{ab}F^{i_1\dots i_ka}_{j_1\dots j_l}G^{r_1\dots r_kb}_{s_1\dots s_l}.
\]
Here $h_{ab}=h_E(e_a,e_b)$, and $(h^{ab})$ denotes the inverse matrix. The corresponding pointwise norm is $|\cdot|_{g,h_E}$. For sections of $E$ itself, this inner product reduces to $\langle\cdot,\cdot\rangle_{h_E}$. We write $L^p(E)$ for measurable sections with $p$-integrable pointwise norm, modulo equality almost everywhere, and $L^1_{\mathrm{loc}}(E)$ for local integrability with respect to $d\lambda_g$.

The component formulas are written in real notation. The results also apply to Hermitian bundles with compatible complex connections: applying the real statements to the underlying real bundle with metric $\operatorname{Re}h_E$ gives the same norms and weak derivatives. The complex integration identities follow by testing with $v$ and $iv$; in complex coordinates the Hermitian pairings include complex conjugation.

We adopt the Einstein summation convention for repeated tensor indices. When an ordered block is summed, each of its entries ranges independently from $1$ to $n$. The fiber sum in the lower-order part of the local adjoint will be displayed explicitly. The ordered blocks are defined after Theorem~\ref{thm:green-higher-order}; the multi-index conventions for the local Leibniz expansions are given before Lemma~\ref{lem:meyers-serrin-local}.

\subsection{Sobolev Spaces}

We use the completion and weak-derivative definitions described in the introduction. In the first, the approximating sections have finite Sobolev norm and need not have compact support. This distinction matters on noncompact manifolds and on manifolds with boundary.

\begin{definition}[The space $H^{m,p}$]\label{def:H_space}
Let $m \geq 0$ be an integer and $1 \leq p < \infty$. We define the space of smooth sections with finite Sobolev norm, denoted by $S^{m,p}(E)$, as:
\[
S^{m,p}(E) := \left\{ u \in \Gamma(E) \mathrel{\Big|} \|u\|_{m,p} < \infty \right\},
\]
where the norm is given by
\[
\|u\|_{m,p} := \left(\sum_{s=0}^{m}\int_{M}|\nabla^{s}u|_{g,h_E}^{p}\, d\lambda_{g}\right)^{\frac{1}{p}}.
\]
Here, $\nabla^s u$ denotes the $s$-th iterated covariant derivative of $u$, viewed as a section of the bundle $T^{(0,s)}(TM) \otimes E$. The Sobolev space $H^{m,p}(E)$ is defined as the completion of $S^{m,p}(E)$ with respect to this norm.
\end{definition}

To introduce the definition via weak covariant derivatives, we first recall the notion of the formal adjoint.

\begin{definition}[Weak Covariant Derivative and the space $W^{m,p}$]\label{def:weak_deriv}
Let $s\geq 0$. The \emph{formal adjoint} of the iterated covariant derivative $\nabla^{s}$, denoted by $(\nabla^{s})^{*}$, is the unique differential operator characterized by the identity
\begin{equation}\label{eq:adjoint_def}
\int_{M}\langle \nabla^{s} u, v \rangle_{g,h_E}\, d\lambda_{g}
= \int_{M}\langle u, (\nabla^{s})^{*} v \rangle_{h_E}\, d\lambda_{g},
\end{equation}
for all $u\in\Gamma_c(\operatorname{int}(M);E)$ and $v\in\Gamma_c(\operatorname{int}(M);T^{(0,s)}(TM)\otimes E)$.

A section $u\in L^{1}_{\operatorname{loc}}(E)$ is said to admit a \emph{weak covariant derivative} of order $s$ if there exists a section $v\in L^{1}_{\operatorname{loc}}(T^{(0,s)}(TM)\otimes E)$ such that
\[
\int_{M}\langle v, \psi \rangle_{g,h_E}\, d\lambda_{g}
= \int_{M}\langle u, (\nabla^{s})^{*}\psi\rangle_{h_E}\, d\lambda_{g}
\]
holds for all test sections $\psi\in\Gamma_c(\operatorname{int}(M);T^{(0,s)}(TM)\otimes E)$. If such a section $v$ exists, it is unique almost everywhere; we denote it by $\nabla_{w}^{s}u$.

Finally, the Sobolev space $W^{m,p}(E)$ is defined as:
\[
W^{m,p}(E):=\left\{ u\in L^{p}(E) \mathrel{\Biggr|} \begin{aligned} &\text{for all } 1\leq s\leq m, \text{ the weak derivative } \nabla_w^s u \\ &\text{exists and satisfies } \nabla_w^{s}u \in L^{p}(T^{(0,s)}(TM)\otimes E) \end{aligned} \right\}.
\]
This space is equipped with the same norm $\|\cdot\|_{m,p}$ defined in Definition \ref{def:H_space}.
\end{definition}

Interior test sections impose no boundary condition on $u$. In particular, $W^{m,p}(E)$ is not the closure of $\Gamma_c(\operatorname{int}(M);E)$ unless that closure happens to be the whole space. The formal adjoint above is a differential expression; boundary conditions enter when one chooses the domain of an operator realization in $L^2$.

The space $W^{m,p}(E)$ is complete. Indeed, if $(u_j)$ is Cauchy in its norm, then $u_j\to u$ in $L^p(E)$ and $\nabla_w^s u_j\to v_s$ in $L^p(T^{(0,s)}(TM)\otimes E)$ for $1\leq s\leq m$. For a fixed interior test section, H\"older's inequality allows passage to the limit in the defining identity, giving $v_s=\nabla_w^s u$. The same argument shows that the completion of the smooth finite-norm sections embeds isometrically in $W^{m,p}(E)$. Thus $H^{m,p}(E)\subseteq W^{m,p}(E)$ under this canonical identification. We write $H^m(E)=H^{m,2}(E)$.

Before writing component formulas, we fix the connection notation. In a local frame $(E_1,\dots,E_n)$ for $TM$ and a local frame $(e_1,\dots,e_r)$ for $E$, the coefficients are defined by
\[
 \nabla_{E_c}E_j=\sum_{s=1}^n\Gamma^s_{cj}E_s,
 \qquad
 \nabla^E_{E_c}e_b=\sum_{a=1}^r(A_c)^a_b e_a.
\]
For a coordinate frame $E_i=\partial_i$, the first coefficients are the Christoffel symbols of $g$; $A_c$ is the matrix of the bundle connection in the chosen frame. Unless another connection on $TM$ is explicitly specified, $\nabla$ on tensor factors is the Levi--Civita connection.

To establish an explicit local characterization of the formal adjoint $(\nabla^{s})^{*}$, we first derive the local coordinate expression of the  covariant derivative acting on tensor-valued sections.

\begin{lemma} \label{lem:connection-components}
Let $(M,g)$ be a Riemannian manifold and let $\nabla$ denote a connection on $TM$. Let $E\longrightarrow M$ be a smooth vector bundle with a connection $\nabla^{E}$. We denote by $\nabla$ the induced connection on the tensor product bundle $T^{(k,l)}(TM)\otimes E$, defined uniquely by the Leibniz rule:
\[
\nabla_X (F\otimes \sigma) = (\nabla_X F)\otimes \sigma + F \otimes \nabla^E_X \sigma,
\]
for all $X \in \Gamma(TM)$, $F\in \Gamma(T^{(k,l)}(TM))$, and $\sigma \in \Gamma(E)$.

Let $(E_{1},\dots,E_{n})$ be a local frame for $TM$ over an open set $U$, with dual coframe $(\varepsilon^{1},\dots,\varepsilon^{n})$, and let $(e_{1},\dots,e_{r})$ be a local frame for $E$ over $U$. If a section $F \in \Gamma(T^{(k,l)}(TM)\otimes E)$ is expressed locally as
\[
F = F^{i_{1}\dots i_{k}a}_{j_{1}\dots j_{l}} \, E_{i_{1}}\otimes \cdots \otimes E_{i_{k}}\otimes \varepsilon^{j_{1}}\otimes \cdots \otimes \varepsilon^{j_{l}}\otimes e_{a},
\]
then the components of the covariant derivative $\nabla_{E_c} F$ with respect to these frames are given by
\begin{align*}
(\nabla_{E_c}F)^{i_{1}\dots i_{k}a}_{j_{1}\dots j_{l}}
&= E_{c}\big(F^{i_{1}\dots i_{k}a}_{j_{1}\dots j_{l}}\big)
+ \sum_{t=1}^{k}\sum_{s=1}^{n}\Gamma^{i_{t}}_{cs}\, F^{i_{1}\dots i_{t-1}\,s\,i_{t+1}\dots i_{k}a}_{j_{1}\dots j_{l}} \\
&\quad - \sum_{t=1}^{l}\sum_{s=1}^{n}\Gamma^{s}_{c j_{t}}\, F^{i_{1}\dots i_{k}\,a}_{j_{1}\dots j_{t-1}\,s\,j_{t+1}\dots j_{l}}
+ \sum_{b=1}^{r}(A_{c})^{a}_{b}\, F^{i_{1}\dots i_{k}b}_{j_{1}\dots j_{l}},
\end{align*}
with the connection coefficients defined above. In the first correction term, $s$ replaces $i_t$, so that the upper index block is $i_1,\dots,i_{t-1},s,i_{t+1},\dots,i_k$. In the second, $s$ replaces $j_t$, giving $j_1,\dots,j_{t-1},s,j_{t+1},\dots,j_l$ in the lower block. The block before the replacement is empty when $t=1$, and the block after it is empty when $t=k$ or $t=l$, respectively. The index $s$ ranges independently from $1$ to $n$; only its position in the tensor component is being specified.
\end{lemma}

\begin{proof}
We have
\[
\nabla_{E_c} F = \nabla_{E_c} \left( \sum_{a=1}^r F^{(a)} \otimes e_a \right),
\]
where $F^{(a)} \in \Gamma(T^{(k,l)}(TM))$ denotes the tensor field with components $(F^{(a)})^{i_1 \dots i_k}_{j_1 \dots j_l} = F^{i_1 \dots i_k a}_{j_1 \dots j_l}$. Applying the Leibniz rule for the induced connection yields
\begin{equation} \label{eq:leibniz-expansion}
\nabla_{E_c} F = \sum_{a=1}^r \left( (\nabla_{E_c} F^{(a)}) \otimes e_a + F^{(a)} \otimes \nabla^E_{E_c} e_a \right).
\end{equation}
Recall that for the tensor field $F^{(a)}$, the connection components are given by (see, e.g., \cite[Proposition 4.16]{LeeR}):
\begin{align*}
(\nabla_{E_c} F^{(a)})^{i_1 \dots i_k}_{j_1 \dots j_l}
&= E_c(F^{i_1 \dots i_k a}_{j_1 \dots j_l})
+ \sum_{t=1}^{k}\sum_{s=1}^{n}\Gamma^{i_{t}}_{cs}\, F^{i_{1}\dots i_{t-1}\,s\,i_{t+1}\dots i_{k}a}_{j_{1}\dots j_{l}} \\
&\quad - \sum_{t=1}^{l}\sum_{s=1}^{n}\Gamma^{s}_{c j_{t}}\, F^{i_{1}\dots i_{k}\,a}_{j_{1}\dots j_{t-1}\,s\,j_{t+1}\dots j_{l}}.
\end{align*}
For the second term in \eqref{eq:leibniz-expansion}, we substitute $\nabla^E_{E_c} e_b = (A_c)^a_b e_a$. Relabeling the summation index for the bundle frame as $b$, we obtain
\[
\sum_{b=1}^r F^{(b)} \otimes \nabla^E_{E_c} e_b = \sum_{b=1}^r F^{(b)} \otimes \left( \sum_{a=1}^r (A_c)^a_b e_a \right) = \sum_{a=1}^r \left( \sum_{b=1}^r (A_c)^a_b F^{(b)} \right) \otimes e_a.
\]
The component of this term along the basis element $e_a$ is $\displaystyle\sum_{b=1}^r (A_c)^a_b F^{i_1 \dots i_k b}_{j_1 \dots j_l}$. Combining these contributions yields the asserted formula.
\end{proof}

\begin{remark} \label{rem:total-cov-derivative}
We regard the total covariant derivative $\nabla F$ as a section of $T^{(k,l+1)}(TM)\otimes E$. Using the canonical permutation that keeps the $E$ factor last, in the local frames defined above we write
\[
\nabla F = \sum_{c=1}^{n} (\nabla_{E_c} F) \otimes \varepsilon^c.
\]
Consequently, the components of $\nabla F$ are defined by $(\nabla F)^{i_{1}\dots i_{k} a}_{j_{1}\dots j_{l} c} := (\nabla_{E_c}F)^{i_{1}\dots i_{k} a}_{j_{1}\dots j_{l}}$. Explicitly,
\begin{align*}
(\nabla F)^{i_{1}\dots i_{k} a}_{j_{1}\dots j_{l} c}
&= E_{c}\big(F^{i_{1}\dots i_{k}a}_{j_{1}\dots j_{l}}\big)
+ \sum_{t=1}^{k}\sum_{s=1}^{n}\Gamma^{i_{t}}_{cs}\, F^{i_{1}\dots i_{t-1}\,s\,i_{t+1}\dots i_{k}a}_{j_{1}\dots j_{l}} \\
&\quad - \sum_{t=1}^{l}\sum_{s=1}^{n}\Gamma^{s}_{c j_{t}}\, F^{i_{1}\dots i_{k}\,a}_{j_{1}\dots j_{t-1}\,s\,j_{t+1}\dots j_{l}}
+ \sum_{b=1}^{r}(A_{c})^{a}_{b}\, F^{i_{1}\dots i_{k}b}_{j_{1}\dots j_{l}}.
\end{align*}
We emphasize that we adopt the convention where the differentiation index (here $c$) corresponds to the \emph{last} covariant index in the component representation.
\end{remark}

\section{The Geometric Integration by Parts Formula}\label{sec:integration_by_parts}

With the local properties of the connection established, we now turn to the analytical core of this work. The following theorem provides a rigorous global integration by parts formula for higher-order covariant derivatives. Unlike standard Euclidean formulas, this result uses the induced connection and explicitly identifies the boundary pairings required for the study of boundary value problems.

For a vector field $Y$, $\iota_Y$ denotes the contraction (interior multiplication) of $Y$ into the \emph{last covariant index} of the tensor factor. Explicitly, if $H\in \Gamma\bigl(T^{(k,m+1)}(TM)\otimes E\bigr)$, then $\iota_{Y}H\in \Gamma\bigl(T^{(k,m)}(TM)\otimes E\bigr)$ is given in local coordinates by
\[
\bigl(\iota_{Y}H\bigr)^{i_{1}\dots i_{k}a}_{j_{1}\dots j_{m}}
:= Y^{q}\,H^{i_{1}\dots i_{k}a}_{j_{1}\dots j_{m}q},
\]
where $Y=Y^q\partial_q$.

Here, $\operatorname{tr}_{g}$ denotes the metric contraction of the \emph{last two covariant indices} of $\nabla H$. Specifically, for $m\geq1$, if $H\in \Gamma\bigl(T^{(k,m)}(TM)\otimes E\bigr)$, then $\nabla H\in \Gamma\bigl(T^{(k,m+1)}(TM)\otimes E\bigr)$ and locally
\[
\bigl(\operatorname{tr}_{g}(\nabla H)\bigr)^{i_{1}\dots i_{k}a}_{j_{1}\dots j_{m-1}}
:= g^{pq}\,(\nabla H)^{i_{1}\dots i_{k}a}_{j_{1}\dots j_{m-1}p q}.
\]
That is, the derivative index is contracted with the last covariant index inherited from $H$.

When $\partial M\neq\varnothing$, $\nu$ denotes its outward unit normal and $\widetilde g$ the metric induced on $T\partial M$. These definitions also apply along the boundary. Boundary integrals below are omitted when $\partial M=\varnothing$.

\begin{theorem}[Integration by Parts for Higher-Order Derivatives]\label{thm:green-higher-order}
Let $(M,g)$ be a Riemannian manifold with boundary $\partial M$, and let $\nabla$ denote the Levi-Civita connection on $TM$. Let $E \to M$ be a smooth vector bundle equipped with a fiber metric $h_{E}$ and a connection $\nabla^{E}$ that is compatible with $h_E$. We continue to denote by $\nabla$ the induced connection on the tensor bundles $T^{(k,l)}(TM)\otimes E$.

For any integer $s\geq 1$, and sections $F\in \Gamma_{c}\bigl(T^{(k,l)}(TM)\otimes E\bigr)$ and $G\in \Gamma_{c}\bigl(T^{(k,l+s)}(TM)\otimes E\bigr)$, the following identity holds:
\begin{align*}
\int_{M}\langle \nabla^{s}F,G\rangle_{g,h_{E}}\,d\lambda_{g}
&= (-1)^{s}\int_{M}\big\langle F,(\operatorname{tr}_{g}\circ \nabla)^{s}G\big\rangle_{g,h_{E}}\,d\lambda_{g} \\
&\quad +\sum_{j=0}^{s-1}(-1)^{\,s-1-j}
\int_{\partial M}\Big\langle \nabla^{j}F,\,\iota_{\nu}\bigl((\operatorname{tr}_{g}\circ \nabla)^{\,s-1-j}G\bigr)\Big\rangle_{g,h_{E}}\,d\lambda_{\widetilde{g}},
\end{align*}
The boundary is allowed to be empty.
\end{theorem}

\begin{proof}
We proceed by induction on $s$. The strategy is to shift derivatives from $F$ to $G$ via covariant integration by parts, tracking the boundary terms that arise at each step.

\medskip
\noindent\textit{Base Case ($s=1$).}
Fix a point $p\in\operatorname{int}(M)$. Since $\nabla$ on $TM$ is the Levi-Civita connection, we can choose geodesic normal coordinates centered at $p$, with coordinate frame $(\partial_{1},\dots,\partial_{n})$. In these coordinates, $g_{ij}(p)=\delta_{ij}$ and the Levi--Civita connection coefficients vanish at $p$. In particular, $\nabla_{\partial_{i}}\partial_{j}(p)=0$. Furthermore, we choose a local orthonormal frame $(e_{1},\dots,e_{r})$ for the bundle $E$ defined in a neighborhood of $p$, such that $h_{E}(e_{a},e_{b})(p)=\delta_{ab}$.

Let $F\in\Gamma_{c}\bigl(T^{(k,l)}(TM)\otimes E\bigr)$ and $G\in\Gamma_{c}\bigl(T^{(k,l+1)}(TM)\otimes E\bigr)$. We define a map $\Phi:\mathfrak{X}(M)\longrightarrow C^{\infty}(M)$ by
\[
\Phi(Y):=\big\langle F,\,\iota_{Y}G\big\rangle_{g,h_{E}}.
\]
The map $Y\mapsto \iota_{Y}G$ is $C^{\infty}(M)$-linear in $Y$ and the inner product is bilinear, so $\Phi$ defines a 1-form. Let $X:=\Phi^{\sharp}$ be the dual vector field, characterized by $\langle X,Y\rangle_{g}=\Phi(Y)=\big\langle F,\,\iota_{Y}G\big\rangle_{g,h_{E}}$. In particular, for each coordinate vector $\partial_c$,
\begin{equation}\label{eq:X-parcial}
\langle X,\partial_{c}\rangle_{g} = \big\langle F,\,\iota_{\partial_{c}}G\big\rangle_{g,h_{E}}.
\end{equation}
At $p$, the divergence of $X$ is given by
\[
\operatorname{div}X(p) = \sum_{c=1}^{n}\big\langle \nabla_{\partial_{c}}X,\partial_{c}\big\rangle_{g}(p).
\]
Using the compatibility of $\nabla$ with $g$ and the fact that $\nabla_{\partial_{c}}\partial_{c}(p)=0$, we obtain
\begin{equation}\label{eq:divX-en-p-derivada}
\operatorname{div}X(p) = \sum_{c=1}^{n}\partial_{c}\big(\langle X,\partial_{c}\rangle_{g}\big)(p).
\end{equation}
Substituting \eqref{eq:X-parcial} into \eqref{eq:divX-en-p-derivada} and using the product rule:
\begin{equation*}
\operatorname{div}X(p) = \sum_{c=1}^{n} \Big( \big\langle \nabla_{\partial_c}F,\,\iota_{\partial_{c}}G\big\rangle_{g,h_{E}}(p) + \big\langle F,\,\nabla_{\partial_c}\bigl(\iota_{\partial_{c}}G\bigr)\big\rangle_{g,h_{E}}(p) \Big).
\end{equation*}
Using the Leibniz rule for contraction $\nabla_{\partial_c}\bigl(\iota_{\partial_{c}}G\bigr) = \iota_{\partial_{c}}\bigl(\nabla_{\partial_{c}}G\bigr) + \iota_{\nabla_{\partial_{c}}\partial_{c}}G$, and noting that the second term vanishes at $p$, we arrive at:
\begin{equation*}
\operatorname{div}X(p) = \sum_{c=1}^{n} \Big( \big\langle \nabla_{\partial_c}F,\,\iota_{\partial_{c}}G\big\rangle_{g,h_{E}}(p) + \big\langle F,\,\iota_{\partial_{c}}(\nabla_{\partial_c}G)\big\rangle_{g,h_{E}}(p) \Big).
\end{equation*}
We now identify these sums using local coordinates. Let
\[
F = F^{i_{1}\dots i_{k}a}_{j_{1}\dots j_{l}} \partial_{i_{1}}\otimes \dots\otimes \partial_{i_{k}}\otimes dx^{j_{1}}\otimes \cdots \otimes dx^{j_{l}} \otimes e_{a},\] \[G = G^{i_{1}\dots i_{k}a}_{j_{1}\dots j_{l}t} \partial_{i_{1}}\otimes \dots \otimes \partial_{i_{k}}\otimes dx^{j_{1}} \otimes \cdots \otimes dx^{j_{l}}\otimes  dx^{t}\otimes e_{a}.
\]
At $p$, where $g_{ij}=\delta_{ij}$ and $h_{ab}=\delta_{ab}$, the first sum becomes:
\begin{align*}
\sum_{c=1}^{n}\big\langle \nabla_{\partial_c}F,\,\iota_{\partial_{c}}G\big\rangle_{g,h_E}(p)
&= \sum_{c=1}^{n}\sum_{a=1}^{r}\sum_{\substack{1\le i_{1},\dots,i_{k}\le n\\ 1\le j_{1},\dots,j_{l}\le n}} (\nabla_{\partial_c}F)^{i_{1}\dots i_{k}a}_{j_{1}\dots j_{l}}(p)\, G^{i_{1}\dots i_{k}a}_{j_{1}\dots j_{l}c}(p) \\
&= \big\langle \nabla F,\,G\big\rangle_{g,h_E}(p).
\end{align*}
For the second sum, noting that $(\nabla_{\partial_c}G)^{i_{1}\dots i_{k}a}_{j_{1}\dots j_{l}t} = (\nabla G)^{i_{1}\dots i_{k}a}_{j_{1}\dots j_{l}t\,c}$, we have:
\begin{align*}
\sum_{c=1}^{n}\big\langle F,\,\iota_{\partial_{c}}(\nabla_{\partial_c}G\bigr)\big\rangle_{g,h_E}(p)
&= \sum_{c=1}^{n}\sum_{a=1}^{r}\sum_{\substack{1\le i_{1},\dots,i_{k}\le n\\ 1\le j_{1},\dots,j_{l}\le n}} F^{i_{1}\dots i_{k}a}_{j_{1}\dots j_{l}}(p)\, (\nabla G)^{i_{1}\dots i_{k}a}_{j_{1}\dots j_{l}c\,c}(p).
\end{align*}
Since $g^{pq}(p)=\delta^{pq}$ and $\bigl(\operatorname{tr}_{g}(\nabla G)\bigr)^{i_{1}\dots i_{k}a}_{j_{1}\dots j_{l}}
=
g^{pq}\,(\nabla G)^{i_{1}\dots i_{k}a}_{j_{1}\dots j_{l}p q}$, we have \[
\bigl(\operatorname{tr}_{g}(\nabla G)\bigr)^{i_{1}\dots i_{k}a}_{j_{1}\dots j_{l}}(p)
=
\sum_{c=1}^{n}(\nabla G)^{i_{1}\dots i_{k}a}_{j_{1}\dots j_{l}c\,c}(p).
\] Thus,
\[
\sum_{c=1}^{n}\big\langle F,\,\iota_{\partial_{c}}(\nabla_{\partial_c}G)\big\rangle_{g,h_E}(p) = \big\langle F,\,\operatorname{tr}_{g}(\nabla G)\big\rangle_{g,h_E}(p).
\]
Combining these results, and extending the pointwise identity to $\partial M$ by smoothness, we have $\operatorname{div}X = \langle \nabla F,G\rangle_{g,h_{E}} + \langle F,\operatorname{tr}_{g}(\nabla G)\rangle_{g,h_{E}}$ globally. Integrating and applying the Divergence Theorem yields:
\[
\int_{M}\big\langle \nabla F,G\big\rangle_{g,h_{E}}\,d\lambda_{g} = -\int_{M}\big\langle F,\operatorname{tr}_{g}(\nabla G)\big\rangle_{g,h_{E}}\,d\lambda_{g} + \int_{\partial M}\big\langle F,\,\iota_{\nu}G\big\rangle_{g,h_{E}}\,d\lambda_{\widetilde{g}}.
\]

\medskip
\noindent\textit{Inductive Step.}
Assume the identity holds for some $s\ge 1$. Let $G\in\Gamma_{c}\bigl(T^{(k,l+s+1)}(TM)\otimes E\bigr)$. Applying the base case ($s=1$) to $(\nabla^s F, G)$, where $\nabla^{s}F\in \Gamma_{c}(T^{(k,l+s)}(TM)\otimes E)$:
\[
\int_{M}\langle \nabla^{s+1}F,G\rangle_{g,h_{E}}\,d\lambda_{g} = -\int_{M}\big\langle \nabla^{s}F,\operatorname{tr}_{g}(\nabla G)\big\rangle_{g,h_{E}}\,d\lambda_{g} + \int_{\partial M}\big\langle \nabla^{s}F,\,\iota_{\nu}G\big\rangle_{g,h_{E}}\,d\lambda_{\widetilde{g}}.
\]
Define $H:=\operatorname{tr}_{g}(\nabla G)$. Applying the inductive hypothesis to $(F,H)$, we obtain:
\begin{align*}
\int_{M}\langle \nabla^{s}F,H\rangle_{g,h_{E}}\,d\lambda_{g} &= (-1)^{s} \int_{M}\big\langle F,(\operatorname{tr}_{g}\circ\nabla)^{s}H\big\rangle_{g,h_{E}}\,d\lambda_{g} \\
&\quad + \sum_{j=0}^{s-1}(-1)^{\,s-1-j} \int_{\partial M}\Big\langle \nabla^{j}F,\,\iota_{\nu}\bigl((\operatorname{tr}_{g}\circ\nabla)^{\,s-1-j}H\bigr)\Big\rangle_{g,h_{E}}\,d\lambda_{\widetilde{g}}.
\end{align*}
Substituting $H=\operatorname{tr}_{g}(\nabla G)$, we observe that
\[
\begin{aligned}
(\operatorname{tr}_{g}\circ\nabla)^{s}H
&= (\operatorname{tr}_{g}\circ\nabla)^{s+1}G,\\
(\operatorname{tr}_{g}\circ\nabla)^{\,s-1-j}H
&= (\operatorname{tr}_{g}\circ\nabla)^{\,s-j}G.
\end{aligned}
\]
Substituting this result back into the expansion of $\displaystyle\int_{M}\langle \nabla^{s+1}F,G\rangle_{g,h_{E}}\,d\lambda_{g}$, the interior integral term becomes
\[
(-1)^{s+1} \int_{M}\big\langle F,(\operatorname{tr}_{g}\circ\nabla)^{s+1}G\big\rangle_{g,h_{E}}\,d\lambda_{g}.
\]
The boundary sum becomes:
\begin{align*}
-\left( \sum_{j=0}^{s-1}(-1)^{\,s-1-j} \int_{\partial M}\Big\langle \nabla^{j}F,\,\iota_{\nu}\bigl((\operatorname{tr}_{g}\circ\nabla)^{\,s-j}G\bigr)\Big\rangle_{g,h_{E}}\,d\lambda_{\widetilde{g}} \right) \\
= \sum_{j=0}^{s-1}(-1)^{\,s-j} \int_{\partial M}\Big\langle \nabla^{j}F,\,\iota_{\nu}\bigl((\operatorname{tr}_{g}\circ\nabla)^{\,s-j}G\bigr)\Big\rangle_{g,h_{E}}\,d\lambda_{\widetilde{g}}.
\end{align*}
Finally, adding the isolated boundary term from the first step $\displaystyle\int_{\partial M}\big\langle \nabla^{s}F,\,\iota_{\nu}G\big\rangle_{g,h_{E}}\,d\lambda_{\widetilde{g}}$ (which corresponds to $j=s$ since $(-1)^{s-s}=1$ and $(\operatorname{tr}\circ\nabla)^0=\operatorname{Id}$), we extend the sum index to $s$:
\[
\sum_{j=0}^{s}(-1)^{\,s-j} \int_{\partial M}\Big\langle \nabla^{j}F,\,\iota_{\nu}\bigl((\operatorname{tr}_{g}\circ\nabla)^{\,s-j}G\bigr)\Big\rangle_{g,h_{E}}\,d\lambda_{\widetilde{g}}.
\]
This yields exactly the formula for $s+1$, completing the induction.
\end{proof}

For a fixed $s \ge 1$, we denote ordered tuples (blocks) of indices in $\{1,\dots,n\}$ as $I_s=(i_{1},\dots,i_{s})$, $J_s=(j_{1},\dots,j_{s})$, etc. We adopt the Einstein summation convention for these blocks. We define:
\[
\partial_{I_s} := \partial_{i_{s}}\cdots\partial_{i_{1}}, \quad
g^{I_s J_s} := \prod_{l=1}^{s}g^{i_{l}j_{l}}, \quad
g_{J_s T_s} := \prod_{l=1}^{s}g_{j_{l}t_{l}}.
\]
For a multi-index $\beta=(\beta_1,\dots,\beta_n)\in\mathbb N_0^n$, we write $|\beta|:=\displaystyle\sum_{i=1}^n\beta_i$ and $\partial^\beta:=\partial_1^{\beta_1}\cdots\partial_n^{\beta_n}$.

The integration by parts formula derived above is of fundamental importance because it allows us to derive an explicit coordinate expression for the formal adjoint of the $s$-th covariant derivative.

\begin{corollary}[Local Expression of the Formal Adjoint]\label{cor:local_adjoint_expression}
Let $(M,g)$ be a Riemannian manifold with or without smooth boundary, and let $E\to M$ be a vector bundle equipped with a fiber metric $h_{E}$ and a compatible connection $\nabla^{E}$. For $s\ge 1$, the formal adjoint of
\[
\nabla^{s}:\Gamma(E)\longrightarrow \Gamma\bigl(T^{(0,s)}(TM)\otimes E\bigr)
\]
is given globally by
\[
(\nabla^{s})^{*}=(-1)^{s}(\operatorname{tr}_{g}\circ\nabla)^{s}.
\]
Furthermore, in a local chart $(U,x^{1},\dots,x^{n})$, if a section $G\in \Gamma\bigl(T^{(0,s)}(TM)\otimes E\bigr)$ has components $G^{a}_{J_s} := G^{a}_{j_{1}\dots j_{s}}$, then
\begin{equation}\label{eq:adjunto formal s derivada covariante coordenadas}
\begin{split}
\bigl((\nabla^{s})^{*}G\bigr)^{a}
&= (-1)^{s}\Bigg[
\frac{1}{\sqrt{\det(g)}}\,
\partial_{I_s}
\Big(\sqrt{\det(g)}\,g^{I_s J_s}\,G^{a}_{J_s}\Big) \\
&\quad + \sum_{b=1}^{r}\sum_{|\beta|\le s-1}
\bigl(C_{\beta}\bigr)^{a\,J_s}_{b}\,\partial^{\beta}\!\bigl(G^{b}_{J_s}\bigr)
\Bigg],
\end{split}
\end{equation}
where the coefficients $\bigl(C_{\beta}\bigr)^{a\,J_s}_{b}$ are smooth functions on $U$ depending solely on the metric components $g_{ij}$ and its inverse $g^{ij}$, the connection coefficients $\Gamma^{k}_{ij}$ and $(A_{i})^{a}_{b}$, and their partial derivatives up to order $s-1$.
\end{corollary}

\begin{proof}
For $F\in\Gamma_c(\operatorname{int}(M);E)$ and $G\in\Gamma_c(\operatorname{int}(M);T^{(0,s)}(TM)\otimes E)$, the boundary terms in Theorem~\ref{thm:green-higher-order} vanish, whether or not $\partial M$ is empty. Consequently,
\[
\int_{M}\langle \nabla^{s}F,G\rangle_{g,h_{E}}\,d\lambda_{g}
=
(-1)^{s}\int_{M}\big\langle F,(\operatorname{tr}_{g}\circ \nabla)^{s}G\big\rangle_{g,h_{E}}\,d\lambda_{g}.
\]
By the uniqueness of the formal adjoint (\cite[Lemma~10.1.25]{Nicolaescu2020}), it follows that
\[
(\nabla^{s})^{*}=(-1)^{s}(\operatorname{tr}_{g}\circ\nabla)^{s}.
\]
It remains to verify the local formula \eqref{eq:adjunto formal s derivada covariante coordenadas}. Fix a chart $(U,x^{1},\dots,x^{n})$ and a local frame $(e_{1},\dots,e_{r})$ for $E|_{U}$. Let $m\ge 0$ and let $H\in \Gamma\bigl(T^{(0,m+1)}(TM)\otimes E\bigr)$ with components $H^{a}_{J_m q} = H^{a}_{j_{1}\dots j_{m}q}$. By definition, the components of the trace are
\[
((\operatorname{tr}_{g}\circ\nabla) H)^{a}_{J_m}
=
g^{pq}\,(\nabla_{\partial_{p}}H)^{a}_{J_m q}.
\]
Applying Lemma \ref{lem:connection-components} (with $k=0$, $l=m+1$), we expand the covariant derivative:
\begin{equation}\label{eq:D-raw}
\begin{split}
((\operatorname{tr}_{g}\circ\nabla)H)^{a}_{J_m}
&= g^{pq}\partial_{p}\bigl(H^{a}_{J_m q}\bigr)
-\sum_{\ell=1}^{m}\Gamma^{\alpha}_{p j_{\ell}}\,g^{pq}\,
H^{a}_{j_{1}\dots j_{\ell-1}\alpha j_{\ell+1}\dots j_{m}q} \\
&\quad -\Gamma^{\alpha}_{p q}\,g^{pq}\,H^{a}_{J_m \alpha}
+\sum_{b=1}^{r}(A_{p})^{a}_{b}\,g^{pq}H^{b}_{J_m q}.
\end{split}
\end{equation}

We recall the identity $\Gamma_{\alpha p}^{\alpha}=\frac{1}{\sqrt{\det(g)}}\partial_{p}\left(\sqrt{\det(g)}\right)$ and the formula for the derivative of the inverse metric, $\partial_{p}(g^{pq})=-\Gamma^{p}_{p\alpha}g^{\alpha q}-\Gamma^{q}_{p\alpha}g^{p\alpha}$. These imply the following divergence identity for any collection of functions $(Y_q)_{q=1}^n$:
\begin{equation}\label{eq:key-density-identity}
g^{pq}\partial_{p}\bigl(Y_{q}\bigr)-\Gamma^{\alpha}_{pq}g^{pq}Y_{\alpha}
=
\frac{1}{\sqrt{\det(g)}}\,\partial_{p}\Big(\sqrt{\det(g)}\,g^{pq}Y_{q}\Big).
\end{equation}
Indeed, expanding the right-hand side of \eqref{eq:key-density-identity}:
\begin{align*}
\frac{1}{\sqrt{\det(g)}}\partial_{p}&\left(\sqrt{\det(g)}g^{pq} Y_{q}\right) \\
&= \frac{1}{\sqrt{\det(g)}}\bigg[\partial_{p}\left(\sqrt{\det(g)}\right)g^{pq}Y_{q} + \sqrt{\det(g)}\,\partial_{p}(g^{pq})Y_{q} + \sqrt{\det(g)}\,g^{pq}\partial_{p}(Y_{q})\bigg] \\
&= \Gamma_{\alpha p}^{\alpha}g^{pq}Y_{q} + \partial_{p}(g^{pq})Y_{q} + g^{pq}\partial_{p}(Y_{q}) \\
&= \Gamma_{\alpha p}^{\alpha}g^{pq}Y_{q} + \left(-\Gamma_{p\alpha}^{p}g^{\alpha q}Y_{q}-\Gamma_{p\alpha }^{q}g^{p\alpha}Y_{q}\right) + g^{pq}\partial_{p}(Y_{q}).
\end{align*}
Observing that the first term cancels with the second (after relabeling indices), and relabeling $\alpha \leftrightarrow q$ in the third term, we recover $g^{pq}\partial_{p}(Y_{q})-\Gamma_{pq}^{\alpha}g^{pq}Y_{\alpha}$.

Applying \eqref{eq:key-density-identity} with $Y_{q}=H^{a}_{J_m q}$, equation \eqref{eq:D-raw} simplifies to:
\begin{equation}\label{eq:D-local}
\begin{split}
((\operatorname{tr}_{g}\circ\nabla)H)^{a}_{J_m}
&= \frac{1}{\sqrt{\det(g)}}\,\partial_{p}\Big(\sqrt{\det(g)}\,g^{pq}H^{a}_{J_m q}\Big) \\
&\quad -\sum_{\ell=1}^{m}\Gamma^{\alpha}_{p j_{\ell}}\,g^{pq}\,
H^{a}_{j_{1}\dots j_{\ell-1}\alpha j_{\ell+1}\dots j_{m}q}
+\sum_{b=1}^{r}(A_{p})^{a}_{b}\,g^{pq}H^{b}_{J_m q}.
\end{split}
\end{equation}

We now prove by induction on $s\ge 1$ that for any $G\in\Gamma\bigl(T^{(0,s)}(TM)\otimes E\bigr)$, the following holds:
\begin{equation}\label{eq:claim}
\begin{split}
\bigl((\operatorname{tr}_{g}\circ\nabla)^{s}G\bigr)^{a}
&= \frac{1}{\sqrt{\det(g)}}\,
\partial_{I_s}
\Big(\sqrt{\det(g)}\,g^{I_s J_s}\,G^{a}_{J_s}\Big) \\
&\quad + \sum_{b=1}^{r}\sum_{|\beta|\le s-1}
\bigl(C_{\beta}\bigr)^{a\,J_s}_{b}\,\partial^{\beta}\!\bigl(G^{b}_{J_s}\bigr).
\end{split}
\end{equation}

For $s=1$, setting $m=0$ in \eqref{eq:D-local} yields
\[
((\operatorname{tr}_{g}\circ\nabla)G)^{a}
=
\frac{1}{\sqrt{\det(g)}}\,\partial_{p}\Big(\sqrt{\det(g)}\,g^{pq}G^{a}_{q}\Big)
+\sum_{b=1}^{r}(A_{p})^{a}_{b}\,g^{pq}G^{b}_{q},
\]
which matches the form of \eqref{eq:claim} with $|\beta|= 0$.

Now, assume \eqref{eq:claim} holds for some $s\ge 1$. We prove the case for $s+1$. Let $G\in\Gamma\bigl(T^{(0,s+1)}(TM)\otimes E\bigr)$ and define $H:=(\operatorname{tr}_{g}\circ\nabla)G\in\Gamma\bigl(T^{(0,s)}(TM)\otimes E\bigr)$. By the inductive hypothesis applied to $H$:
\begin{equation}\label{eq:IH-on-H}
\begin{split}
\bigl((\operatorname{tr}_{g}\circ\nabla)^{s}H\bigr)^{a}
&= \frac{1}{\sqrt{\det(g)}}\,
\partial_{I_s}
\Big(\sqrt{\det(g)}\,g^{I_s J_s}\,H^{a}_{J_s}\Big) \\
&\quad + \sum_{b=1}^{r}\sum_{|\beta|\le s-1}
\bigl(C_{\beta}\bigr)^{a\,J_s}_{b}\,\partial^{\beta}\!\bigl(H^{b}_{J_s}\bigr).
\end{split}
\end{equation}
Since $(\operatorname{tr}_{g}\circ\nabla)^{s}H = (\operatorname{tr}_{g}\circ\nabla)^{s+1}G$, we expand $H^{a}_{J_s}$ in terms of $G$ using \eqref{eq:D-local} with $m=s$:
\begin{equation}\label{eq:H-expanded}
\begin{split}
H^{a}_{J_s}
&= \frac{1}{\sqrt{\det(g)}}\,\partial_{p}\Big(\sqrt{\det(g)}\,g^{pq}G^{a}_{J_s q}\Big) \\
&\quad -\sum_{\ell=1}^{s}\Gamma^{\alpha}_{p j_{\ell}}\,g^{pq}\,
G^{a}_{j_{1}\dots j_{\ell-1}\alpha j_{\ell+1}\dots j_{s}q}
+\sum_{b=1}^{r}(A_{p})^{a}_{b}\,g^{pq}G^{b}_{J_s q}.
\end{split}
\end{equation}

Substituting \eqref{eq:H-expanded} into the principal term of \eqref{eq:IH-on-H} splits it into three parts:
\begin{equation}\label{eq:split-I-II-III}
\frac{1}{\sqrt{\det(g)}}\,
\partial_{I_s}
\Big(\sqrt{\det(g)}\,g^{I_s J_s}\,H^{a}_{J_s}\Big)
=
\mathrm{(I)}
+\mathrm{(II)}
+\mathrm{(III)},
\end{equation}
where
\begin{align*}
\mathrm{(I)} &:= \frac{1}{\sqrt{\det(g)}}\, \partial_{I_s} \Big(g^{I_s J_s}\partial_{p}\big(\sqrt{\det(g)}\,g^{pq}G^{a}_{J_s q}\big)\Big),\\[2mm]
\mathrm{(II)} &:= -\frac{1}{\sqrt{\det(g)}}\, \partial_{I_s} \Big(\sqrt{\det(g)}\,g^{I_s J_s} \sum_{\ell=1}^{s}\Gamma^{\alpha}_{p j_{\ell}}\,g^{pq}\, G^{a}_{j_{1}\dots j_{\ell-1}\alpha j_{\ell+1}\dots j_{s}q}\Big),\\[2mm]
\mathrm{(III)} &:= \frac{1}{\sqrt{\det(g)}}\, \partial_{I_s} \Big(\sqrt{\det(g)}\,g^{I_s J_s} \sum_{b=1}^{r}(A_{p})^{a}_{b}\,g^{pq}G^{b}_{J_s q}\Big).
\end{align*}

For term $\mathrm{(I)}$, we utilize the identity $\phi \partial_p f = \partial_p(\phi f) - (\partial_p \phi)f$. Applying this with $\phi = g^{I_s J_s}$ and $f = \sqrt{\det(g)}\,g^{pq}G^{a}_{J_s q}$, we have:
\[
g^{I_s J_s}\partial_{p}\big(\sqrt{\det(g)}\,g^{pq}G^{a}_{J_s q}\big)
=
\partial_p \Big( g^{I_s J_s} \sqrt{\det(g)}\,g^{pq}G^{a}_{J_s q} \Big)
-
\partial_p \big( g^{I_s J_s} \big) \sqrt{\det(g)}\,g^{pq}G^{a}_{J_s q}.
\]
Substituting this back into $\mathrm{(I)}$:
\begin{align}
\mathrm{(I)}
&= \frac{1}{\sqrt{\det(g)}}\, \partial_{I_s} \partial_p \Big( \sqrt{\det(g)}\, g^{I_s J_s} g^{pq} \, G^{a}_{J_s q} \Big)
\label{eq:I-expanded} \\
&\quad - \frac{1}{\sqrt{\det(g)}}\, \partial_{I_s} \Big( \partial_p \big( g^{I_s J_s} \big) \sqrt{\det(g)}\,g^{pq}G^{a}_{J_s q} \Big).
\nonumber
\end{align}
The first term of \eqref{eq:I-expanded} is exactly the principal term required for the order $s+1$. If we define the block indices for order $s+1$ as $I_{s+1} = (I_s, p)$ and $J_{s+1} = (J_s, q)$, then $g^{I_{s+1}J_{s+1}} = g^{I_s J_s}g^{pq}$ and $\partial_{I_{s+1}} = \partial_{I_s}\partial_p$. Thus, this term is
\[
\frac{1}{\sqrt{\det(g)}}\, \partial_{I_{s+1}} \Big( \sqrt{\det(g)}\, g^{I_{s+1} J_{s+1}} \, G^{a}_{J_{s+1}} \Big).
\]

The second term of \eqref{eq:I-expanded} involves $\partial_{I_s}$ acting on a term containing $G$ with no derivatives. By the Leibniz rule, the highest order of derivatives acting on $G$ in this term is $s$. Thus, there exist smooth coefficients $\bigl(B_{\beta}\bigr)$ such that this remainder term is of the form:
\begin{equation}\label{eq:I-remainder}
\sum_{c=1}^{r}\sum_{|\beta|\le s} \bigl(B_{\beta}\bigr)^{a\,J_{s+1}}_{c}\, \partial^{\beta}\!\bigl(G^{c}_{J_{s+1}}\bigr).
\end{equation}

For $\mathrm{(II)}$, the term inside $\partial_{I_s}$ contains $G$ algebraically (zero derivatives). Thus, applying $\partial_{I_s}$ yields terms with at most $s$ derivatives on $G$. Similarly, for $\mathrm{(III)}$, the dependence on $G$ is linear without additional derivatives, so expanding yields terms of order at most $s$. Combining these, there exist coefficients $\bigl(\widetilde{B}_{\beta}\bigr)$ such that:
\begin{equation}\label{eq:II-III-bound}
\mathrm{(II)} + \mathrm{(III)} = \sum_{c=1}^{r}\sum_{|\beta|\le s} \bigl(\widetilde{B}_{\beta}\bigr)^{a\,J_{s+1}}_{c}\, \partial^{\beta}\!\bigl(G^{c}_{J_{s+1}}\bigr).
\end{equation}

Finally, consider the lower-order term from \eqref{eq:IH-on-H}:
\[
\sum_{b=1}^{r}\sum_{|\beta|\le s-1} \bigl(C_{\beta}\bigr)^{a\,J_s}_{b}\,\partial^{\beta}\!\bigl(H^{b}_{J_s}\bigr).
\]
Substituting \eqref{eq:H-expanded} into $H^{b}_{J_s}$, we see that $H$ involves first-order derivatives of $G$. Therefore, applying $\partial^\beta$ (where $|\beta|\le s-1$) to $H$ yields derivatives of $G$ of order at most $(s-1)+1 = s$. Thus, there exist coefficients $\bigl(\overline{D}_{\beta}\bigr)$ such that this sum reduces to:
\begin{equation}\label{eq:lowterm-from-IH}
\sum_{c=1}^{r}\sum_{|\beta|\le s} \bigl(\overline{D}_{\beta}\bigr)^{a\,J_{s+1}}_{c}\, \partial^{\beta}\!\bigl(G^{c}_{J_{s+1}}\bigr).
\end{equation}

Combining the principal term from \eqref{eq:I-expanded} with the remainder terms \eqref{eq:I-remainder}, \eqref{eq:II-III-bound}, and \eqref{eq:lowterm-from-IH}, we conclude that there exist smooth coefficients $\bigl(C_{\beta}\bigr)$ with $|\beta|\le s$ such that
\begin{align*}
\bigl((\operatorname{tr}_{g}\circ\nabla)^{s+1}G\bigr)^{a}
&= \frac{1}{\sqrt{\det(g)}}\,
\partial_{I_{s+1}}
\Big(\sqrt{\det(g)}\,g^{I_{s+1} J_{s+1}}\,G^{a}_{J_{s+1}}\Big) \\
&\quad + \sum_{c=1}^{r}\sum_{|\beta|\le s}
\bigl(C_{\beta}\bigr)^{a\,J_{s+1}}_{c}\,
\partial^{\beta}\!\bigl(G^{c}_{J_{s+1}}\bigr),
\end{align*}
which is exactly the form of \eqref{eq:claim} for $s+1$. This completes the induction.
\end{proof}
\section{Local-to-Global Estimates and Structural Theorems}\label{sec:structural_theorems}

Having established the global integration by parts formula and the local structure of the formal adjoint, we now compare weak covariant derivatives with weak derivatives of component functions. These local comparisons and the independence of Sobolev spaces from smooth geometric choices on compact bases are classical; see \cite[Theorem~10.2.36(a)]{Nicolaescu2020} for bundles, \cite[Proposition~2.2]{Hebey1} for scalar functions, and \cite[Theorem~2]{VelazquezSandoval2026} for a detailed scalar proof. For tangent vectors, Lee~\cite[Lemma~2.53, p.~37]{LeeR} proves the local comparison between a Riemannian norm and the Euclidean norm on a compact coordinate set. The elementary fiber-metric comparison below supplies the uniform constants used in localization.

\begin{lemma}\label{lem:metric_equivalence}
Let $M$ be a smooth compact manifold (with or without boundary), and let $E \to M$ be a smooth vector bundle of finite rank. Let $h$ and $\widetilde{h}$ be two smooth fiber metrics on $E$. Then, there exist constants $c, C > 0$ such that for all $p \in M$ and all $v \in E_p$,
\[
c\,h_{p}(v,v)\leq \widetilde{h}_{p}(v,v)\leq C\,h_{p}(v,v).
\]
\end{lemma}

\begin{proof}
Let $p \in M$. Consider the unit sphere in the fiber $E_p$ with respect to the metric $h$:
\[
S_{p}:=\{v\in E_{p}\mid h_{p}(v,v)=1\}.
\]
Since $E_p$ is a finite-dimensional vector space, $S_p$ is a closed and bounded set, and thus compact. The function $f_{p}:S_{p}\longrightarrow\mathbb{R}$ defined by $f_{p}(v)=\widetilde{h}_{p}(v,v)$ is continuous, so it attains a maximum. Consequently, the function $f:M\longrightarrow\mathbb{R}$ given by
\[
f(p):=\max_{v\in S_{p}}f_{p}(v)
\]
is well-defined. Since the bundle and the metrics are smooth, $f$ is continuous on $M$. As $M$ is compact, there exists a point $p_{0}\in M$ such that $f(p)\leq f(p_{0})=:C$ for all $p\in M$.

Moreover,
\[
C=f(p_{0})=\max_{v\in S_{p_{0}}}\widetilde{h}_{p_{0}}(v,v)>0,
\]
since $\widetilde{h}$ is positive definite and $v\neq 0$ for all $v\in S_{p_{0}}$. Thus, for all $p\in M$ and all $v\in S_{p}$, we have
\[
\widetilde{h}_{p}(v,v)\leq C.
\]
By an analogous argument, considering the minimum, there exists a constant $c>0$ such that for all $p\in M$ and all $v\in S_{p}$,
\[
c\leq \widetilde{h}_{p}(v,v).
\]

Now, let $v\in E_{p}\setminus\{0\}$. By homogeneity, the vector $\frac{v}{|v|_{h_{p}}}$ belongs to $S_{p}$ (where $|v|_{h_p} = \sqrt{h_p(v,v)}$). Therefore,
\[
c\leq
\widetilde{h}_{p}\!\left(\frac{v}{|v|_{h_{p}}},\frac{v}{|v|_{h_{p}}}\right)
\leq C.
\]
Using the bilinearity of the metric, this implies
\[
c\leq \frac{1}{|v|^{2}_{h_{p}}}\widetilde{h}_{p}(v,v)
\leq C,
\]
or equivalently,
\[
c\,h_{p}(v,v)\leq \widetilde{h}_{p}(v,v)\leq C\,h_{p}(v,v).
\]
The case $v=0$ holds trivially, as all terms vanish.
\end{proof}
The proof of Lemma~\ref{lem:metric_equivalence} also applies on any compact subset of a smooth base. The local characterization below is a quantitative form of a known fact, rather than a new characterization of Sobolev spaces. Compare \cite[Lemma~3.5 and Corollary~3.6]{Guidetti2017}, where $\nabla=d+A$ gives the local covariant comparison and its iteration, and \cite{ChanCzubak2024Survey} for functions and forms. Here both directions, at every integer order, are obtained from the formal adjoint in Corollary~\ref{cor:local_adjoint_expression}, together with explicit local norm estimates.

We use regular coordinate balls and regular boundary coordinate half-balls: the chart and the bundle frame extend smoothly to a neighborhood of the compact closure $\overline U$. Such coordinate neighborhoods form a basis; see \cite[Proposition~1.19 and Exercise~1.42]{LeeS}. In a boundary chart, $\phi(U)=B(0,R)\cap\{x_n\geq0\}$. Set
\[
 \Omega_U:=\phi(U\cap\operatorname{int}(M));
\]
this is a Euclidean ball or the open half-ball $B(0,R)\cap\{x_n>0\}$. All coefficient functions and the finitely many derivatives used below are bounded on $\overline U$, and the metric matrices have uniformly positive eigenvalues there. Integrals over $U$ are identified with integrals over $\Omega_U$, since $\partial M$ has zero $n$-dimensional Riemannian measure.

For a block $I_s$, let $\beta_{I_s}\in\mathbb N_0^n$ count how often each coordinate index occurs, so that $\partial^{\beta_{I_s}}=\partial_{I_s}$. A block and its associated multi-index are different objects: summing over $I_s$ means summing over all $n^s$ ordered blocks, whereas $\displaystyle\sum_{|\beta|\leq s-1}$ is over distinct multi-indices $\beta\in\mathbb N_0^n$. We use $\alpha\leq\beta$ to mean $\alpha_i\leq\beta_i$ for every $i\in\{1,\dots,n\}$, and $\alpha<\beta$ to mean $\alpha\leq\beta$ with $\alpha\neq\beta$. Thus $\displaystyle\sum_{\alpha<\beta_{T_s}}$ has a fixed target block $T_s$ and ranges over precisely these proper sub-multi-indices, with coefficient
\[
 \binom{\beta_{T_s}}{\alpha}
 :=\prod_{i=1}^{n}\binom{(\beta_{T_s})_i}{\alpha_i}.
\]
If derivatives depending only on the associated multi-index are grouped from a block sum, their multiplicity is $\frac{s!}{\beta!}$, where $\beta!:=\displaystyle\prod_{i=1}^{n}\beta_i!$.

\begin{lemma}[Local Norm Equivalence]\label{lem:meyers-serrin-local}
Let $(M,g)$ be a Riemannian manifold with Levi--Civita connection $\nabla$. Let $m\geq 0$, $1\leq p<\infty$, and let $(U,\phi)$ be a regular coordinate ball or boundary coordinate half-ball in $M$, and let $E\to M$ be a smooth vector bundle of rank $r$ equipped with a fiber metric $h_{E}$ and a compatible connection $\nabla^{E}$. Let $F\in L^{1}_{\operatorname{loc}}(E)$ and let $\{e_{1},\dots,e_{r}\}$ be a local frame for $E$ over $U$, such that $F$ is expressed locally as $F=\displaystyle\sum_{a=1}^{r}F^{a}e_{a}$.

Then,
\[
F|_{U}\in W^{m,p}(E|_{U}) \quad \text{if and only if} \quad
F^{a}\circ \phi^{-1}\in W^{m,p}(\Omega_U)
\]
for all $1\le a\le r$.

Furthermore, there exist constants $C_{m},c_{m}>0$ (independent of $F$) such that for all $F\in W^{m,p}(E|_{U})$,
\[
c_{m}\sum_{a=1}^{r}
\left\|F^{a}\circ \phi^{-1}\right\|_{W^{m,p}(\Omega_U)}\leq \|F\|_{W^{m,p}(E|_{U})}
\leq
C_{m}\sum_{a=1}^{r}
\left\|F^{a}\circ \phi^{-1}\right\|_{W^{m,p}(\Omega_U)}.
\]
\end{lemma}

\begin{proof}
The case $m=0$ is the comparison of the fiber metrics and measures on $\overline U$. For $m\geq1$, every test section below is supported in $\operatorname{int}(U)=U\cap\operatorname{int}(M)$, so its components are test functions in $\Omega_U$. In particular, all weak integrations by parts remain interior computations in a boundary chart.

\noindent\textit{($\Rightarrow$)} First, assume that $F|_{U}\in W^{m,p}(E|_{U})$. Then, for each $s\le m$, the $s$-th weak covariant derivative of $F|_{U}$ exists; let us denote it by $G\in L^{p}\big(T^{(0,s)}(TU)\otimes E|_{U}\big)$. By definition, for all $\Psi\in \Gamma_c\big(\operatorname{int}(U);T^{(0,s)}(TU)\otimes E|_U\big)$, the following holds:
\[
\int_{U}\langle F,(\nabla^{s})^{*}\Psi\rangle_{h_E}\,d\lambda_{g}
=\int_{U}\langle G,\Psi\rangle_{g,h_E}\,d\lambda_{g}.
\]
In local coordinates, writing $F= F^{a} e_a$ and $\Psi = \Psi^{c}_{J_s}\, dx^{J_s}\otimes e_c$, this identity reads:
\[
\int_{U}
h_{bc}\,
F^{b}\,
\big((\nabla^{s})^{*}\Psi\big)^{c}
\,d\lambda_{g}
=\int_{U}
g^{I_s J_s}\,h_{bc}\,
G^{b}_{I_s}\,
\Psi^{c}_{J_s}
\,d\lambda_{g}.
\]

Fix a target block $T_s=(t_{1},\dots,t_{s})$ and a fixed fiber index $a$. Given $\psi\in C_{c}^{\infty}(\Omega_U)$, we define a convenient test tensor field $\Psi$ by its components:
\begin{equation}\label{eq:test_tensor_convenient}
\Psi^{c}_{J_s}
=\frac{1}{\sqrt{\det(g)}}\,h^{ca}\,
g_{J_s T_s}\,
(\psi\circ \phi).
\end{equation}
By Corollary~\ref{cor:local_adjoint_expression}, the formal adjoint acts on $\Psi$ as
\[
\big((\nabla^{s})^{*}\Psi\big)^{c} =
(-1)^{s}\left[\frac{1}{\sqrt{\det(g)}}
\partial_{I_s}\left(
\sqrt{\det(g)}g^{I_{s}J_{s}}\Psi^{c}_{J_s}\right)+\sum_{d=1}^{r}\sum_{|\beta|\le s-1}
\big(D_{\beta}\big)^{c\, J_s}_{d}
\partial^{\beta}
\Psi^{d}_{J_s}
\right].
\]
Substituting \eqref{eq:test_tensor_convenient} into this expression and using the facts that $\frac{\sqrt{\det(g)}}{\sqrt{\det(g)}}=1$ and $g^{I_{s}J_{s}}g_{J_{s}T_{s}}=\delta^{I_{s}}_{T_{s}}$, we obtain
\[
\begin{aligned}
\big((\nabla^{s})^{*}\Psi\big)^{c}
&=(-1)^{s}\Biggl[\frac{1}{\sqrt{\det(g)}}\partial_{T_s}
\left(h^{ca}(\psi\circ \phi)\right)\\
&\qquad+\sum_{d=1}^{r}\sum_{|\beta|\le s-1}
\big(D_{\beta}\big)^{c\, J_s}_{d}
\partial^{\beta}\left(\frac{h^{da}g_{J_s T_s}}{\sqrt{\det(g)}}
(\psi\circ \phi)\right)\Biggr].
\end{aligned}
\]

Let $I_{1}:=\displaystyle\int_{U}\langle F,(\nabla^{s})^{*}\Psi\rangle_{h_E}\,d\lambda_{g}$ and $I_{2}:=\displaystyle\int_{U}\langle G,\Psi\rangle_{g,h_{E}}\,d\lambda_{g}$. Substituting \eqref{eq:test_tensor_convenient} into $I_{2}$ and noting that $g^{I_s J_s}g_{J_s T_s} = \delta^{I_s}_{T_s}$, we obtain
\begin{equation}\label{eq:I2_forward}
I_{2}=\int_{U}\frac{1}{\sqrt{\det(g)}}\delta^{I_s}_{T_s}\delta_{b}^{a}G^{b}_{I_s}(\psi\circ \phi)\,d\lambda_{g}
=\int_{\Omega_U}(G^{a}_{T_s}\circ \phi^{-1})\psi \,d\lambda_{n}.
\end{equation}

For $I_{1}$, we use the Leibniz formula on $\partial_{T_s}$:
\[
\partial_{T_s}\left(h^{ca}(\psi\circ \phi)\right)
= h^{ca}\partial_{T_s}(\psi\circ \phi) +\sum_{\alpha<\beta_{T_s}}\binom{\beta_{T_s}}{\alpha}\partial^{\alpha}(\psi\circ \phi)\partial^{\beta_{T_s}-\alpha}\left(h^{ca}\right).
\]
Substituting back into the integral $I_{1}$, we obtain
\begin{align*}
I_{1} &= (-1)^{s}\int_{U}\frac{1}{\sqrt{\det(g)}}F^{a}\partial_{T_s}(\psi\circ \phi) \,d\lambda_{g} \\
&\quad + (-1)^{s}\sum_{\alpha<\beta_{T_s}}\int_{U}\frac{1}{\sqrt{\det(g)}}h_{bc}F^{b}\binom{\beta_{T_s}}{\alpha}\partial^{\alpha}(\psi\circ \phi)\partial^{\beta_{T_s}-\alpha}\left(h^{ca}\right)\,d\lambda_{g} \\
&\quad + \sum_{d=1}^{r}\sum_{|\beta|\le s-1}(-1)^{s}\int_{U}h_{bc}F^{b}\big(D_{\beta}\big)^{c\, J_s}_{d}
\partial^{\beta}\left(\frac{h^{da}g_{J_s T_s}}{\sqrt{\det(g)}}
(\psi\circ \phi)\right)\,d\lambda_{g}.
\end{align*}

We now proceed by induction on $s$. For $s=1$, we have $T_1=(t_1)$, $J_{1}=(j_{1})$, so the expression reduces to
\begin{align*}
I_{1} &= -\int_{\Omega_U}(F^{a}\circ \phi^{-1})\partial_{t_{1}}\psi \,d\lambda_{n}
-\int_{\Omega_U}((h_{bc}F^{b})\circ \phi^{-1})\partial_{t_{1}}\left(h^{ca}\circ\phi^{-1}\right)\psi\,d\lambda_{n} \\
&\quad -\sum_{d=1}^{r}\int_{\Omega_U}\left(\left(h_{bc}F^{b}(D_{0})^{c\, j_1}_{d}h^{da}g_{j_1 t_1}\right)\circ\phi^{-1}\right)\psi\,d\lambda_{n}.
\end{align*}
Define
\[
H_{t_{1}}^{a}:=h_{bc}F^{b}\partial_{t_{1}}\left(h^{ca}\right)+\sum_{d=1}^{r}h_{bc}F^{b}(D_{0})^{c\, j_1}_{d}h^{da}g_{j_1 t_1}.
\]
Since $F|_U \in W^{m,p}(E|_{U})$ implies that $F^{b}\circ\phi^{-1}\in L^p(\Omega_U)$, and the metric coefficients are smooth on the compact set $\overline{\Omega_U}$, it follows that $H_{t_{1}}^{a}\circ \phi^{-1}\in L^{p}(\Omega_U)$. Since $I_1 = I_2$, we have
\[
-\int_{\Omega_U}(F^{a}\circ \phi^{-1})\partial_{t_{1}}\psi \,d\lambda_{n}
=\int_{\Omega_U}\big((G_{t_{1}}^{a}\circ \phi^{-1})+(H^{a}_{t_{1}}\circ \phi^{-1})\big)\psi \,d\lambda_{n},
\]
which shows that $F^{a}\circ\phi^{-1}\in W^{1,p}(\Omega_U)$.

Now assume the statement holds up to order $s-1$ for $s\leq m$. The inductive hypothesis guarantees that the components of $F$ in $U$ have weak derivatives of order $\leq s-1$, so we can integrate by parts in the weak sense for the lower-order terms.

First, consider the terms coming from the Leibniz expansion in $I_1$:
\begin{align*}
(-1)^{s}&\sum_{\alpha<\beta_{T_s}}\int_{U}\frac{h_{bc}F^{b}}{\sqrt{\det(g)}}\binom{\beta_{T_s}}{\alpha}\partial^{\alpha}(\psi\circ \phi)\partial^{\beta_{T_s}-\alpha}\left(h^{ca}\right)\,d\lambda_{g} \\
&=(-1)^{s}\sum_{\alpha<\beta_{T_s}}\int_{\Omega_U}\left[\left(h_{bc}F^{b}\binom{\beta_{T_s}}{\alpha}\partial^{\beta_{T_s}-\alpha}\left(h^{ca}\right)\right)\circ \phi^{-1}\right]\partial^{\alpha}\psi \,d\lambda_{n}.
\end{align*}
Since $|\alpha| < |\beta_{T_s}| = s$, we have $|\alpha| \le s-1$. Integrating by parts (weakly) $|\alpha|$ times, we transfer the derivatives from $\psi$ to the coefficient (which involves $F^b$), so the previous term equals
\[
\sum_{\alpha<\beta_{T_s}}(-1)^{s+|\alpha|}\int_{\Omega_U}\partial^{\alpha}\left[\left(h_{bc}F^{b}\binom{\beta_{T_s}}{\alpha}\partial^{\beta_{T_s}-\alpha}\left(h^{ca}\right)\right)\circ \phi^{-1}\right]
\psi\, d\lambda_{n}.
\]

Second, consider the terms in $I_1$ involving the connection coefficients ($D_{\beta}$). Using the same change of variables and recalling that $|\beta|\le s-1$, we get
\begin{align*}
&\sum_{d=1}^{r}\sum_{|\beta|\le s-1}(-1)^{s}\int_{U}h_{bc}F^{b}\big(D_{\beta}\big)^{c\, J_s}_{d}
\partial^{\beta}\left(\frac{h^{da}g_{J_s T_s}}{\sqrt{\det(g)}}
(\psi\circ \phi)\right)\,d\lambda_{g} \\
&=(-1)^{s}\sum_{d=1}^{r}\sum_{|\beta|\le s-1}\int_{\Omega_U}\left[\left(h_{bc}F^{b}\big(D_{\beta}\big)^{c\, J_s}_{d}\sqrt{\det(g)}\right)\circ \phi^{-1}\right]
\\
&\qquad\qquad\times\partial^{\beta}\left(\left(\frac{h^{da}g_{J_s T_s}}{\sqrt{\det(g)}}
\circ \phi^{-1}\right)\psi\right)\,d\lambda_{n}.
\end{align*}
Integrating by parts $|\beta|$ times yields
\[
\begin{aligned}
&(-1)^{s}\sum_{d=1}^{r}\sum_{|\beta|\le s-1}(-1)^{|\beta|}\int_{\Omega_U}\partial^{\beta}\left[\left(h_{bc}F^{b}\big(D_{\beta}\big)^{c\, J_s}_{d}\sqrt{\det(g)}\right)\circ \phi^{-1}\right]\\
&\qquad\qquad\times\left(\frac{h^{da}g_{J_s T_s}}{\sqrt{\det(g)}}
\circ \phi^{-1}\right)\psi\,d\lambda_{n}.
\end{aligned}
\]

We now define $\widetilde{G}^{a}_{T_s}$ on $\Omega_U$ by grouping $G^{a}_{T_s}\circ \phi^{-1}$ (from $I_2$) and moving the terms calculated above from $I_1$ to the other side (introducing an extra factor of $-1$, which becomes $(-1)^{s+1}$):
\begin{align*}
\widetilde{G}^{a}_{T_s} &:= G^{a}_{T_s}\circ \phi^{-1}
+\sum_{\alpha<\beta_{T_s}}(-1)^{s+1+|\alpha|}\partial^{\alpha}\left[\left(h_{bc}F^{b}\binom{\beta_{T_s}}{\alpha}\partial^{\beta_{T_s}-\alpha}\left(h^{ca}\right)\right)\circ \phi^{-1}\right] \\
&\quad + \sum_{d=1}^{r}\sum_{|\beta|\le s-1}(-1)^{s+1+|\beta|}\partial^{\beta}\left[\left(h_{bc}F^{b}\big(D_{\beta}\big)^{c\, J_s}_{d}\sqrt{\det(g)}\right)\circ \phi^{-1}\right]\left(\frac{h^{da}g_{J_s T_s}}{\sqrt{\det(g)}}
\circ \phi^{-1}\right).
\end{align*}

The inductive hypothesis states that $F^{b}\circ\phi^{-1}\in W^{s-1,p}(\Omega_U)$ for all $b\in \{1,\dots,r\}$. Since $\widetilde{G}_{T_s}^{a}$ consists of smooth bounded functions (defined on $\overline{\Omega_U}$) and weak derivatives of $F^{b}\circ\phi^{-1}$ of order at most $s-1$, we conclude that $\widetilde{G}_{T_s}^{a}\in L^{p}(\Omega_U)$. Using the equality $I_{1}=I_{2}$ and the definition of $\widetilde{G}$, we obtain
\[
(-1)^{s}\int_{\Omega_U}(F^{a}\circ \phi^{-1})\partial_{T_s}\psi \,d\lambda_{n}
= \int_{\Omega_U}\widetilde{G}_{T_s}^{a}\psi \,d\lambda_{n}.
\]
This shows that $F^{a}\circ \phi^{-1}$ possesses a weak derivative $\partial_{T_s}(F^a \circ \phi^{-1}) = \widetilde{G}_{T_s}^{a} \in L^{p}(\Omega_U)$. Thus, $F^{a}\circ \phi^{-1}\in W^{s,p}(\Omega_U)$.
Since $s\le m$ was arbitrary, we conclude that $F^{a}\circ\phi^{-1}\in W^{m,p}(\Omega_U)$ for all $a \in \{1,\dots,r\}$.

\medskip
\noindent\textit{($\Leftarrow$)} Now assume that $F^{a}\circ\phi^{-1}\in W^{m,p}(\Omega_U)$ for all $a\in\{1,\dots,r\}$. Fix $s\le m$ and let $\Psi\in\Gamma_c\big(\operatorname{int}(U);T^{(0,s)}(TU)\otimes E|_U\big)$. We have
\[
\int_{U}\langle F,(\nabla^{s})^{*}\Psi\rangle_{h_E}\,d\lambda_{g}
=
\int_{U}
h_{ab}\,
F^{a}
\big((\nabla^{s})^{*}\Psi\big)^{b}\,
d\lambda_{g}.
\]
By Corollary~\ref{cor:local_adjoint_expression}, the adjoint in coordinates takes the form
\[
\big((\nabla^{s})^{*}\Psi\big)^{b} = (-1)^{s}\left[\frac{1}{\sqrt{\det(g)}}
\partial_{I_s}\left(
\sqrt{\det(g)}g^{I_{s}J_{s}}\Psi^{b}_{J_s}\right)
+\sum_{c=1}^{r}\sum_{|\beta|\le s-1}
\big(D_{\beta}\big)^{b\, J_s}_{c}\,
\partial^{\beta}
\Psi^{c}_{J_s}
\right].
\]
Then the integral splits into two terms, $I_{1}$ and $I_{2}$:
\begin{align*}
\int_{U}\langle F,(\nabla^{s})^{*}\Psi\rangle_{h_E}\,d\lambda_{g}
&= \underbrace{(-1)^{s}\int_{U}h_{ab}F^{a}\frac{1}{\sqrt{\det(g)}}
\partial_{I_s}\left(
\sqrt{\det(g)}g^{I_{s}J_{s}}\Psi^{b}_{J_s}\right)\,d\lambda_{g}}_{I_{1}} \\
&\quad + \underbrace{(-1)^{s}\sum_{c=1}^{r}\sum_{|\beta|\leq s-1}\int_{U}h_{ab}F^{a}\big(D_{\beta}\big)^{b\, J_s}_{c}\,
\partial^{\beta}
\Psi^{c}_{J_s}\,d\lambda_{g}}_{I_{2}}.
\end{align*}
Changing the domain of integration to $\Omega_U$ and writing the compositions explicitly:
\begin{align*}
I_1 &= (-1)^{s}\int_{\Omega_U}\left[\left(h_{ab}F^{a}\right)\circ\phi^{-1}\right]
\partial_{I_s}\left(\left(
\sqrt{\det(g)}g^{I_{s}J_{s}}\Psi^{b}_{J_s}\right)\circ\phi^{-1}\right)\,d\lambda_{n}, \\
I_2 &= (-1)^{s}\sum_{c=1}^{r}\sum_{|\beta|\leq s-1}\int_{\Omega_U}\left[\left(h_{ab}F^{a}(D_{\beta})^{b\, J_s}_{c}\sqrt{\det(g)}\right)\circ \phi^{-1}\right]\partial^{\beta}(\Psi^{c}_{J_s}\circ \phi^{-1})\,d\lambda_{n}.
\end{align*}

We now integrate by parts weakly in $I_1$ (valid since $F^{a}\circ \phi^{-1}\in W^{m,p}(\Omega_U)$ and $s \le m$):
\[
I_{1} = \int_{\Omega_U}\partial_{I_s}\left[\left(h_{ab}\circ\phi^{-1}\right)(F^{a}\circ \phi^{-1})\right]\left(\sqrt{\det(g)}g^{I_{s}J_{s}}\Psi^{b}_{J_s}\circ\phi^{-1}\right)\,d\lambda_{n}.
\]
Applying the Leibniz rule to the bracketed term:
\begin{align*}
I_{1} &= \int_{\Omega_U}\Bigl[\bigl(h_{ab}g^{I_s J_s}\sqrt{\det(g)}\circ\phi^{-1}\bigr)\partial_{I_s}(F^{a}\circ\phi^{-1})\Bigr](\Psi^{b}_{J_s}\circ\phi^{-1})\,d\lambda_{n} \\
&\quad + \int_{\Omega_U}\Bigg[\sum_{\substack{I_{s}\\\alpha<\beta_{I_s}}}\binom{\beta_{I_s}}{\alpha}\partial^{\beta_{I_s}-\alpha}\bigl(h_{ab}\circ\phi^{-1}\bigr)\partial^{\alpha}(F^{a}\circ \phi^{-1}) \\
&\qquad\qquad\qquad \times \left(\bigl(g^{I_s J_s}\sqrt{\det(g)}\Psi^{b}_{J_s}\bigr)\circ\phi^{-1}\right)\Bigg]\,d\lambda_{n}.
\end{align*}
Define the coefficient functions $(B_{\gamma})_{ab}^{J_s}$ on $\Omega_U$ for $|\gamma| \le s-1$ as
\[
(B_{\gamma})_{ ab}^{J_s} := \sum_{\substack{I_{s}\\\gamma < \beta_{I_s}}}\binom{\beta_{I_s}}{\gamma}\partial^{\beta_{I_s}-\gamma}\bigl(h_{ab}\circ\phi^{-1}\bigr)\left(\bigl(g^{I_s J_s}\sqrt{\det(g)}\bigr)\circ\phi^{-1}\right).
\]
Substituting this definition into $I_1$:
\begin{align*}
I_{1} &= \int_{\Omega_U}\Bigg[\bigl(h_{ab}g^{I_s J_s}\sqrt{\det(g)}\circ\phi^{-1}\bigr)\partial_{I_s}(F^{a}\circ\phi^{-1}) \\
&\qquad\qquad + \sum_{|\gamma|\leq s-1}(B_{\gamma})_{ ab}^{J_s}\partial^{\gamma}(F^{a}\circ \phi^{-1})\Bigg](\Psi^{b}_{J_s}\circ\phi^{-1})\,d\lambda_{n}.
\end{align*}

For $I_{2}$, integrating by parts $|\beta|$ times yields
\begin{align*}
I_{2} &= \sum_{c=1}^{r}\sum_{|\beta|\le s-1}(-1)^{s+|\beta|}\int_{\Omega_U}\partial^{\beta}\Bigl[\Bigl(\bigl(h_{ab}(D_{\beta})^{b\, J_s}_{c}\sqrt{\det(g)}\bigr)\circ\phi^{-1} \Bigr)\\
&\qquad\qquad\qquad\qquad\qquad \times (F^a \circ \phi^{-1})\Bigr](\Psi^{c}_{J_s}\circ \phi^{-1})\,d\lambda_{n}.
\end{align*}
Applying Leibniz rule again to expand the derivative $\partial^{\beta}$ of the product:
\begin{align*}
I_{2} &= \int_{\Omega_U}\sum_{c=1}^{r}\sum_{|\beta|\le s-1}(-1)^{s+|\beta|}\sum_{\gamma \le \beta}\binom{\beta}{\gamma}\partial^{\beta-\gamma}\Bigl(\bigl(h_{ab}(D_{\beta})^{b\, J_s}_{c}\sqrt{\det(g)}\bigr)\circ\phi^{-1}\Bigr) \\
&\qquad\qquad\qquad \times \partial^{\gamma}(F^{a}\circ \phi^{-1})(\Psi^{c}_{J_s}\circ \phi^{-1})\,d\lambda_{n}.
\end{align*}
Define the coefficient functions $(C_{\gamma})^{J_s}_{ac}$ on $\Omega_U$ for each $|\gamma| \le s-1$ as
\[
(C_{\gamma})^{J_s}_{ac} := \sum_{b=1}^{r}\sum_{\substack{|\beta|\le s-1 \\ \gamma \le \beta}} (-1)^{s+|\beta|} \binom{\beta}{\gamma} \partial^{\beta-\gamma}\left( \left(h_{ab}(D_{\beta})^{b\, J_s}_{c}\sqrt{\det(g)}\right)\circ\phi^{-1} \right).
\]
Substituting this definition into $I_2$:
\[
I_{2} = \int_{\Omega_U}\sum_{|\gamma|\leq s-1}(C_{\gamma})^{J_s}_{ac}\partial^{\gamma}(F^{a}\circ \phi^{-1})(\Psi^{c}_{J_s}\circ \phi^{-1})\,d\lambda_{n}.
\]

Combining both results, and renaming the fiber index $b\to c$ in $I_1$ (using $h_{ac}$):
\begin{align*}
\int_{U}\langle F,(\nabla^{s})^{*}\Psi\rangle_{h_E}\,d\lambda _{g} &= \int_{\Omega_U}\Bigg[\left(h_{ac}g^{I_s J_s}\sqrt{\det(g)}\circ\phi^{-1}\right)\partial_{I_s}(F^{a}\circ\phi^{-1}) \\
&\quad + \sum_{|\gamma|\leq s-1}\left((B_{\gamma})_{ ac}^{J_s} + (C_{\gamma})^{J_s}_{ac}\right)\partial^{\gamma}(F^{a}\circ \phi^{-1})\Bigg](\Psi^{c}_{J_s}\circ\phi^{-1})\,d\lambda_{n}.
\end{align*}

We now define the correction coefficients $(\widehat{D}_{\gamma})$ on $U$ by pulling back $B$ and $C$ and inverting the metric weights:
\[
\big(\widehat{D}_{\gamma}\big)^{a}_{b\,I_s}
:=\frac{
g_{I_s \tilde{J}_s}h^{a\tilde{c}}
}{\sqrt{\det(g)}}\Big[\big((B_{\gamma})^{\tilde{J}_s}_{b\tilde{c}}\circ\phi\big)+
\big((C_{\gamma})^{\tilde{J}_s}_{b\tilde{c}}
\circ\phi\big)
\Big].
\]
Let $G\in L^{p}(T^{(0,s)}(TM)\otimes E|_{U})$ be given in coordinates by:
\begin{equation}\label{eq:weak_derivative_explicit}
G^{a}_{I_s}
=
\partial_{I_s} F^{a}+
\sum_{|\gamma|\le s-1}
\big(\widehat{D}_{\gamma}\big)^{a}_{b\,I_s}
\partial^{\gamma}F^{b}.
\end{equation}
We verify that $G$ is the required weak derivative:
\begin{align*}
\int_{U}\langle G,\Psi\rangle_{g,h_{E}}\,d\lambda_{g} &= \int_{U}h_{ac}g^{I_s J_s}G_{I_s}^{a}\Psi_{J_s}^{c}\,d\lambda_{g} \\
&= \int_{U}h_{ac}g^{I_s J_s}\left(\partial_{I_s}F^{a}+\sum_{|\gamma|\leq s-1}(\widehat{D}_{\gamma})^{a}_{b\, I_s}\partial^{\gamma}F^{b}\right)\Psi^{c}_{J_s}\,d\lambda_{g}.
\end{align*}
The first term matches the principal part. For the lower-order terms, substituting $\widehat{D}_{\gamma}$:
\begin{align*}
&\int_{U}\sum_{|\gamma|\leq s-1} h_{ac}g^{I_s J_s}\left( \frac{g_{I_s \tilde{J}_s}h^{a\tilde{c}}}{\sqrt{\det(g)}}\Big[\big((B_{\gamma})^{\tilde{J}_s}_{b\tilde{c}}\circ\phi\big)+ \big((C_{\gamma})^{\tilde{J}_s}_{b\tilde{c}}\circ\phi\big)\Big] \right) \partial^{\gamma}F^{b}\Psi^{c}_{J_s} \,d\lambda_{g} \\
&= \int_{U}\sum_{|\gamma|\leq s-1} \frac{1}{\sqrt{\det(g)}} (h_{ac}h^{a\tilde{c}}) (g^{I_s J_s}g_{I_s \tilde{J}_s}) \Big[\big((B_{\gamma})^{\tilde{J}_s}_{b\tilde{c}}\circ\phi\big)+ \big((C_{\gamma})^{\tilde{J}_s}_{b\tilde{c}}\circ\phi\big)\Big] \partial^{\gamma}F^{b}\Psi^{c}_{J_s} \,d\lambda_{g}.
\end{align*}
Using the identities $h_{ac}h^{a\tilde{c}} = \delta^{\tilde{c}}_{c}$ and $g^{I_s J_s}g_{I_s \tilde{J}_s} = \delta^{J_s}_{\tilde{J}_s}$, we obtain
\begin{align*}
&= \int_{U}\sum_{|\gamma|\leq s-1} \frac{1}{\sqrt{\det(g)}} \Big[\big((B_{\gamma})^{J_s}_{bc}\circ\phi\big)+ \big((C_{\gamma})^{J_s}_{bc}\circ\phi\big)\Big] \partial^{\gamma}F^{b}\Psi^{c}_{J_s}\,d\lambda_{g} \\
&= \int_{\Omega_U}\sum_{|\gamma|\leq s-1} \left[ (B_{\gamma})^{J_s}_{bc} + (C_{\gamma})^{J_s}_{bc} \right] \partial^{\gamma}(F^{b}\circ \phi^{-1})(\Psi^{c}_{J_s}\circ\phi^{-1})\,d\lambda_{n}.
\end{align*}
This expression matches the expansion of $\displaystyle\int_{U}\langle F,(\nabla^{s})^{*}\Psi\rangle_{h_E}\,d\lambda _{g}$ exactly. Therefore, $G$ is the $s$-th weak covariant derivative of $F$ in $U$. Moreover, $G\in L^{p}(T^{(0,s)}(TU)\otimes E|_{U})$ because each coordinate component is a linear combination of weak derivatives of the components of $F$ (up to order $s$) multiplied by smooth, bounded functions. Since $s\le m$ was arbitrary, $F|_{U}\in W^{m,p}\big(E|_{U}\big)$.

Finally, we establish the norm estimates. The Sobolev norm is defined as
\[
\|F\|_{W^{m,p}(E|_{U})} := \left(\sum_{s=0}^{m} \int_{U} |\nabla^s_w F|_{g,h_E}^p \, d\lambda_g\right)^{\frac{1}{p}}
\]
From \eqref{eq:weak_derivative_explicit}, for any $s \le m$ and ordered block $I_s=(i_1,\dots,i_s)$:
\[
(\nabla^{s}_{w}F)^{a}_{I_s}
=
\partial_{I_s}F^{a}
+
\sum_{|\gamma|\le s-1}
\big(\widehat{D}_{\gamma}\big)^{a}_{b\,I_s}
\,\partial^{\gamma}F^{b}.
\]
Define a fiber metric $\overline{h}_{E}$ on $E|_{U}$ by $\overline{h}_{E}(u,v)=\displaystyle\sum_{a=1}^{r}u^{a}v^{a}$ for sections $u=\displaystyle\sum _{a=1}^{r}u^{a}e_{a}, v=\displaystyle\sum_{a=1}^{r} v^{a}e_{a}$. Also consider the Riemannian metric $\overline g:=\phi^*g_{\mathrm{euc}}$ on $\overline{U}$. Since $\overline{U}$ is compact, Lemma \ref{lem:metric_equivalence} provides constants $c_0, C_0 > 0$ for the pointwise equivalence of the fiber norms, uniform for all $s \in \{0, \dots, m\}$. Here the auxiliary norm uses the same covariant derivatives, changing only the fiber norms and the measure. By integrating these estimates and observing that the Riemannian-Lebesgue measures $d\lambda_g$ and $d\lambda_{\overline{g}}$ are comparable on compact sets through positive constants \cite[Lemma 2]{VelazquezSandoval2026}, we conclude that there exist constants $c, C > 0$ such that the following equivalence of Sobolev norms holds:
\begin{equation}\label{eq:norm_equivalence_local}
c\|F\|_{W^{m,p}(E|_{U}),\overline{g},\overline{h}_{E}} \leq \|F\|_{W^{m,p}(E|_{U})} \leq C\|F\|_{W^{m,p}(E|_{U}),\overline{g},\overline{h}_{E}}.
\end{equation}
We express the auxiliary norm as:

\[\|F\|_{W^{m,p}(E|_{U}),\overline{g},\overline{h}_{E}}= \Bigg(\sum_{s=0}^{m}\int_{U}\Bigg(\sum_{a=1}^{r}\sum_{I_s}\bigg(\partial_{I_s} F^{a}+ \sum_{|\gamma|\le s-1} \big(\widehat{D}_{\gamma}\big)^{a}_{b\,I_s} \partial^{\gamma}F^{b}\bigg)^{2}\Bigg)^{\frac{p}{2}}d\lambda_{\overline{g}}\Bigg)^{\frac{1}{p}}\] \[= \Bigg(\sum_{s=0}^{m}\int_{\Omega_U}\Bigg(\sum_{a=1}^{r}\sum_{I_s}\bigg(\partial_{I_s} (F^{a}\circ\phi^{-1}) + \sum_{|\gamma|\le s-1} \left(\big(\widehat{D}_{\gamma}\big)^{a}_{b\,I_s}\circ \phi^{-1}\right) \partial^{\gamma}(F^{b}\circ\phi^{-1})\bigg)^{2}\Bigg)^{\frac{p}{2}}d\lambda_{n}\Bigg)^{\frac{1}{p}}.\]

To bound this from above, set
\[
 K:=
 \max_{\substack{
 1\leq s\leq m,\;1\leq a,b\leq r\\
 I_s\in\{1,\dots,n\}^{s},\;|\gamma|\leq s-1}}
 \left\|
 \big(\widehat D_{\gamma}\big)^a_{b\,I_s}
 \right\|_{L^\infty(U)}.
\]
Only finitely many terms enter this maximum, and each coefficient
extends smoothly to a neighborhood of $\overline U$; hence $K<\infty$.

Using $(A+B)^{2}\leq 2A^{2}+2B^{2}$, we obtain:
\begin{align*}
\|F\|_{W^{m,p}(E|_{U}),\overline{g},\overline{h}_{E}} &\leq \sqrt{2}\Bigg(\sum_{s=0}^{m}\int_{\Omega_U}\Bigg(\sum_{a=1}^{r}\sum_{|\alpha|=s}\frac{s!}{\alpha!}\left(\partial^{\alpha} (F^{a}\circ\phi^{-1})\right)^{2} \\
&\quad +rn^{s}\left( \sum_{b=1}^{r}\sum_{|\gamma|\le s-1} K \left|\partial^{\gamma}(F^{b}\circ\phi^{-1})\right|\right)^{2}\Bigg)^{\frac{p}{2}}d\lambda_{n}\Bigg)^{\frac{1}{p}}.
\end{align*}
Using $\|\cdot\|_{2} \leq \|\cdot\|_{1}$ to introduce square roots on coefficients, and then applying $(\displaystyle\sum |x_i|)^p \le N^{p-1}\displaystyle\sum |x_i|^p$, all multiplicative constants absorb into a constant $C'>0$:
\begin{align*}
&\leq C' \left(\sum_{s=0}^{m}\int_{\Omega_U}\sum_{a=1}^{r}\left(\sum_{|\alpha|=s}\left|\partial^{\alpha} (F^{a}\circ\phi^{-1})\right|^{p}+ \sum_{|\gamma|\le s-1} \left| \partial^{\gamma}(F^{a}\circ\phi^{-1})\right|^{p}\right)d\lambda_{n}\right)^{\frac{1}{p}} \\
&\leq C'' \left(\sum_{a=1}^{r}\sum_{|\beta|\le m}\int_{\Omega_U} \left|\partial^{\beta}(F^{a}\circ\phi^{-1})\right|^{p}d\lambda_{n}\right)^{\frac{1}{p}}
= C'' \left(\sum_{a=1}^{r} \left\|F^{a}\circ \phi^{-1}\right\|_{W^{m,p}(\Omega_U)}^{p}\right)^{\frac{1}{p}}.
\end{align*}
Using $\|\cdot\|_{p} \leq \|\cdot\|_{1}$, we obtain the upper bound:
\[
\|F\|_{W^{m,p}(E|_{U})} \leq C_{m}\sum_{a=1}^{r} \left\|F^{a}\circ \phi^{-1}\right\|_{W^{m,p}(\Omega_U)}.
\]

For the lower bound, from \eqref{eq:weak_derivative_explicit}, we isolate the partial derivative:
\begin{equation}\label{eq:partial_isolated}
\partial_{I_s} (F^{a}\circ\phi^{-1})
=
(\nabla^{s}_{w}F)^{a}_{I_s} \circ \phi^{-1}
-
\sum_{|\gamma|\le s-1}
\left(\big(\widehat{D}_{\gamma}\big)^{a}_{b\,I_s}\circ \phi^{-1}\right)
\partial^{\gamma}(F^{b}\circ\phi^{-1}).
\end{equation}
The $L^p$ norm of the first term is bounded by the bundle Sobolev norm. Taking $L^p(\Omega_U)$ norms and summing over components and orders:
\begin{equation}\label{eq:auxiliary_inequality}
\sum_{a=1}^{r} \|F^{a}\circ \phi^{-1}\|_{W^{m,p}(\Omega_U)}
\leq
\widetilde{C} \left( \|F\|_{W^{m,p}(E|_{U}),\overline{g},\overline{h}_{E}} + \sum_{a=1}^{r} \|F^{a}\circ \phi^{-1}\|_{W^{m-1,p}(\Omega_U)} \right).
\end{equation}
We proceed by induction on $m$. For $m=1$, the $L^p$ norms are clearly bounded. Assume the bound holds for $m-1$. Then
\[
\sum_{a=1}^{r} \|F^{a}\circ \phi^{-1}\|_{W^{m-1,p}(\Omega_U)} \leq K_{m-1} \|F\|_{W^{m-1,p}(E|_{U}),\overline{g},\overline{h}_{E}}.
\]
Substituting this into \eqref{eq:auxiliary_inequality}:
\[
\sum_{a=1}^{r} \|F^{a}\circ \phi^{-1}\|_{W^{m,p}(\Omega_U)}
\leq
\widetilde{C} \left( 1 + K_{m-1} \right) \|F\|_{W^{m,p}(E|_{U}),\overline{g},\overline{h}_{E}}.
\]
Combining this with \eqref{eq:norm_equivalence_local}, we obtain the lower bound constant $c_m > 0$.
\end{proof}

Since the local characterization in Lemma \ref{lem:meyers-serrin-local} relies on the Sobolev spaces defined on the restricted bundle $E|_U$, we first ensure that the restriction of a global Sobolev section is well-defined and consistent with the local theory:
\begin{proposition}[Restriction to open subsets]\label{prop:restriction_sobolev}
Let $(M,g)$ be a Riemannian manifold and $E \to M$ a smooth vector bundle equipped with a fiber metric $h_E$. If $U \subseteq M$ is an open subset and $u \in W^{m,p}(E)$, then $u|_U \in W^{m,p}(E|_U)$.
\end{proposition}

\begin{proof}
Let $s \leq m$. To show that $u|_U$ possesses a weak covariant derivative of order $s$ in $L^p(T^{(0,s)}(TU) \otimes E|_U)$, let $\sigma \in \Gamma_c(\operatorname{int}(U);T^{(0,s)}(TU)\otimes E|_U)$ be an arbitrary test section. We extend $\sigma$ to a section $\widetilde{\sigma} \in \Gamma_c(\operatorname{int}(M);T^{(0,s)}(TM)\otimes E)$ by setting it to zero outside $U$. 

Since $u \in W^{m,p}(E)$, it possesses a global $s$-th weak covariant derivative $\nabla_w^s u$ satisfying:
\[
\int_{M} \langle \nabla_{w}^{s}u, \widetilde{\sigma} \rangle_{g,h_{E}} d\lambda_{g} = \int_{M} \langle u, (\nabla^{s})^{*} \widetilde{\sigma} \rangle_{h_{E}} d\lambda_{g}.
\]
Because $(\nabla^s)^*$ is a local operator, the support of $(\nabla^s)^* \widetilde{\sigma}$ is contained in $\operatorname{supp}(\widetilde{\sigma}) = \operatorname{supp}(\sigma) \subseteq U$. Furthermore, on the open set $U$, we have $(\nabla^s)^* \widetilde{\sigma} \equiv (\nabla^s)^* \sigma$. It follows that:

\[\int_{U} \langle (\nabla_{w}^{s}u)|_{U}, \sigma \rangle_{g,h_{E}} d\lambda_{g} = \int_{M} \langle \nabla_{w}^{s}u, \widetilde{\sigma} \rangle_{g,h_{E}} d\lambda_{g}\] \[ 
= \int_{M} \langle u, (\nabla^{s})^{*} \widetilde{\sigma} \rangle_{h_{E}} d\lambda_{g}= \int_{U} \langle u|_{U}, (\nabla^{s})^{*} \sigma \rangle_{h_{E}} d\lambda_{g}.\]

This identity identifies $(\nabla_{w}^{s}u)|_{U}$ as the $s$-th weak covariant derivative of $u|_U$. Since $\nabla_w^s u \in L^p$, its restriction remains in $L^p(T^{(0,s)}(TU) \otimes E|_U)$. As $s \leq m$ was arbitrary, we conclude that $u|_U \in W^{m,p}(E|_U)$.
\end{proof}
Another fundamental property is that the weak covariant derivative is a local operator:
\begin{proposition}[Locality of the Weak Covariant Derivative]\label{prop:weak_derivative_locality}
Let $(M,g)$ be a Riemannian manifold and let $E\to M$ be a smooth vector bundle with fiber metric $h_{E}$ and a connection $\nabla^{E}$. Let $u\in L^{1}_{\operatorname{loc}}(E)$ be a section possessing a weak $s$-th covariant derivative $\nabla^{s}_{w}u$. Let $U\subseteq M$ be an open set. If $u=0$ a.e. in $U$, then $\nabla^{s}_{w}u=0$ a.e. in $U$.
\end{proposition}

\begin{proof}
Let $\sigma \in \Gamma_c(\operatorname{int}(U);T^{(0,s)}(TU)\otimes E|_U)$ be an arbitrary test section with support in $U$. We extend $\sigma$ by zero to the whole manifold $M$, denoting the extension by $\widetilde{\sigma}$. Since $\operatorname{supp}(\sigma) \subset U$, $\widetilde{\sigma}$ is smooth and compactly supported on $M$. By the definition of the weak derivative,
\[
\int_{M}\langle \nabla_{w}^{s}u,\widetilde{\sigma}\rangle _{g,h_{E}}d\lambda_{g}
= \int_{M}\langle u,(\nabla^{s})^{*}\widetilde{\sigma}\rangle_{h_{E}}d\lambda_{g}.
\]
Since $(\nabla^{s})^{*}$ is a local operator (being a differential operator with smooth coefficients), $\operatorname{supp}((\nabla^{s})^{*}\widetilde{\sigma}) \subseteq \operatorname{supp}(\widetilde{\sigma}) \subset U$. Since $u$ vanishes almost everywhere on $U$, the right-hand side is zero. Thus,
\[
\int_{U}\langle \nabla_{w}^{s}u, \sigma \rangle _{g,h_{E}}d\lambda_{g} = 0
\]
for all $\sigma \in \Gamma_c(\operatorname{int}(U);T^{(0,s)}(TU)\otimes E|_U)$. By testing one coordinate component at a time and using the fundamental lemma for locally integrable functions, we conclude that $\nabla_{w}^{s}u = 0$ almost everywhere in $U$.
\end{proof}

The following observations will be used in the partition-of-unity arguments of Sections~\ref{sec:global_approximation} and~\ref{sec:embeddings_rellich}.

\begin{remark}[Multiplication by smooth functions]\label{rem:smooth-multipliers}
Let $m\geq0$ and $1\leq p<\infty$. If $\chi\in C^\infty(M)$ has bounded covariant derivatives through order $m$, including $\chi$ itself, then multiplication by $\chi$ is bounded on $W^{m,p}(E)$:
\begin{equation}\label{eq:weak-cutoff-bound}
 \|\chi u\|_{W^{m,p}(E)}
 \leq C_m\left(\sum_{t=0}^m\|\nabla^t\chi\|_{L^\infty}\right)
       \|u\|_{W^{m,p}(E)}.
\end{equation}
By Lemma~\ref{lem:meyers-serrin-local} and the Euclidean weak Leibniz rule, the iterated covariant product rule holds for weak derivatives as well. Taking pointwise norms, and using that permutations of tensor factors preserve the induced norm, gives
\[
 |\nabla_w^s(\chi u)|_{g,h_E}
 \leq\sum_{t=0}^s\binom{s}{t}|\nabla^{s-t}\chi|_g
       |\nabla_w^t u|_{g,h_E},\qquad 0\leq s\leq m.
\]
Taking $L^p$ norms and summing over $s$ proves \eqref{eq:weak-cutoff-bound}. In particular, every $\chi\in C_c^\infty(M)$ satisfies its hypotheses.

\end{remark}

\begin{remark}[Extension by zero]\label{rem:zero-extension}
Let $m\geq0$, $1\leq p<\infty$, and let $U\subseteq M$ be open. Suppose that $u\in W^{m,p}(E|_U)$ vanishes almost everywhere in $U\setminus K$ for a compact set $K\subset U$, which may meet $\partial M$. Denote by $\widetilde u$ its extension by zero from $U$ to $M$. Then $\widetilde u\in W^{m,p}(E)$ and
\[
 \nabla_w^s\widetilde u=\widetilde{\nabla_w^s u},\qquad 0\leq s\leq m,
 \qquad \|\widetilde u\|_{W^{m,p}(E)}=\|u\|_{W^{m,p}(E|_U)},
\]
where $\widetilde{\nabla_w^s u}$ denotes extension by zero of the local weak derivative.

To verify this, choose $\chi\in C_c^\infty(U)$ equal to one on a neighborhood of $K$. Fix $1\leq s\leq m$ and $\Psi\in\Gamma_c(\operatorname{int}(M);T^{(0,s)}(TM)\otimes E)$. The section $\chi\Psi|_U$ is an admissible local test section, since
\[
 \operatorname{supp}(\chi\Psi|_U)\subseteq\operatorname{supp}\chi\cap\operatorname{supp}\Psi\Subset U\cap\operatorname{int}(M).
\]
By Proposition~\ref{prop:weak_derivative_locality}, $\nabla_w^s u$ vanishes almost everywhere in $U\setminus K$. Consequently,
\[
 \int_M\langle\widetilde{\nabla_w^s u},\Psi\rangle_{g,h_E}\,d\lambda_g
 =\int_U\langle\nabla_w^s u,\chi\Psi|_U\rangle_{g,h_E}\,d\lambda_g
\]
\[
 =\int_U\langle u,(\nabla^s)^*(\chi\Psi|_U)\rangle_{h_E}\,d\lambda_g
 =\int_M\langle\widetilde u,(\nabla^s)^*\Psi\rangle_{h_E}\,d\lambda_g.
\]
For the last equality, $\chi=1$ on a neighborhood of $K$, so $(\nabla^s)^*(\chi\Psi|_U)$ and $(\nabla^s)^*(\Psi|_U)$ agree there; outside $K$, the section $u$ vanishes almost everywhere. This proves the weak derivative identity. Extension by zero preserves each $L^p$ norm, so it also preserves the Sobolev norm. The extension takes place within $M$, across the artificial edge of $U$, and imposes no additional vanishing condition along $U\cap\partial M$.
\end{remark}

\section{Global Approximation Results}\label{sec:global_approximation}

The following theorem recovers the covariant Meyers--Serrin theorem. For manifolds without boundary it is contained, with more general connection hypotheses, in \cite[Corollary~3.6]{Guidetti2017}; see also \cite{ChanCzubak2024Survey} for the scalar and differential-form settings. We retain the same partition-of-unity argument for smooth boundary, using the interior definition of weak derivatives and smooth approximation up to the flat boundary in coordinate half-balls. No completeness or uniform geometry is needed, because the local approximation errors can be chosen independently and made summable.

\begin{theorem}[Meyers--Serrin for Vector Bundles]\label{thm:meyers-serrin-bundle}
Let $(M,g)$ be a Riemannian manifold with or without smooth boundary, and let $\pi_E:E\to M$ be a smooth vector bundle equipped with a fiber metric $h_E$ and a compatible connection $\nabla^E$. For every integer $m\geq0$ and $1\leq p<\infty$, we have
\[
H^{m,p}(E)=W^{m,p}(E).
\]
\end{theorem}

\begin{proof}
The inclusion $H^{m,p}(E)\subseteq W^{m,p}(E)$ was established after Definition~\ref{def:weak_deriv}. To prove density, choose a countable, locally finite cover by regular coordinate balls and boundary coordinate half-balls $(U_\alpha,\phi_\alpha)_{\alpha\in\mathbb N}$. Paracompactness and second countability provide such a cover, chosen so that the frames extend beyond $\overline{U_\alpha}$. By \cite[Theorem~2.23, pp.~43--44]{LeeS}, let $(\psi_\alpha)$ be a subordinate smooth partition of unity with $\operatorname{supp}\psi_\alpha\Subset U_\alpha$. Put $\Omega_\alpha=\phi_\alpha(U_\alpha\cap\operatorname{int}(M))$.

Consider a section $F\in W^{m,p}(E)$. We can decompose it as
\[
F=\sum_{\alpha=1}^{\infty}\psi_{\alpha}F,
\]
where the sum is locally finite. For each $\alpha$, by Proposition~\ref{prop:restriction_sobolev} and Remark~\ref{rem:smooth-multipliers}, $(\psi_{\alpha}F)|_{U_{\alpha}}$ belongs to $W^{m,p}(E|_{U_{\alpha}})$, since $\psi_{\alpha}$ is smooth and has compact support contained in $U_{\alpha}$.

Fix $\alpha$ and work in the chart $(U_{\alpha},\phi_{\alpha})$, equipped with a local frame $(e_{1},\dots,e_{r})$ of $E$ over $U_{\alpha}$. Applying Lemma \ref{lem:meyers-serrin-local} to the section $(\psi_{\alpha}F)|_{U_{\alpha}}$, we deduce that for every fiber index $a \in \{1,\dots,r\}$,
\[
\big(\psi_{\alpha}F\big)^{a}\circ \phi_{\alpha}^{-1}\in W^{m,p}(\Omega_\alpha).
\]
Moreover, there exist constants $C_{m,\alpha}>0$ such that, for any section $H\in W^{m,p}(E|_{U_{\alpha}})$,
\[
\|H\|_{W^{m,p}(E|_{U_{\alpha}})}\leq
C_{m,\alpha}\sum_{a=1}^{r}
\big\|H^{a}\circ \phi_{\alpha}^{-1}\big\|_{W^{m,p}(\Omega_\alpha)}.
\]

The component functions
\[
 f_\alpha^a=(\psi_\alpha F)^a\circ\phi_\alpha^{-1},
 \qquad K_\alpha:=\phi_\alpha(\operatorname{supp}\psi_\alpha),
\]
vanish almost everywhere outside the fixed compact set $K_\alpha$. In an interior chart, $K_\alpha\Subset\Omega_\alpha$,  the classical Meyers--Serrin theorem~\cite{MeyersSerrin1964} gives smooth functions $h_{\alpha,k}^a$ converging to $f_\alpha^a$ in $W^{m,p}(\Omega_\alpha)$. With a fixed $\eta_\alpha\in C_c^\infty(\Omega_\alpha)$ equal to one near $K_\alpha$, the functions $\widetilde G_{\alpha,k}^a=\eta_\alpha h_{\alpha,k}^a$ belong to $C_c^\infty(\Omega_\alpha)$ and still converge to $f_\alpha^a$, by the scalar multiplier estimate in Remark~\ref{rem:smooth-multipliers} and the identity $\eta_\alpha f_\alpha^a=f_\alpha^a$.

In a boundary chart, write $\Omega_\alpha=B(0,R_\alpha)\cap\mathbb R^n_+$. Now $K_\alpha$ is a compact subset of $B(0,R_\alpha)\cap\{x_n\geq0\}$: it may meet the flat boundary, but its distance from $\partial B(0,R_\alpha)\cap\{x_n\geq0\}$ is positive. For every $a\in\{1,\dots,r\}$, define
\[
 \widehat f_\alpha^a(x):=
 \begin{cases}
 f_\alpha^a(x),&x\in\Omega_\alpha,\\
 0,&x\in\mathbb R^n_+\setminus\Omega_\alpha.
 \end{cases}
\]
Lemma~\ref{lem:halfball-zero} gives $\widehat f_\alpha^a\in W^{m,p}(\mathbb R^n_+)$, with the same Sobolev norm as $f_\alpha^a$. This zero extension is made only inside the open half-space; no value is assigned for $x_n<0$.

By Lemma~\ref{lem:halfspace-tools}, choose $h_{\alpha,k}^a\in C_c^\infty(\mathbb R^n)$ such that $h_{\alpha,k}^a|_{\mathbb R^n_+}\to\widehat f_\alpha^a$ in $W^{m,p}(\mathbb R^n_+)$. Fix $\eta_\alpha\in C_c^\infty(B(0,R_\alpha))$ equal to one near $K_\alpha$, and define
\[
 \widetilde G_{\alpha,k}^a:=(\eta_\alpha h_{\alpha,k}^a)|_{\Omega_\alpha}.
\]
The same cutoff is used for every $k$ and every component $a$. Because $\eta_\alpha\widehat f_\alpha^a=\widehat f_\alpha^a$ almost everywhere in $\mathbb R^n_+$, the error is
\[
 \widetilde G_{\alpha,k}^a-f_\alpha^a
 =\left[\eta_\alpha\bigl(h_{\alpha,k}^a|_{\mathbb R^n_+}-\widehat f_\alpha^a\bigr)\right]\big|_{\Omega_\alpha}.
\]
The scalar multiplier estimate in Remark~\ref{rem:smooth-multipliers}, applied on the open half-space, therefore yields
\[
 \|\widetilde G_{\alpha,k}^a-f_\alpha^a\|_{W^{m,p}(\Omega_\alpha)}
 \leq C_{m,\eta_\alpha}
       \|h_{\alpha,k}^a|_{\mathbb R^n_+}-\widehat f_\alpha^a\|_{W^{m,p}(\mathbb R^n_+)}
 \longrightarrow0.
\]
Each product $\eta_\alpha h_{\alpha,k}^a$ is smooth on all of $\mathbb R^n$. Its restriction thus defines a smooth function on $\phi_\alpha(U_\alpha)$, including its flat boundary, and its support there lies in the fixed compact set $\operatorname{supp}\eta_\alpha\cap\{x_n\geq0\}$. It vanishes near the spherical edge, but its values on the flat boundary need not vanish.

In either case, the approximants have support in a fixed compact subset of $\phi_\alpha(U_\alpha)$, relative to the closed half-space in a boundary chart. Since there are finitely many fiber components, they satisfy
\[
 \sum_{a=1}^r\|\widetilde G_{\alpha,k}^a-f_\alpha^a\|_{W^{m,p}(\Omega_\alpha)}
 \longrightarrow0.
\]

Using the smooth representatives on $\phi_\alpha(U_\alpha)$ just constructed, we define a smooth section $G_{\alpha,k}\in \Gamma(E|_{U_{\alpha}})$ by
\[
G_{\alpha,k}
=
\sum_{a=1}^{r}
\Big[(\widetilde{G}_{\alpha,k})^{a}\circ \phi_{\alpha}\Big]\; e_a.
\]
Then, by the norm estimate from Lemma \ref{lem:meyers-serrin-local}, it follows that
\[
\|G_{\alpha,k}-(\psi_{\alpha}F)|_{U_{\alpha}}\|_{W^{m,p}(E|_{U_{\alpha}})}\xrightarrow[k\to\infty]{}0.
\]

Extending each $G_{\alpha,k}$ by zero outside $U_{\alpha}$, we may view it as a global section with compact support, i.e., $G_{\alpha,k}\in \Gamma_{c}(E)$. Furthermore, the norm on the restriction is given by
\[
\|G_{\alpha,k}-(\psi_{\alpha}F)|_{U_{\alpha}}\|_{W^{m,p}(E|_{U_{\alpha}})}
= \left(\sum_{s=0}^{m}\int_{U_{\alpha}}|\nabla_{w}^{s}(G_{\alpha,k}-(\psi_{\alpha} F)|_{U_{\alpha}})|^{p}_{g,h_{E}}d\lambda_{g}\right)^{\frac{1}{p}}.
\]
The local error has compact support in $U_\alpha$, so Remark~\ref{rem:zero-extension} shows that its weak derivatives extend by zero. Therefore, the integral over $U_\alpha$ coincides with the global integral:
\[
\left(\sum_{s=0}^{m}\int_{M}|\nabla_{w}^{s}(G_{\alpha,k}-\psi_{\alpha} F)|^{p}_{g,h_{E}}d\lambda_{g}\right)^{\frac{1}{p}}.
\]
Consequently,
\[
\|G_{\alpha,k}-\psi_{\alpha}F\|_{W^{m,p}(E)}\xrightarrow[k\to\infty]{}0.
\]

Now, fix $\alpha\in \mathbb{N}$ and an integer $q\in \mathbb{N}$. By the convergence established above, there exists an index $k_{\alpha,q}\in \mathbb{N}$ such that
\[
\|\psi_{\alpha}F-G_{\alpha,k_{\alpha,q}}\|_{W^{m,p}(E|_{U_{\alpha}})}<\frac{1}{2^{\alpha+q}}.
\]

Since $\operatorname{supp}(G_{\alpha,k})\subseteq U_{\alpha}$ for all $\alpha,k\in\mathbb{N}$ and the cover $(U_{\alpha})_{\alpha\in\mathbb{N}}$ is locally finite, the sum
\[
G_{q}:=\sum_{\alpha=1}^{\infty}G_{\alpha,k_{\alpha,q}}
\]
is locally finite and therefore defines a global smooth section $G_{q}\in \Gamma(E)$.

We now show that $G_q$ approximates $F$ in the Sobolev norm. The errors form an absolutely summable series in the complete space $W^{m,p}(E)$. Its limit agrees locally with $F-G_q$, since the cover is locally finite. Thus the triangle inequality gives:
\[
\|F-G_{q}\|_{W^{m,p}(E)}
\leq \sum_{\alpha=1}^{\infty}\|\psi_{\alpha}F-G_{\alpha,k_{\alpha,q}}\|_{W^{m,p}(E)}
\leq \sum_{\alpha=1}^{\infty}\frac{1}{2^{\alpha+q}}=\frac{1}{2^{q}}.
\]
In particular,
\[
\|G_q\|_{W^{m,p}(E)}
\leq \|F\|_{W^{m,p}(E)} + \|F-G_q\|_{W^{m,p}(E)}
\leq\|F\|_{W^{m,p}(E)} + 2^{-q}<\infty,
\]
which implies that $G_q\in W^{m,p}(E)\cap \Gamma(E)$ for all $q\in \mathbb{N}$.

Since $2^{-q}\to 0$ as $q\to\infty$, we conclude that $G_{q}\to F$ in $W^{m,p}(E)$. Thus, every section in $W^{m,p}(E)$ can be approximated in the Sobolev norm by smooth sections, proving that
\[
H^{m,p}(E)=W^{m,p}(E).
\]
\end{proof}
The invariance below is classical; compare \cite[Theorem~10.2.36(a)]{Nicolaescu2020} for bundles and \cite[Proposition~2.2]{Hebey1} for scalar functions. The proof below extends the scalar jet comparison of \cite[Theorem~2]{VelazquezSandoval2026} to sections of a vector bundle and uses smooth density to pass to weak Sobolev spaces.

\begin{theorem}[Geometric Invariance on Compact Manifolds]\label{thm:geometric_invariance}
Let $M$ be a compact smooth manifold with or without smooth boundary, and let $E\to M$ be a smooth vector bundle. For every integer $m\geq0$ and $1\leq p<\infty$, the spaces $W^{m,p}(E)$ obtained from different smooth Riemannian metrics, fiber metrics, and compatible connections coincide as sets of measurable sections and have equivalent norms.
\end{theorem}

\begin{proof}
We adapt the finite-dimensional jet argument of \cite[Theorem~2]{VelazquezSandoval2026}. Denote the two sets of geometric data by $(g,h_E,\nabla^E)$ and $(\widetilde g,\widetilde h_E,\widetilde\nabla^E)$. The induced covariant derivatives are denoted by $\nabla$ and $\widetilde\nabla$, and the Sobolev norms by $\|\cdot\|_{m,p}$ and $\|\cdot\|_{m,p}^{\sim}$.

Choose a finite cover by regular coordinate balls and boundary half-balls $U_\alpha$, with coordinates $\phi_\alpha$ and bundle frames extending to neighborhoods of $\overline U_\alpha$. For $u=\displaystyle\sum_a u_\alpha^a e_{a,\alpha}$, put $f_\alpha^a=u_\alpha^a\circ\phi_\alpha^{-1}$. Iterating Lemma~\ref{lem:connection-components} gives, for $1\leq j\leq m$,
\begin{equation}\label{eq:jet-triangular}
 (\nabla^j u)^a_{i_1\ldots i_j}
 =\partial_{i_1}\cdots\partial_{i_j}f_\alpha^a
  +\sum_{b=1}^r\sum_{|\beta|<j}
       B^{ab,\beta}_{\alpha,i_1\ldots i_j}\,
       \partial^\beta f_\alpha^b,
\end{equation}
where the coefficients are smooth and depend on the connection coefficients and their derivatives through order $j-1$. Indeed, differentiating the leading term gives the unique term of order $j+1$ at the next step; differentiating its coefficients and applying the connection terms produces only derivatives of order at most $j$. The same formula holds for $\widetilde\nabla$, with different coefficients.

Let $N=\binom{n+m}{m}$, and consider the bundle
\[
 \mathscr T_m(E):=\bigoplus_{j=0}^m T^{(0,j)}(TM)\otimes E.
\]
For a smooth section $u$, write $J_\nabla^m u=(u,\nabla u,\dots,\nabla^m u)$ and, at each $x\in M$, define the vector space of realizable covariant jets
\[
 \mathscr D_{\nabla,x}(E)
 :=\{J_\nabla^m u(x)\mid u\in\Gamma(E)\}
 \subseteq\mathscr T_m(E)_x.
\]
Fix a chart and frame as above. The component derivatives determine a map
\[
 \begin{aligned}
 F_{\alpha,\nabla,x}\colon\mathscr D_{\nabla,x}(E)
 &\longrightarrow\mathbb R^{rN},\\
 J_\nabla^m u(x)&\longmapsto
 \bigl(\partial^\beta f_\alpha^a(\phi_\alpha(x))\bigr)_{
       1\leq a\leq r,\ |\beta|\leq m}.
 \end{aligned}
\]
We verify that this map is a linear isomorphism. If two sections have the same covariant jet at $x$, their order-zero components agree. Formula~\eqref{eq:jet-triangular} then shows successively that their partial derivatives of orders $1,\dots,m$ agree. Thus the map is well defined. Conversely, if its value is zero, the same formula makes every covariant derivative through order $m$ zero at $x$. Its kernel is therefore trivial, proving injectivity.

For surjectivity, let $v=(v_\beta^a)\in\mathbb R^{rN}$ be arbitrary. In the larger chart on which the chosen frame is defined, choose a smooth cutoff $\chi$ with compact support, equal to one near $x$. The section
\[
 u(q)=\chi(q)\sum_{a=1}^r
       \left(\sum_{|\beta|\leq m}
       \frac{v_\beta^a}{\beta!}
       (\phi_\alpha(q)-\phi_\alpha(x))^\beta\right)e_{a,\alpha}(q),
\]
extended by zero outside the chart, is smooth and satisfies
\[
 \partial^\beta(u_\alpha^a\circ\phi_\alpha^{-1})(\phi_\alpha(x))
 =v_\beta^a,
 \qquad 1\leq a\leq r,\quad |\beta|\leq m.
\]
This also applies when $x\in\partial M$: the polynomial and cutoff are smooth up to the flat boundary, and the support stays inside the relative chart domain. Hence $F_{\alpha,\nabla,x}$ is surjective.

Formula~\eqref{eq:jet-triangular} expresses $F_{\alpha,\nabla,x}^{-1}v$ with coefficients smooth in $x$. In particular, the inverse images of the standard basis of $\mathbb R^{rN}$ form a smooth local frame. Thus $\mathscr D_\nabla(E):=\bigcup_x\mathscr D_{\nabla,x}(E)$ is a smooth subbundle of $\mathscr T_m(E)$ of rank $rN$. The same construction gives $\mathscr D_{\widetilde\nabla}(E)$ and $F_{\alpha,\widetilde\nabla,x}$ for the second set of connections.

For $x\in\overline U_\alpha$ and $v\in\mathbb R^{rN}$, define
\[
 \begin{aligned}
 Q_\alpha(x,v)
   &:=\sum_{j=0}^m
     |(F_{\alpha,\nabla,x}^{-1}v)_j|_{g,h_E}^2,\\
 \widetilde Q_\alpha(x,v)
   &:=\sum_{j=0}^m
     |(F_{\alpha,\widetilde\nabla,x}^{-1}v)_j|_{\widetilde g,\widetilde h_E}^2.
 \end{aligned}
\]
These are continuous in $x$ and positive definite quadratic forms in $v$, since the two maps are isomorphisms. On the compact set $\overline U_\alpha\times\mathbb S^{rN-1}$, where $\mathbb S^{rN-1}$ is the Euclidean unit sphere, the continuous positive function
\[
 (x,v)\longmapsto\frac{\widetilde Q_\alpha(x,v)}{Q_\alpha(x,v)}
\]
has a positive minimum $c_\alpha$ and a finite maximum $C_\alpha$. Homogeneity gives
\[
 c_\alpha Q_\alpha(x,v)\leq\widetilde Q_\alpha(x,v)
 \leq C_\alpha Q_\alpha(x,v),
 \qquad x\in\overline U_\alpha,\quad v\in\mathbb R^{rN}.
\]
For every smooth $u$, both covariant jets correspond to the same vector of partial derivatives:
\[
 \begin{aligned}
 F_{\alpha,\nabla,x}(J_\nabla^m u(x))
 &=\bigl(\partial^\beta f_\alpha^a(\phi_\alpha(x))\bigr)_{a,\beta}\\
 &=F_{\alpha,\widetilde\nabla,x}(J_{\widetilde\nabla}^m u(x)).
 \end{aligned}
\]
Applying the preceding comparison to this vector, and taking the minimum and maximum of the constants over the finite cover, yields
\begin{equation}\label{eq:jet-pointwise-comparison}
 c_m\sum_{j=0}^m|\nabla^j u|_{g,h_E}^2
 \leq\sum_{j=0}^m|\widetilde\nabla^j u|_{\widetilde g,\widetilde h_E}^2
 \leq C_m\sum_{j=0}^m|\nabla^j u|_{g,h_E}^2
 \quad\text{on }M.
\end{equation}

The $\ell^2$ and $\ell^p$ norms on $\mathbb R^{m+1}$ are equivalent. As in the proof of Lemma~\ref{lem:meyers-serrin-local}, the Riemannian--Lebesgue measures $d\lambda_g$ and $d\lambda_{\widetilde g}$ are comparable on the compact manifold $M$ through positive constants \cite[Lemma~2]{VelazquezSandoval2026}. Taking the $\frac{p}{2}$ power in \eqref{eq:jet-pointwise-comparison}, using these comparisons, and integrating gives
\begin{equation}\label{eq:jet-sobolev-comparison}
 c_{m,p}\|u\|_{m,p}
 \leq\|u\|_{m,p}^{\sim}
 \leq C_{m,p}\|u\|_{m,p},
 \qquad u\in\Gamma(E).
\end{equation}

Finally, let $u$ belong to the weak Sobolev space for the first data. By Theorem~\ref{thm:meyers-serrin-bundle}, choose smooth sections $u_\nu\to u$ in that space. Estimate~\eqref{eq:jet-sobolev-comparison} makes $(u_\nu)$ Cauchy in the space for the second data, so it converges there by completeness. The two $L^p$ norms are equivalent, hence this limit is the same measurable section $u$. Passing to the limit proves \eqref{eq:jet-sobolev-comparison} for $u$. Interchanging the data proves the opposite inclusion, and therefore equality of the spaces with equivalent norms.
\end{proof}

\begin{remark}[Uniform hypotheses on a noncompact base]\label{rem:noncompact-invariance}
On a noncompact base, the same conclusion holds if the two base and fiber metrics are uniformly equivalent and the differences $\widetilde\nabla^E-\nabla^E$ and $\widetilde\nabla^{TM}-\nabla^{TM}$ have uniformly bounded covariant derivatives through order $m-1$, measured with the first set of data. No connection bounds are needed for $m=0$. For fixed base and fiber metrics, the corresponding independence under a bounded perturbation of the connection is proved in \cite[Proposition~2.19(ii)]{KohrNistor2022}. Uniform local descriptions under bounded geometry are given in \cite[Theorems~14 and~44]{GrosseSchneider2013}.
\end{remark}

\section{Sobolev Embeddings and the Rellich--Kondrashov Theorem}\label{sec:embeddings_rellich}

The continuous and compact embeddings below are classical. Palais~\cite[Chapter~X, Section~4, Theorems~3--4]{Palais1965} treats the Hilbert scale using Fourier series on the torus and localization to bundles. The scalar $L^p$ theory on compact manifolds is developed in \cite[Theorems~2.20 and~2.34]{Aubin} and \cite[Theorems~2.6 and~2.9]{Hebey1}; corresponding bundle statements appear in \cite[Theorem~10.2.36(c)]{Nicolaescu2020} and \cite[Proposition~16.22]{BleeckerBooss2013}. The two-sided estimates of Lemma~\ref{lem:meyers-serrin-local} transfer the Euclidean estimates to sections in each chart, and a finite partition of unity gives the global embeddings. Regular boundary half-balls allow the same argument when $\partial M\neq\varnothing$.

\begin{theorem}[Sobolev Embedding for Vector Bundles]\label{thm:sobolev_embedding_bundle}
Let $(M,g)$ be a compact Riemannian manifold of dimension $n$ with or without smooth boundary, and let $E\longrightarrow M$ be a smooth vector bundle equipped with a fiber metric $h_{E}$ and a compatible connection $\nabla^{E}$. Then, for any integers $j\geq0$ and $m\geq1$, and real numbers $1\leq p,q<\infty$ satisfying
\[
 \frac{1}{q}\geq\frac{1}{p}-\frac{m}{n},
\]
we have
\[
W^{j+m,p}(E)\subseteq W^{j,q}(E),
\]
and this inclusion is continuous.
\end{theorem}

\begin{proof}
Fix $j,m,p,q$ as in the statement. Consider a finite open cover $(U_\alpha,\phi_\alpha)_{\alpha=1}^N$ of $M$ consisting of regular coordinate balls and boundary coordinate half-balls, each being the domain of a local frame $(e_{1,\alpha},\dots,e_{r,\alpha})$. Let $(\psi_\alpha)_{\alpha=1}^N$ be a subordinate smooth partition of unity with $\operatorname{supp}\psi_\alpha\Subset U_\alpha$, and set
\[
 \Omega_\alpha:=\phi_\alpha(U_\alpha\cap\operatorname{int}(M)),
 \qquad K_\alpha:=\phi_\alpha(\operatorname{supp}\psi_\alpha).
\]
For an interior chart, $\Omega_\alpha$ is a bounded smooth ball. Applying the scalar Sobolev embeddings \cite[Theorem~1.2.26 and Remark~1.2.27, pp.~39--40]{DrabekMilota2007} to the weak derivatives of orders at most $j$ gives $W^{j+m,p}(\Omega_\alpha)\hookrightarrow W^{j,q}(\Omega_\alpha)$ when $q\geq p$; for $q<p$, H\"older's inequality gives the same continuous inclusion. For a boundary chart, we apply Lemma~\ref{lem:boundary-local-embeddings} with $\ell=j+m$ and $k=j$ to functions vanishing outside $K_\alpha$. Thus, in either type of chart there exists $C_\alpha>0$ such that
\begin{equation}\label{eq:embedding_local}
 \|v\|_{W^{j,q}(\Omega_\alpha)}
 \leq C_\alpha\|v\|_{W^{j+m,p}(\Omega_\alpha)}
\end{equation}
for every $v\in W^{j+m,p}(\Omega_\alpha)$ vanishing almost everywhere outside $K_\alpha$. The constants may depend on the chart and $K_\alpha$, but not on $v$.

Let $F\in W^{j+m,p}(E)$ and assume that $F=F^{a}_{\alpha}e_{a,\alpha}$ in $U_{\alpha}$. By Remark~\ref{rem:smooth-multipliers} and Proposition~\ref{prop:restriction_sobolev}, $(\psi_{\alpha}F)|_{U_{\alpha}}\in W^{j+m,p}(E|_{U_{\alpha}})$. By Lemma \ref{lem:meyers-serrin-local}, we know that for all $a\in \{1,\dots,r\}$ and all $\alpha\in \{1,\dots,N\}$, the function $(\psi_{\alpha} F^{a}_{\alpha})\circ \phi_{\alpha}^{-1}$ belongs to $W^{j+m,p}(\Omega_\alpha)$ and vanishes almost everywhere outside $K_\alpha$. Consequently, by \eqref{eq:embedding_local}, it also belongs to $W^{j,q}(\Omega_\alpha)$ and satisfies
\[
\|\psi_{\alpha}F^{a}_{\alpha}\circ \phi_{\alpha}^{-1}\|_{W^{j,q}(\Omega_\alpha)}
\leq C_{\alpha}\|\psi_{\alpha}F^{a}_{\alpha}\circ \phi_{\alpha}^{-1}\|_{W^{j+m,p}(\Omega_\alpha)}.
\]
Applying Lemma \ref{lem:meyers-serrin-local} again, since the component functions are in $W^{j,q}$, we conclude that $(\psi_{\alpha}F)|_{U_{\alpha}}\in W^{j,q}(E|_{U_\alpha})$.
Furthermore, there exist constants $K_{\alpha,j,m}, K_{\alpha,j}>0$ such that
\[
\|\psi_{\alpha}F\|_{W^{j,q}(E|_{U_{\alpha}})} \leq K_{\alpha,j}\sum_{a=1}^{r}\|\psi_{\alpha}F^{a}_{\alpha}\circ\phi^{-1}_{\alpha}\|_{W^{j,q}(\Omega_\alpha)},
\]
and
\[
K_{\alpha,j,m}\sum_{a=1}^{r}\|\psi_{\alpha}F^{a}_{\alpha}\circ\phi_{\alpha}^{-1}\|_{W^{j+m,p}(\Omega_\alpha)} \leq \|\psi_{\alpha}F\|_{W^{j+m,p}(E|_{U_{\alpha}})}.
\]

For each $\alpha\in\{1,\dots,N\}$, the section $(\psi_{\alpha}F)|_{U_{\alpha}}$ vanishes almost everywhere outside the compact set $\operatorname{supp}\psi_{\alpha}\Subset U_{\alpha}$. Its extension by zero is $\psi_{\alpha}F$. Applying Remark~\ref{rem:zero-extension} with Sobolev order $j$ and exponent $q$, we obtain $\psi_{\alpha}F\in W^{j,q}(E)$, whose weak derivatives of orders $0\leq s\leq j$ are the zero extensions of the corresponding local weak derivatives. The local and global $W^{j,q}$ norms coincide. Since $F=\displaystyle\sum_{\alpha=1}^{N}\psi_{\alpha}F$, it follows that $F\in W^{j,q}(E)$.

Finally, we estimate the norm:
\begin{align*}
\|F\|_{W^{j,q}(E)} &= \left\|\sum_{\alpha=1}^{N}\psi_{\alpha}F\right\|_{W^{j,q}(E)} \leq \sum_{\alpha=1}^{N}\|\psi_{\alpha}F\|_{W^{j,q}(E)} = \sum_{\alpha=1}^{N}\|\psi_{\alpha}F\|_{W^{j,q}(E|_{U_{\alpha}})} \\
&\leq \sum_{\alpha=1}^{N}\sum_{a=1}^{r}K_{\alpha,j}\|\psi_{\alpha}F^{a}_{\alpha}\circ\phi_{\alpha}^{-1}\|_{W^{j,q}(\Omega_\alpha)} \\
&\leq \sum_{\alpha=1}^{N}\sum_{a=1}^{r}K_{\alpha,j}C_{\alpha}\|\psi_{\alpha}F^{a}_{\alpha}\circ\phi_{\alpha}^{-1}\|_{W^{j+m,p}(\Omega_\alpha)}.
\end{align*}
Defining $K:=\displaystyle\max_{1\leq \alpha \leq N}\{K_{\alpha,j}C_{\alpha}\}$ and using the lower bound for the local Sobolev norm:
\[
\|F\|_{W^{j,q}(E)} \leq K \sum_{\alpha=1}^{N}\frac{1}{K_{\alpha,j,m}}\|\psi_{\alpha}F\|_{W^{j+m,p}(E|_{U_{\alpha}})} \leq C\sum_{\alpha=1}^{N}\|\psi_{\alpha}F\|_{W^{j+m,p}(E)},
\]
where $C$ is a constant independent of $F$.

To conclude continuity for the original section $F\in W^{j+m,p}(E)$, we use the weak product estimate of Remark~\ref{rem:smooth-multipliers}:
\[
 |\nabla_w^{s}(\psi_{\alpha}F)|_{g,h_{E}}
 \leq \sum_{t=0}^{s}\binom{s}{t}|\nabla^{s-t}\psi_\alpha|_{g}
       |\nabla_w^{t}F|_{g,h_{E}},\qquad 0\leq s\leq j+m.
\]
All derivatives of the finitely many $\psi_\alpha$ through order $j+m$ are bounded on $M$. Taking $L^p$ norms and summing the derivative orders gives a constant $C'>0$, independent of $F$ and $\alpha$, such that
\begin{equation}\label{eq:partition_bound}
 \|\psi_{\alpha}F\|_{W^{j+m,p}(E)}\leq C'\|F\|_{W^{j+m,p}(E)},
 \qquad \alpha\in\{1,\dots,N\}.
\end{equation}
Substituting this bound into the preceding estimate yields
\[
 \|F\|_{W^{j,q}(E)}\leq NCC'\|F\|_{W^{j+m,p}(E)},
\]
which proves continuity.
\end{proof}

We also state and prove the Rellich--Kondrashov compactness theorem for vector bundles.

\begin{theorem}[Rellich--Kondrashov for Vector Bundles]\label{thm:rellich_kondrashov_bundle}
Let $(M,g)$ be a compact Riemannian manifold of dimension $n$ with or without smooth boundary, and let $E\longrightarrow M$ be a smooth vector bundle with fiber metric $h_{E}$ and compatible connection $\nabla^{E}$. Then, for any integers $j\geq0$ and $m\geq1$, and real numbers $1\leq p,q<\infty$ satisfying
\[
 \frac{1}{q}>\frac{1}{p}-\frac{m}{n},
\]
the inclusion
\[
W^{j+m,p}(E)\subseteq W^{j,q}(E)
\]
is compact. In particular, $W^{k,p}(E)\subseteq W^{k-1,p}(E)$ is compact for every integer $k\geq1$ and $1\leq p<\infty$.
\end{theorem}

\begin{proof}
Fix $j,m,p,q$ as in the statement. Consider the same finite cover $(U_{\alpha},\phi_{\alpha})_{\alpha=1}^{N}$ and partition of unity $(\psi_{\alpha})_{\alpha=1}^{N}$ as before. Theorem~\ref{thm:sobolev_embedding_bundle} gives the continuous embedding $W^{j+m,p}(E)\subseteq W^{j,q}(E)$.

Let $(F_{k})_{k\in \mathbb{N}}\subseteq W^{j+m,p}(E)$ be a bounded sequence; that is, there exists $C_{0}>0$ such that $\|F_{k}\|_{W^{j+m,p}(E)}\leq C_{0}$ for all $k$.

For each $k\in \mathbb{N}$ and each $\alpha \in \{1,\dots,N\}$, we have $(\psi_{\alpha}F_{k})|_{U_{\alpha}}\in W^{j+m,p}(E|_{U_{\alpha}})$ and $(\psi_{\alpha}F_{k})|_{U_{\alpha}}\in W^{j,q}(E|_{U_{\alpha}})$. By Lemma~\ref{lem:meyers-serrin-local}, there exist constants $K_{\alpha,j,m} > 0$ and $K_{\alpha,j} > 0$ such that the following norm estimates hold:
\begin{equation}\label{eq:norm_equiv_rellich}
\begin{split}
K_{\alpha,j,m} \sum_{a=1}^{r} \|\psi_{\alpha}(F_{k})_{\alpha}^{a}\circ\phi_{\alpha}^{-1}\|_{W^{j+m,p}(\Omega_\alpha)} &\leq \|\psi_{\alpha}F_{k}\|_{W^{j+m,p}(E|_{U_{\alpha}})}, \\
\|\psi_{\alpha}F_{k}\|_{W^{j,q}(E|_{U_{\alpha}})} &\leq K_{\alpha,j} \sum_{a=1}^{r} \|\psi_{\alpha}(F_{k})_{\alpha}^{a}\circ\phi_{\alpha}^{-1}\|_{W^{j,q}(\Omega_\alpha)}.
\end{split}
\end{equation}

By combining the first inequality of \eqref{eq:norm_equiv_rellich} with the estimate \eqref{eq:partition_bound}, it follows that for each $\alpha\in \{1,\dots,N\}$ and  each component index $a \in \{1,\dots,r\}$, the sequence of functions $\left( \psi_{\alpha}(F_{k})_{\alpha}^{a}\circ\phi_{\alpha}^{-1} \right)_{k\in\mathbb{N}}$ is bounded in $W^{j+m,p}(\Omega_\alpha)$:
\[
\|\psi_{\alpha}(F_{k})_{\alpha}^{a}\circ\phi_{\alpha}^{-1}\|_{W^{j+m,p}(\Omega_\alpha)}
\leq \frac{1}{K_{\alpha,j,m}}\|\psi_{\alpha}F_{k}\|_{W^{j+m,p}(E)}
\leq \frac{C' C_{0}}{K_{\alpha,j,m}}.
\]
For an interior chart, the scalar Rellich--Kondrashov theorem \cite[Theorem~1.2.28, p.~40]{DrabekMilota2007} applies to each weak derivative of order at most $j$, which forms a bounded sequence in $W^{m,p}(\Omega_\alpha)$. Successive extraction over these finitely many derivatives gives a subsequence converging in $W^{j,q}(\Omega_\alpha)$ for each localized component. For a boundary chart, all functions $\bigl(\psi_\alpha(F_k)_\alpha^a\bigr)\circ\phi_\alpha^{-1}$ vanish almost everywhere outside the same compact set $K_\alpha=\phi_\alpha(\operatorname{supp}\psi_\alpha)$. The strict inequality in the statement allows us to apply the compactness assertion of Lemma~\ref{lem:boundary-local-embeddings}, with $\ell=j+m$ and $k=j$. Thus each component has a subsequence converging in $W^{j,q}(\Omega_\alpha)$. There are finitely many charts $\alpha\in\{1,\dots,N\}$ and finitely many fiber components
$a\in\{1,\dots,r\}$, so we may perform successive subsequence extractions.
More precisely, by iteratively refining the sequence $(F_k)$ a finite number of times,
we obtain a subsequence (still denoted by $(F_k)$) such that, for each chart
$\alpha\in\{1,\dots,N\}$ and each fiber component $a\in\{1,\dots,r\}$, the sequence $\psi_{\alpha}(F_{k})_{\alpha}^{a}\circ\phi_{\alpha}^{-1}$ converges —and hence is Cauchy— in $W^{j,q}(\Omega_\alpha)$.

Let $\varepsilon>0$. Since the subsequence is Cauchy component-wise, there exists $N'\in \mathbb{N}$ such that for all $k,l\geq N'$,
\[
\sum_{a=1}^{r}\|\psi_{\alpha}(F_{k}-F_{l})_{\alpha}^{a}\circ\phi^{-1}_{\alpha}\|_{W^{j,q}(\Omega_\alpha)} < \frac{\varepsilon}{N K_{\alpha,j}}, \qquad \forall\alpha\in\{1,\dots,N\}.
\]
Then, using Remark~\ref{rem:zero-extension} to identify the local and global norms of the localized differences, we obtain:

\[\|F_{k}-F_{l}\|_{W^{j,q}(E)} = \left\|\sum_{\alpha=1}^{N}\psi_{\alpha}(F_{k}-F_{l})\right\|_{W^{j,q}(E)} \leq \sum_{\alpha =1}^{N}\|\psi_{\alpha}(F_{k}-F_{l})\|_{W^{j,q}(E)}\] \[\leq \sum_{\alpha=1}^{N} K_{\alpha,j}\sum_{a=1}^{r}\|\psi_{\alpha}(F_{k}-F_{l})_{\alpha}^{a}\circ\phi_{\alpha}^{-1}\|_{W^{j,q}(\Omega_\alpha)} < \sum_{\alpha=1}^{N}\frac{\varepsilon}{N} = \varepsilon.\]
Thus, $(F_{k})_{k\in \mathbb{N}}$ is a Cauchy sequence in the Banach space $W^{j,q}(E)$, and therefore converges. We conclude that the embedding is compact.
\end{proof}
\begin{remark}[Noncompact manifolds]\label{rem:noncompact-embeddings}
For a manifold without boundary, bounded geometry means a positive injectivity radius and uniform bounds on $\nabla^k R^M$ for every $k\geq0$; for the bundle one requires the corresponding bounds on $\nabla^k R^E$. These are the hypotheses of \cite[Definitions~3.1 and~3.3]{GrosseKohrNistor2026}. Under the index conditions stated in their Proposition~3.9, Gro\ss e, Kohr, and Nistor obtain global continuous Sobolev embeddings in part~(i) and a Rellich--Kondrashov--Lions interpolation estimate in part~(ii), for exponents strictly between $1$ and $\infty$. Their Proposition~3.8 supplies the uniform localization used in the proof. These are global norm estimates; bounded geometry alone does not imply global compactness.

For smooth boundary, the bounded-geometry definition and its intrinsic characterization are given in \cite[Definition~2.5 and Theorem~2.10]{AmmannGrosseNistor2019}. Proposition~2.32 there gives the equivalent localized $H^k$ norm in Fermi and interior charts, which is used in their analysis of the Laplacian.
\end{remark}

\section{Applications: Green's Formula and Norm Equivalence}\label{sec:applications}

In the Euclidean setting, it is well known that the integration by parts formula is a fundamental pillar of Sobolev space theory and nonlinear analysis, particularly in the study of partial differential equations (PDEs). A classic application of this formula is establishing the equivalence between the standard Sobolev norm and a structured norm based on powers of the Laplacian. For a bounded domain $\Omega \subset \mathbb{R}^n$ and any function $u \in H_0^m(\Omega)$, the standard Sobolev norm is given by
\begin{equation}\label{eq: euclidean sobolev equivalence} \|u\|_{H^m(\Omega)}^2=\sum_{j=0}^m\int_\Omega|\nabla^ju|^2\,d\lambda_n=\sum_{\substack{j=0\\j\text{ even}}}^m\int_\Omega|\Delta^{\frac{j}{2}}u|^2\,d\lambda_n+\sum_{\substack{j=1\\j\text{ odd}}}^m
       \int_\Omega|\nabla\Delta^{\frac{j-1}{2}}u|^2\,d\lambda_n,
\end{equation}
where $\Delta = -\displaystyle\sum_{i=1}^n \partial_i^2$ is the standard Euclidean Laplacian. These norms and their counterparts on Riemannian manifolds are widely used in the geometric PDE literature; see for example \cite{Gazzola2010}, \cite{Robert2011}, \cite{FPT2024}. In this section, we use the integration by parts formula (Theorem~\ref{thm:green-higher-order}) to generalize Green's identity and subsequently give an explicit, purely geometric proof of this norm equivalence on vector bundles over a closed manifold, accounting for the intrinsic curvature. Finally, as a corollary, we recover the well-known norm equivalence for the Laplace--Beltrami operator on closed manifolds.

\subsection{Green's Formula for the Bochner Laplacian}

An immediate consequence of our main theorem is the standard Green's formula in the context of vector bundles.

\begin{corollary}[Green's Formula for the Bochner Laplacian]\label{cor:green-bochner}
Let $(M,g)$ be a Riemannian manifold with boundary $\partial M$, and let $\pi_E:E\to M$ be a smooth vector bundle equipped with a fiber metric $h_E$ and a compatible connection $\nabla^{E}$. Let $\nu$ denote the outward unit normal to $\partial M$.

Then, for any $u,v\in \Gamma_{c}(E)$, the following identity holds:
\begin{equation*}
\displaystyle\int_{M}\langle \nabla^{E}u,\nabla^{E}v\rangle_{g,h_E}\,d\lambda_g
=
\displaystyle\int_{M}\langle u,\Delta_{B}v\rangle_{h_E}\,d\lambda_g
+
\displaystyle\int_{\partial M}\big\langle u,\nabla^{E}_{\nu}v\big\rangle_{h_E}\,d\lambda_{\widetilde g},
\end{equation*}
where the Bochner Laplacian is the formal differential expression $\Delta_{B} := \nabla^{E*}\nabla^{E}=-\operatorname{tr}_g\nabla^2$. In particular, if $\partial M = \varnothing$, then
\begin{equation*}
\displaystyle\int_{M}\langle \nabla^{E}u,\nabla^{E}v\rangle_{g,h_E}\,d\lambda_g
=
\displaystyle\int_{M}\langle u,\Delta_{B}v\rangle_{h_E}\,d\lambda_g.
\end{equation*}
\end{corollary}

\begin{proof}
We apply Theorem~\ref{thm:green-higher-order} with $s=1$ and $k=l=0$, setting
\[
F=u\in\Gamma_c(E), \quad \text{and} \quad G=\nabla^{E}v\in \Gamma_c(T^{*}M\otimes E).
\]
The identity from the theorem then takes the form:
\begin{equation}\label{eq:proof-step}
\displaystyle\int_{M}\langle \nabla^{E}u,\nabla^{E}v\rangle_{g,h_E}\,d\lambda_g
=
-\displaystyle\int_{M}\big\langle u,(\operatorname{tr}_{g}\circ \nabla)(\nabla^{E}v)\big\rangle_{h_E}\,d\lambda_g
+
\displaystyle\int_{\partial M}\Big\langle u,\,\iota_{\nu}(\nabla^{E}v)\Big\rangle_{h_E}\,d\lambda_{\widetilde{g}}.
\end{equation}
Under the sign convention adopted for the Bochner Laplacian, we have
\[
\Delta_{B} = -(\operatorname{tr}_{g}\circ \nabla) \circ \nabla^{E},
\]
and therefore, the interior term can be rewritten as
\[
-\big\langle u,(\operatorname{tr}_{g}\circ \nabla)(\nabla^{E}v)\big\rangle_{h_E}
=
\langle u,\Delta_{B}v\rangle_{h_E}.
\]

It remains to identify the boundary term. In local coordinates $(x^{1},\dots,x^{n})$, we write the normal vector as $\nu = \nu^{i}\partial_i$. By the definition of the interior multiplication $\iota_{\nu}$ (acting on the last covariant index), we have $\iota_{\nu}(\nabla^{E}v) = \nu^{i}(\nabla^{E}v)_i$. Furthermore, since $(\nabla^{E}v)_i = \nabla^{E}_{\partial_i}v$, we obtain $\iota_{\nu}(\nabla^{E}v) = \nu^{i}\nabla^{E}_{\partial_i}v$. Finally, using the $C^{\infty}(M)$-linearity of the connection with respect to the vector field, it follows that
\[
\nabla^{E}_{\nu}v = \nabla^{E}_{\nu^{i}\partial_i}v = \nu^{i}\nabla^{E}_{\partial_i}v,
\]
which implies $\iota_{\nu}(\nabla^{E}v) = \nabla^{E}_{\nu}v$. Substituting these results into \eqref{eq:proof-step} yields the desired formula.
\end{proof}

\subsection{Technical Preliminaries and the Commutator Lemma}

To handle the contractions produced by commuting covariant derivatives, we make the schematic $*$ notation precise. All tensor bundles in this subsection carry the metrics and connections induced by $g$, $h_E$, the Levi--Civita connection, and $\nabla^E$. In particular, the connection on a dual bundle is the dual connection. Appendix~\ref{app:star-properties} gives the precise meaning of a parallel bundle homomorphism and proves the properties of the tensor operations used here.

\begin{notation}\label{notacion estrella}
Let $G_1,\dots,G_r,H$ be bundles obtained from $TM,T^*M,E,E^*$ by finitely many tensor products. For $u_i\in\Gamma(G_i)$, a term of the form $u_1*\cdots*u_r\in\Gamma(H)$ means
\[
 T=\Phi(u_1\otimes\cdots\otimes u_r),\qquad
 \Phi:G_1\otimes\cdots\otimes G_r\longrightarrow H,
\]
where $\Phi$ is obtained by composing tensor-factor permutations, contractions between a bundle and its dual, and the identifications supplied by $g,h_E$ and their inverses, and by taking finite linear combinations with constant real coefficients. All contractions must have matching tensor types and the stated target bundle $H$.

The symbol does not denote a single product: different occurrences may represent different maps $\Phi$ or different finite sums, including the zero sum. In a schematic identity, it asserts the existence of such fixed contractions and coefficients, rather than an identity for every possible choice of $\Phi$. Variable coefficients, when present, must be displayed as additional tensor factors.
\end{notation}

For example, for $R\in\Gamma(T^{(0,4)}(TM))$ and $\omega\in\Gamma(T^*M)$, the formula
\[
 \Phi(A)_b=g^{ac}g^{de}A_{abcde},\qquad
 \Phi(R\otimes\omega)_b=g^{ac}g^{de}R_{abcd}\omega_e
\]
defines a term $R*\omega\in\Gamma(T^*M)$. It first identifies the tensor product with a covariant five-tensor and then contracts the first and third slots and the fourth and fifth slots. As another example, for the curvature $R^E$ and a section $u$, the expression
\[
 X\longmapsto\sum_{i=1}^n R^E(e_i,X)\nabla^E_{e_i}u
\]
is a term $R^E*\nabla u\in\Gamma(T^*M\otimes E)$, where $(e_i)$ is any orthonormal frame of $TM$. Here the $E^*$ factor of $\operatorname{End}(E)=E\otimes E^*$ is contracted with the $E$ factor of $\nabla u$, and a covariant curvature slot is contracted with the derivative slot using $g^{-1}$.

Two properties of this notation will be used. First, for every compact subset $K\subseteq M$ and each fixed $\Phi$, there is a constant $C_{K,\Phi}>0$ such that
\begin{equation}\label{eq:star-bound}
 |\Phi(u_1\otimes\cdots\otimes u_r)(x)|
 \leq C_{K,\Phi}|u_1(x)|\cdots|u_r(x)|,
 \qquad x\in K.
\end{equation}
Proposition~\ref{prop:star-compact-bound} gives the proof. In the norm estimates below, we apply it with $K=M$.

Second, these natural homomorphisms are parallel: they commute with covariant differentiation for the induced connections. Proposition~\ref{prop:star-leibniz} establishes the resulting Leibniz rule
\[
 \nabla_X\Phi(u_1\otimes\cdots\otimes u_r)
 =\sum_{i=1}^r
   \Phi(u_1\otimes\cdots\otimes\nabla_Xu_i\otimes\cdots\otimes u_r).
\]
The same fixed map $\Phi$ occurs in every summand. This differential identity does not require compactness.

We use the curvature convention
\[
 R^F(X,Y)v=\nabla^F_X\nabla^F_Yv-\nabla^F_Y\nabla^F_Xv
              -\nabla^F_{[X,Y]}v
\]
for any bundle $F$ with a connection. For the Levi--Civita curvature $R^M$, define the Ricci tensor by
\[
 \operatorname{Ric}(X,Y)
 =\sum_{i=1}^n g\bigl(R^M(e_i,X)Y,e_i\bigr),
\]
where $(e_1,\dots,e_n)$ is any orthonormal basis. The Ricci endomorphism $\operatorname{Ric}^{\sharp}:TM\to TM$ is obtained by raising an index with $g$; explicitly, it is characterized by
\[
 g(\operatorname{Ric}^{\sharp}X,Y)=\operatorname{Ric}(X,Y)
 \qquad\text{for all }X,Y.
\]
In coordinates, $(\operatorname{Ric}^{\sharp}X)^i=\displaystyle\sum_{j,k=1}^n g^{ik}\operatorname{Ric}_{jk}X^j$. Skew-adjointness of $R^M(e_i,X)$ gives
\[
 g\left(-\sum_{i=1}^nR^M(e_i,X)e_i,Y\right)
 =\sum_{i=1}^n g\bigl(R^M(e_i,X)Y,e_i\bigr)
 =\operatorname{Ric}(X,Y),
\]
so that $\operatorname{Ric}^{\sharp}X=-\displaystyle\sum_{i=1}^nR^M(e_i,X)e_i$ with our curvature convention.

Put $E_s=T^{(0,s)}(TM)\otimes E$, and denote its induced connection by $\nabla^{E_s}$. We identify the total covariant derivative of a section of $E_s$ with a section of $E_{s+1}$ by placing the new derivative in the last covariant slot. The Bochner Laplacian on $E_s$ is
\[
 \Delta_B^{(s)}=(\nabla^{E_s})^*\nabla^{E_s}
              =-\operatorname{tr}_g\nabla^2,
\]
where $\nabla^2$ is the total second covariant derivative on $E_s$, and the trace contracts its two derivative slots. The superscript $(s)$ specifies the bundle, whereas an unparenthesized exponent denotes an operator power. We retain $\Delta_B=\Delta_B^{(0)}$ on $E$. The curvature on $E_s$ acts by
\begin{align*}
 &(R^{E_s}(X,Y)v)(Z_1,\dots,Z_s)\\
 &\quad=R^E(X,Y)\bigl(v(Z_1,\dots,Z_s)\bigr)
  -\sum_{\alpha=1}^s
       v(Z_1,\dots,R^M(X,Y)Z_\alpha,\dots,Z_s).
\end{align*}
This follows by applying the curvature commutator to tensor products and duals; the minus sign on each covariant slot follows by differentiating its evaluation pairing. Consequently, the curvature of $E_s$ is a sum of actions of the base curvature $R^M$ and the bundle curvature $R^E$.

\begin{lemma}[Commutator of the Laplacian and Covariant Derivatives]\label{lema conmutador laplaciano orden m}
Let $(M,g)$ be a Riemannian manifold with or without smooth boundary, and let $E\to M$ have a fiber metric and a compatible connection. For every integer $m\geq1$ and $u\in\Gamma(E)$,
\begin{equation}\label{eq:comm-higher}
 \Delta_B^{(m)}(\nabla^m u)-\nabla^m(\Delta_Bu)
 =\sum_{k=0}^m
   \bigl(\nabla^kR^M*\nabla^{m-k}u
        +\nabla^kR^E*\nabla^{m-k}u\bigr).
\end{equation}
The right-hand side denotes finite sums of universal contractions as in Notation~\ref{notacion estrella}; their coefficients depend only on the derivative order and tensor types.
\end{lemma}

\begin{proof}
We first prove the one-derivative identity for an arbitrary metric bundle $F$ with a compatible connection $\nabla^F$. Let $v\in\Gamma(F)$, so that $\nabla^Fv\in\Gamma(T^*M\otimes F)$, and write
\[
 C_F(v)=\Delta_B^{T^*M\otimes F}(\nabla^Fv)-\nabla^F(\Delta_B^Fv),
\]
where the superscripts on the Laplacians indicate the bundles on which they act.

We keep the convention that the newest derivative occupies the last covariant slot. Thus the total second covariant derivative is $\nabla^2v=\nabla(\nabla^Fv)$, with
\[
 (\nabla^2v)(Y,Z)
 =\nabla^F_Z\nabla^F_Yv-\nabla^F_{\nabla^M_ZY}v,
\]
and the third derivative is the tensor
\[
 (\nabla^3v)(Y,Z,W)
=\nabla^F_W\bigl((\nabla^2v)(Y,Z)\bigr)-(\nabla^2v)(\nabla^M_WY,Z)
      -(\nabla^2v)(Y,\nabla^M_WZ).
\]
The torsion-free identity $\nabla^M_ZY-\nabla^M_YZ=[Z,Y]$ relates the total second derivative to the curvature convention above:
\begin{equation}\label{eq:curvature-total-derivative}
 \begin{aligned}
 (\nabla^2v)(Y,Z)-(\nabla^2v)(Z,Y)
 &=\nabla^F_Z\nabla^F_Yv-\nabla^F_Y\nabla^F_Zv
       -\nabla^F_{\nabla^M_ZY-\nabla^M_YZ}v\\
 &=\nabla^F_Z\nabla^F_Yv-\nabla^F_Y\nabla^F_Zv
       -\nabla^F_{[Z,Y]}v\\
 &=R^F(Z,Y)v.
 \end{aligned}
\end{equation}
This identity holds for arbitrary local vector fields $Y,Z$ and applies in particular to bundles equipped with induced connections.

Fix $x\in\operatorname{int}(M)$, an orthonormal basis $(e_1,\ldots,e_n)$ of $T_xM$, and $X\in T_xM$. Choose normal coordinates $(y^1,\ldots,y^n)$ centered at $x$ such that $\partial_{y^i}|_x=e_i$. Writing $X=\displaystyle\sum_{i=1}^n X^i e_i$, extend the vectors to local fields, denoted by the same symbols, by
\[
 e_i(y):=\partial_{y^i}|_y,
 \qquad X(y):=\sum_{i=1}^n X^i e_i(y),
\]
where the coefficients $X^i$ are constant. Since the Christoffel symbols vanish at $x$, these extensions satisfy $\nabla^M_Ye_i(x)=\nabla^M_YX(x)=0$ for every $Y\in T_xM$. The frame is orthonormal at $x$; orthonormality away from $x$ is not needed. All identities in the following trace calculation are evaluated at $x$, after taking the indicated covariant derivatives.

At $x$, the metric trace of an $F$-valued covariant two-tensor $T$ is
\[
 \operatorname{tr}_g T
 =\sum_{i,j=1}^n g^{ij}(x)T(e_i,e_j)
 =\sum_{i=1}^n T(e_i,e_i),
\]
since the chosen normal coordinates satisfy $g^{ij}(x)=\delta^{ij}$. For the section $\nabla^Fv$ of $T^*M\otimes F$, its Bochner Laplacian traces the two covariant derivative slots of $\nabla^2(\nabla^Fv)$ and leaves its original one-form slot free. Under the identification with $\nabla^3v$, the contracted slots are therefore the second and third. For the other term, metric compatibility gives
\[
 \nabla^F_X\bigl(\operatorname{tr}_g(\nabla^2v)\bigr)
 =\operatorname{tr}_g\bigl(\nabla_X(\nabla^2v)\bigr)
 =\sum_{i=1}^n(\nabla^3v)(e_i,e_i,X).
\]
Here the trace retains the two slots of $\nabla^2v$, while the new derivative occupies the last slot $X$. Thus the two terms in $C_F(v)$ are
\begin{equation}\label{eq:comm-traces}
 \begin{aligned}
 \bigl(\Delta_B^{T^*M\otimes F}(\nabla^Fv)\bigr)(X)
   &=-\sum_{i=1}^n\bigl((\nabla^2(\nabla^Fv))(e_i,e_i)\bigr)(X)\\
   &=-\sum_{i=1}^n(\nabla^3v)(X,e_i,e_i),\\
 \bigl(\nabla^F\Delta_B^Fv\bigr)(X)
   &=-\nabla^F_X\bigl(\operatorname{tr}_g(\nabla^2v)\bigr)\\
   &=-\sum_{i=1}^n(\nabla^3v)(e_i,e_i,X).
 \end{aligned}
\end{equation}

We compare these traces in two steps. Substituting $Y=X$ and $Z=e_i$ in \eqref{eq:curvature-total-derivative} gives the local tensor identity
\[
 (\nabla^2v)(X,e_i)-(\nabla^2v)(e_i,X)=R^F(e_i,X)v.
\]
Differentiating this identity covariantly in the direction $e_i$ gives, at $x$, the following formula, where $\nabla R^F$ uses the induced connection on $T^{(0,2)}(TM)\otimes\operatorname{End}(F)$:
\begin{equation}\label{eq:comm-first-exchange}
 \begin{aligned}
 &(\nabla^3v)(X,e_i,e_i)-(\nabla^3v)(e_i,X,e_i)\\
 &\qquad=(\nabla_{e_i}R^F)(e_i,X)v
           +R^F(e_i,X)\nabla^F_{e_i}v.
 \end{aligned}
\end{equation}
The terms involving $\nabla^M_{e_i}e_i$ and $\nabla^M_{e_i}X$ vanish at $x$ by the choice of extensions. For the second exchange, apply \eqref{eq:curvature-total-derivative} to the section $\nabla^Fv$ of $T^*M\otimes F$, using the induced connection $\nabla^{T^*M\otimes F}$. The relation to the total third derivative is
\[
 \bigl((\nabla^2(\nabla^Fv))(Y,Z)\bigr)(W)
 =(\nabla^3v)(W,Y,Z).
\]
Taking $Y=X$, $Z=e_i$, and $W=e_i$, and using the product curvature formula, we obtain
\begin{equation}\label{eq:comm-second-exchange}
 \begin{aligned}
 &(\nabla^3v)(e_i,X,e_i)-(\nabla^3v)(e_i,e_i,X)\\
 &\qquad=\bigl(R^{T^*M\otimes F}(e_i,X)(\nabla^Fv)\bigr)(e_i)\\
 &\qquad=R^F(e_i,X)\nabla^F_{e_i}v
           -\nabla^F_{R^M(e_i,X)e_i}v.
 \end{aligned}
\end{equation}
Adding \eqref{eq:comm-first-exchange} and \eqref{eq:comm-second-exchange} and substituting into \eqref{eq:comm-traces} yields
\begin{equation}\label{eq:comm-exact}
 C_F(v)(X)=-\sum_{i=1}^n\bigl\{(\nabla_{e_i}R^F)(e_i,X)v+2R^F(e_i,X)\nabla^F_{e_i}v-\nabla^F_{R^M(e_i,X)e_i}v\bigr\}.
\end{equation}
Every sum is a tensor contraction, independent of the orthonormal basis. Using the Ricci endomorphism defined above, we can equivalently write
\[
 \begin{aligned}
 &\bigl(\nabla^F\Delta_B^Fv
       -\Delta_B^{T^*M\otimes F}(\nabla^Fv)\bigr)(X)\\
 &\qquad=\nabla^F_{\operatorname{Ric}^\sharp X}v
       +2\sum_{i=1}^nR^F(e_i,X)\nabla^F_{e_i}v
       +\sum_{i=1}^n(\nabla_{e_i}R^F)(e_i,X)v.
 \end{aligned}
\]
The scalar case is obtained by taking $F=M\times\mathbb R$ with the canonical connection $\nabla^F_Xf=X(f)$. This connection is flat because
\[
 R^F(X,Y)f=X(Y(f))-Y(X(f))-[X,Y](f)=0.
\]
Thus $R^F=0$ and $\nabla R^F=0$, while $\nabla^Ff=df$ and $\Delta_B^Ff=\Delta_gf$. Under $T^*M\otimes F\cong T^*M$, the induced connection is the usual connection on one-forms. Substituting in \eqref{eq:comm-exact} gives
\[
 \bigl(\Delta_B^{T^*M}(df)-d\Delta_gf\bigr)(X)
 =\sum_{i=1}^n df\bigl(R^M(e_i,X)e_i\bigr)
 =-df(\operatorname{Ric}^{\sharp}X).
\]
Multiplying by $-1$ and using that $X$ is arbitrary yields the identity of one-forms
\[
 d\Delta_g f-\Delta_B^{T^*M}(df)=df\circ\operatorname{Ric}^\sharp,
\]
where $\Delta_g=-\operatorname{tr}_g\nabla^2$ is the nonnegative Laplace--Beltrami operator and $df\circ\operatorname{Ric}^{\sharp}$ is the one-form $X\mapsto df(\operatorname{Ric}^{\sharp}X)$. The metric $g$ is arbitrary: flatness of the line-bundle connection does not make the Ricci tensor vanish. The identities extend to a smooth boundary by continuity.

For the general bundle $E$, take $F=E_s$. Its curvature formula above expresses $R^{E_s}$ in terms of $R^M$ and $R^E$. Covariantly differentiating that formula introduces only $\nabla R^M$ and $\nabla R^E$, since the contractions and identity factors are parallel. Therefore the first two terms in \eqref{eq:comm-exact} give contractions of the forms $\nabla R*v$ and $R*\nabla^{E_s}v$, with $R=R^M$ or $R^E$. Its last term is another contraction of the form $R^M*\nabla^{E_s}v$. Consequently,
\begin{equation}\label{eq:comm_base}
 \Delta_B^{(s+1)}(\nabla^{E_s}v)-\nabla^{E_s}(\Delta_B^{(s)}v)
 =\sum_{R\in\{R^M,R^E\}}\bigl(R*\nabla^{E_s}v+\nabla R*v\bigr).
\end{equation}
Taking $s=0$ and $v=u$ proves \eqref{eq:comm-higher} for $m=1$: the terms $R*\nabla u$ have $k=0$, and the terms $\nabla R*u$ have $k=1$. Suppose now that \eqref{eq:comm-higher} holds for some $m\geq1$. Since $\nabla^{m+1}u=\nabla^{E_m}(\nabla^m u)$, adding and subtracting $\nabla^{E_m}(\Delta_B^{(m)}\nabla^m u)$ gives
\begin{align*}
 &\Delta_B^{(m+1)}\nabla^{m+1}u-\nabla^{m+1}\Delta_Bu\\
 &\quad=(\Delta_B^{(m+1)}\nabla^{E_m}-\nabla^{E_m}\Delta_B^{(m)})(\nabla^m u)
       +\nabla^{E_m}\bigl(\Delta_B^{(m)}\nabla^m u-\nabla^m\Delta_Bu\bigr).
\end{align*}
Apply \eqref{eq:comm_base} with $s=m$ and $v=\nabla^m u$ to the first term. It gives contractions of the forms $R*\nabla^{m+1}u$ and $\nabla R*\nabla^m u$, for $R=R^M$ or $R^E$. For the second term, apply the inductive hypothesis and the Leibniz rule of Proposition~\ref{prop:star-leibniz} to each contraction:
\[
 \nabla(\nabla^k R*\nabla^{m-k}u)
   =\nabla^{k+1}R*\nabla^{m-k}u
       +\nabla^kR*\nabla^{m+1-k}u.
\]
In the first summand on the right, put $\ell=k+1$. As $k$ ranges from $0$ to $m$, these terms have the form $\nabla^\ell R*\nabla^{m+1-\ell}u$ with $1\leq\ell\leq m+1$. The second summand has the same form with $\ell=k$, hence $0\leq\ell\leq m$. The contributions from the first commutator have $\ell=0$ and $\ell=1$. Collecting terms with the same $R$ and $\ell$ therefore yields
\[
 \Delta_B^{(m+1)}(\nabla^{m+1}u)-\nabla^{m+1}(\Delta_Bu)
 =\sum_{\ell=0}^{m+1}
   \bigl(\nabla^\ell R^M*\nabla^{m+1-\ell}u
        +\nabla^\ell R^E*\nabla^{m+1-\ell}u\bigr).
\]
For each fixed $R$ and $\ell$, the sum of the corresponding contractions is again an allowed finite linear combination in Notation~\ref{notacion estrella}; its coefficients include all multiplicities produced above. This is \eqref{eq:comm-higher} with $m$ replaced by $m+1$, completing the induction.
\end{proof}

\subsection{Equivalence of Norms on Closed Manifolds}

\begin{theorem}[Equivalence of Norms via the Bochner Laplacian]
\label{teo:sobolev_cerrada_riemanniana}
Let $(M,g)$ be a closed Riemannian manifold and $E\to M$ a smooth vector bundle equipped with a fiber metric $h_E$ and a compatible connection $\nabla^E$. For any integer $m\in\mathbb N$ and section $u\in H^m(E)$, we define the structured norm:
\[
\|u\|_{\Delta_B}^2
:=
\displaystyle\sum_{\substack{j=0\\ j\text{ even}}}^{m}
\displaystyle\int_M |\Delta_B^{\frac{j}{2}}u|^2\,d\lambda_g
+
\displaystyle\sum_{\substack{j=1\\ j\text{ odd}}}^{m}
\displaystyle\int_M \big|\nabla \Delta_B^{\frac{j-1}{2}}u\big|^2\,d\lambda_g.
\]
Then there exist constants $c,C>0$ such that
\[
c\|u\|_{H^m(E)}^2
\le
\|u\|_{\Delta_B}^2
\le
C\|u\|_{H^m(E)}^2,
\qquad \forall u\in H^m(E).
\]
\end{theorem}

\begin{proof}
It suffices to prove the estimates for smooth sections and then pass to the completion. We write $\|\cdot\|_2$ for the appropriate tensor $L^2$ norm. Constants denoted by $C$ may change from one occurrence to another; they depend on the fixed geometry and the derivative orders, but not on the section. There are no boundary terms because $M$ is closed.

\medskip
\noindent\textit{Step 1: $\|u\|_{\Delta_B}^2\leq C\|u\|_{H^m(E)}^2$.}

By definition, the Bochner Laplacian is the metric trace of the second covariant derivative: $\Delta_B = -\operatorname{tr}_g(\nabla^2)$. Since $\nabla g^{-1}=0$, covariant differentiation commutes with contraction over any fixed pair of tensor slots. Explicitly, with the remaining tensor indices suppressed, the Leibniz rule gives
\[
 \nabla_X(\operatorname{tr}_g T)
 =(\nabla_Xg^{-1})^{pq}T_{pq}+g^{pq}(\nabla_XT)_{pq}
 =\operatorname{tr}_g(\nabla_XT).
\]

Applying this commutativity iteratively for any integer $k \ge 1$, we obtain:
\[
\Delta_B^k u = (-1)^k \operatorname{tr}_g^{(k)}(\nabla^{2k}u),
\]
where $\operatorname{tr}_g^{(k)}$ contracts the pairs of slots $(1,2),(3,4),\dots,(2k-1,2k)$. At each point, in an orthonormal basis $(e_1,\dots,e_n)$, this means
\[
 \operatorname{tr}_g^{(k)}(\nabla^{2k}u)
 =\sum_{i_1,\dots,i_k=1}^n
   (\nabla^{2k}u)(e_{i_1},e_{i_1},\dots,e_{i_k},e_{i_k}).
\]
The same original slots are retained when derivatives are taken; no derivatives are commuted in these identities. This implies that $\Delta_B^k u$ is simply a specific tensor contraction of $\nabla^{2k}u$, which in our notation can be written as $\Delta_B^k u = g^{-1} * \dots * g^{-1} * \nabla^{2k}u$. The pointwise contraction estimate~\eqref{eq:star-bound}, with $K=M$, gives $C_k > 0$ such that $|\Delta_B^k u|_{h_E} \le C_k |\nabla^{2k}u|_{g,h_E}$. Squaring and integrating over $M$:
\[
\displaystyle\int_M |\Delta_B^k u|_{h_E}^2\,d\lambda_g \le C_k^2 \displaystyle\int_M |\nabla^{2k}u|_{g,h_E}^2\,d\lambda_g.
\]
Similarly, for the odd terms, $\nabla \Delta_B^k u = (-1)^k \operatorname{tr}_g^{(k)}(\nabla^{2k+1}u)$, where the same first $2k$ slots are contracted in pairs and the last derivative slot remains free. Bounding this in the exact same manner yields:
\[
\displaystyle\int_M |\nabla \Delta_B^k u|_{g,h_E}^2\,d\lambda_g \le (C_k')^2 \displaystyle\int_M |\nabla^{2k+1}u|_{g,h_E}^2\,d\lambda_g.
\]
Summing over all $j \in \{0, \dots, m\}$ (separating into even and odd $j$), we conclude that each term in $\|u\|_{\Delta_B}^2$ is bounded by the corresponding $j$-th term in the standard Sobolev norm, establishing the first inequality.

\medskip
\noindent\textit{Step 2: $\|u\|_{H^m(E)}^2\leq C\|u\|_{\Delta_B}^2$.}

For each nonnegative integer $j$, denote the truncated structured sum by
\[
 S_j(u)=\sum_{\substack{0\leq l\leq j\\l\text{ even}}}
          \|\Delta_B^{\frac{l}{2}}u\|_2^2
       +\sum_{\substack{1\leq l\leq j\\l\text{ odd}}}
          \|\nabla\Delta_B^{\frac{l-1}{2}}u\|_2^2.
\]
We prove by strong induction that $\|u\|_{H^j(E)}^2\leq C_jS_j(u)$ for every smooth $u$. The assertions for $j=0,1$ are immediate.

We first establish the order-two estimate on any fixed induced bundle $F$. Green's formula, applied to $\nabla v$, gives
\begin{align*}
 \|\nabla^2v\|_2^2
  &=\int_M\langle\nabla v,
          \Delta_B^{T^*M\otimes F}(\nabla v)\rangle\,d\lambda_g\\
  &=\int_M\langle\nabla v,\nabla\Delta_B^Fv\rangle\,d\lambda_g
       +\int_M\langle\nabla v,C_F(v)\rangle\,d\lambda_g.
\end{align*}
Another application of Green's formula turns the first integral into $\|\Delta_B^Fv\|_2^2$. The exact identity \eqref{eq:comm-exact} includes both curvature times $\nabla v$ and the derivative of the bundle curvature times $v$. Since these coefficients are bounded on $M$, it yields
\[
 |C_F(v)|\leq C(|\nabla v|+|v|).
\]
The Cauchy--Schwarz and Young inequalities consequently give
\begin{equation}\label{eq:order-two-estimate}
 \|\nabla^2v\|_2^2
 \leq\|\Delta_B^Fv\|_2^2
       +C\bigl(\|\nabla v\|_2^2+\|v\|_2^2\bigr).
\end{equation}
Taking $F=E$ proves the induction claim for $j=2$.

Now let $j\geq3$ and assume the claim at every smaller order. Apply \eqref{eq:order-two-estimate} on $E_{j-2}$ to $v=\nabla^{j-2}u$. We obtain
\[
 \|\nabla^ju\|_2^2
 \leq\|\Delta_B^{(j-2)}\nabla^{j-2}u\|_2^2
       +C\|u\|_{H^{j-1}(E)}^2.
\]
Lemma~\ref{lema conmutador laplaciano orden m}, the pointwise contraction estimate, and boundedness of the relevant curvature derivatives imply
\[
 \|\Delta_B^{(j-2)}\nabla^{j-2}u\|_2^2
 \leq C\|\nabla^{j-2}\Delta_Bu\|_2^2
       +C\sum_{l=0}^{j-2}\|\nabla^l u\|_2^2.
\]
Here squaring the finite sum only changes the constant $C$. Combining these inequalities and applying the induction claim gives
\begin{align*}
 \|\nabla^ju\|_2^2
 &\leq C\|\Delta_Bu\|_{H^{j-2}(E)}^2
        +C\|u\|_{H^{j-1}(E)}^2\\
 &\leq C S_{j-2}(\Delta_Bu)+C S_{j-1}(u).
\end{align*}
The first application is legitimate because the induction statement holds for every smooth section, including $\Delta_Bu$. Re-indexing by $r=l+2$ gives
\[
 S_{j-2}(\Delta_Bu)
 =\sum_{\substack{2\leq r\leq j\\r\text{ even}}}
      \|\Delta_B^{\frac{r}{2}}u\|_2^2
  +\sum_{\substack{3\leq r\leq j\\r\text{ odd}}}
      \|\nabla\Delta_B^{\frac{r-1}{2}}u\|_2^2
 \leq S_j(u).
\]
Also $S_{j-1}(u)\leq S_j(u)$. Adding the already controlled derivatives of order less than $j$ proves $\|u\|_{H^j(E)}^2\leq C_jS_j(u)$. For $j=m$, this is the desired lower estimate because $S_m(u)=\|u\|_{\Delta_B}^2$. Passage to the completion finishes the proof.
\end{proof}

An immediate and highly useful consequence of this theorem is the corresponding norm equivalence for scalar functions on closed manifolds, which recovers the natural generalization of \eqref{eq: euclidean sobolev equivalence} via the Laplace--Beltrami operator.

\begin{corollary}[Norm Equivalence for the Laplace--Beltrami Operator]
\label{cor:laplace-beltrami}
Let $(M,g)$ be a closed Riemannian manifold. For any integer $m \in \mathbb{N}$ and any scalar function $u \in H^m(M)$, the standard scalar Sobolev norm
\[
\|u\|_{H^m(M)}^2 := \displaystyle\sum_{j=0}^m \displaystyle\int_M |\nabla^j u|_g^2 \, d\lambda_g
\]
is equivalent to the structured norm defined via the Laplace--Beltrami operator $\Delta_g$:
\[
\|u\|_{\Delta_g}^2 := \displaystyle\sum_{\substack{j=0\\ j\text{ even}}}^{m} \displaystyle\int_M |\Delta_g^{\frac{j}{2}}u|^2\,d\lambda_g + \displaystyle\sum_{\substack{j=1\\ j\text{ odd}}}^{m} \displaystyle\int_M \big|\nabla \Delta_g^{\frac{j-1}{2}}u\big|_g^2\,d\lambda_g.
\]
\end{corollary}

\begin{proof}
Consider the trivial line bundle $E = M \times \mathbb{R}$ equipped with the standard Euclidean metric on the fibers and the standard flat connection. Smooth sections of $E$ are naturally identified with smooth scalar functions $u \in C^\infty(M)$. Under this identification, the compatible connection $\nabla^E$ reduces to the standard Riemannian connection (with $\nabla^E u$ corresponding to the differential $du$, whose pointwise norm coincides with that of the gradient $\nabla u$), and the fiber metric $h_E$ is simply standard scalar multiplication. 

Consequently, the Bochner Laplacian $\Delta_B = \nabla^{E*} \nabla^E = -\operatorname{tr}_g(\nabla^2)$ acting on sections of $M \times \mathbb{R}$ coincides exactly with the nonnegative Laplace--Beltrami operator $\Delta_g = -\operatorname{div} \circ \nabla$ acting on scalar functions. The result then follows immediately from applying Theorem \ref{teo:sobolev_cerrada_riemanniana} to this specific bundle.
\end{proof}

\appendix
\section{Localization and Approximation at the Boundary}\label{app:boundary}

Throughout this appendix, $\mathbb R^n_+=\{x_n>0\}$ is the open half-space, $B^+=B(0,R)\cap\mathbb R^n_+$, and $B'=B(0,R)\cap\{x_n=0\}$ is its flat boundary portion. A function smooth up to $B'$ is understood as the restriction of a smooth function defined in an open neighborhood of $B^+\cup B'$ in $\mathbb R^n$. Lemmas~\ref{lem:halfball-zero} and~\ref{lem:halfspace-tools} provide the localization and smooth approximation used in Section~\ref{sec:global_approximation}. Lemma~\ref{lem:boundary-local-embeddings} uses an extension operator to prove the localized continuous and compact embeddings needed in Section~\ref{sec:embeddings_rellich}.

\begin{boundarylemma}[Zero extension within the half-space]\label{lem:halfball-zero}
Let $m\geq0$, $1\leq p<\infty$, and $v\in W^{m,p}(B^+)$. Suppose that $v=0$ almost everywhere in $B^+\setminus K$ for a compact set $K\subset B(0,R)\cap\{x_n\geq0\}$. Define
\[
 \widetilde v(x):=
 \begin{cases}
  v(x),&x\in B^+,\\
  0,&x\in\mathbb R^n_+\setminus B^+.
 \end{cases}
\]
Then $\widetilde v\in W^{m,p}(\mathbb R^n_+)$, and
\[
 D^\beta\widetilde v=\widetilde{D^\beta v},\qquad |\beta|\leq m,
 \qquad
 \|\widetilde v\|_{W^{m,p}(\mathbb R^n_+)}
 =\|v\|_{W^{m,p}(B^+)}.
\]
Here $\widetilde{D^\beta v}$ denotes the same zero extension of $D^\beta v$. The compact set may meet $B'$; the extension is defined only for $x_n>0$, across the spherical edge of the half-ball.
\end{boundarylemma}

\begin{proof}
Let $H:=\{x\in\mathbb R^n\mid x_n\geq0\}$ and $U:=B(0,R)\cap H=B^+\cup B'$. Equip $H$ with the Euclidean metric and the trivial line bundle $E:=H\times\mathbb R$ with its standard fiber metric and canonical connection. Since $B'$ has zero $n$-dimensional measure and weak derivatives are tested in $\operatorname{int}(U)=B^+$, we may regard $v$ as a measurable section on $U$, assigning arbitrary values on $B'$. In these coordinates, weak covariant derivatives have the weak partial derivatives as their components. The corresponding finite-dimensional component norms are equivalent, so this section belongs to $W^{m,p}(E|_U)$.

The set $K$ is compact in the relatively open subset $U\subset H$. Remark~\ref{rem:zero-extension} therefore applies. The restriction to $\mathbb R^n_+$ of the resulting extension is precisely $\widetilde v$, and taking components of the weak derivative identity gives
\[
 D^\beta\widetilde v=\widetilde{D^\beta v},
 \qquad |\beta|\leq m.
\]
Consequently, $\widetilde v\in W^{m,p}(\mathbb R^n_+)$, and
\[
 \begin{aligned}
 \|\widetilde v\|_{W^{m,p}(\mathbb R^n_+)}^p
 &=\sum_{|\beta|\leq m}\int_{\mathbb R^n_+}
          |\widetilde{D^\beta v}|^p\,dx\\
 &=\sum_{|\beta|\leq m}\int_{B^+}|D^\beta v|^p\,dx
 =\|v\|_{W^{m,p}(B^+)}^p.
 \end{aligned}
\]
\end{proof}

\begin{boundarylemma}[Approximation at a flat boundary]\label{lem:halfspace-tools}
Let $m\geq0$ and $1\leq p<\infty$. The restrictions of $C_c^\infty(\mathbb R^n)$ are dense in $W^{m,p}(\mathbb R^n_+)$. Consequently, a function $v$ as in Lemma~\ref{lem:halfball-zero} can be approximated in $W^{m,p}(B^+)$ by functions smooth up to the flat boundary and supported away from the spherical edge.
\end{boundarylemma}

\begin{proof}
The density assertion follows from Leoni's theorem for open sets with continuous boundary, \cite[Theorem~11.35, p.~330]{Leoni2017}, applied to $\mathbb R^n_+$ when $m\geq1$; for $m=0$ it is the usual $L^p$ density. To obtain the localized conclusion, extend $v$ by Lemma~\ref{lem:halfball-zero} and choose $f_k\in C_c^\infty(\mathbb R^n)$ whose restrictions converge to $\widetilde v$ in $W^{m,p}(\mathbb R^n_+)$. Fix $\chi\in C_c^\infty(B(0,R))$ equal to one near $K$, and put $v_k:=(\chi f_k)|_{B^+}$. Since $\chi\widetilde v=\widetilde v$, the scalar multiplier estimate in Remark~\ref{rem:smooth-multipliers} gives
\[
 \begin{aligned}
 \|v_k-v\|_{W^{m,p}(B^+)}
 &\leq\|\chi(f_k|_{\mathbb R^n_+}-\widetilde v)\|_{W^{m,p}(\mathbb R^n_+)}\\
 &\leq C_\chi\|f_k|_{\mathbb R^n_+}-\widetilde v\|_{W^{m,p}(\mathbb R^n_+)}
 \longrightarrow0.
 \end{aligned}
\]
The smooth extensions $\chi f_k$ have support in the fixed compact set $\operatorname{supp}\chi\subset B(0,R)$. Their restrictions to $B^+\cup B'$ therefore have support in $\operatorname{supp}\chi\cap\{x_n\geq0\}$, away from the spherical edge. Their values at the flat boundary are $(\chi f_k)(x',0)$, which need not vanish.
\end{proof}

\begin{boundarylemma}[Localized embeddings at a flat boundary]\label{lem:boundary-local-embeddings}
Let $K\subset B(0,R)\cap\{x_n\geq0\}$ be compact, let $\ell>k\geq0$ be integers, and let $1\leq p,q<\infty$. If $\ell-\frac{n}{p}\geq k-\frac{n}{q}$, there is a constant $C>0$ such that every $v\in W^{\ell,p}(B^+)$ vanishing almost everywhere in $B^+\setminus K$ belongs to $W^{k,q}(B^+)$ and satisfies
\[
 \|v\|_{W^{k,q}(B^+)}\leq C\|v\|_{W^{\ell,p}(B^+)}.
\]
If $\ell-\frac{n}{p}>k-\frac{n}{q}$, every bounded sequence of such functions, with the same compact set $K$, has a subsequence converging in $W^{k,q}(B^+)$.
\end{boundarylemma}

\begin{proof}
Let $\widetilde v$ be the zero extension of $v$ to the open half-space supplied by Lemma~\ref{lem:halfball-zero}. Thus
\[
 \widetilde v\in W^{\ell,p}(\mathbb R^n_+),
 \qquad \|\widetilde v\|_{W^{\ell,p}(\mathbb R^n_+)}
 =\|v\|_{W^{\ell,p}(B^+)}.
\]
This step crosses the spherical edge of $B^+$, but does not cross $x_n=0$. To extend through the flat boundary, we use Stein's extension theorem \cite[Theorem~13.17, pp.~424--425]{Leoni2017}. Since the half-space is a Lipschitz graph domain, the theorem supplies a bounded linear operator
\[
 \mathcal E_\ell:W^{\ell,p}(\mathbb R^n_+)\longrightarrow W^{\ell,p}(\mathbb R^n),
 \qquad (\mathcal E_\ell w)|_{\mathbb R^n_+}=w,
\]
with $\|\mathcal E_\ell w\|_{W^{\ell,p}(\mathbb R^n)}\leq C_{\mathcal E}\|w\|_{W^{\ell,p}(\mathbb R^n_+)}$. No zero boundary condition is required.

Set $B:=B(0,2R)$ and $z:=(\mathcal E_\ell\widetilde v)|_B$. Weak derivatives commute with restriction to open subsets, so $z\in W^{\ell,p}(B)$, $z|_{B^+}=v$, and
\[
 \|z\|_{W^{\ell,p}(B)}
\leq\|\mathcal E_\ell\widetilde v\|_{W^{\ell,p}(\mathbb R^n)}\leq C_{\mathcal E}\|\widetilde v\|_{W^{\ell,p}(\mathbb R^n_+)}
 =C_{\mathcal E}\|v\|_{W^{\ell,p}(B^+)}.
\]
Since $B$ has smooth boundary, the Sobolev embeddings on bounded Lipschitz domains apply to the full spaces $W^{\ell,p}(B)$, without boundary conditions.

Put $m:=\ell-k\geq1$. The index hypothesis is equivalent to $\frac1q\geq\frac1p-\frac mn$. For $q\geq p$, the scalar Sobolev embedding theorem \cite[Theorem~1.2.26 and Remark~1.2.27, pp.~39--40]{DrabekMilota2007}, together with H\"older's inequality on $B$ when needed, gives $W^{m,p}(B)\hookrightarrow L^q(B)$ continuously. For $q<p$, the same scalar estimate follows directly from H\"older's inequality, since $B$ has finite measure. For every multi-index $\alpha$ with $|\alpha|\leq k$, we have $D^\alpha z\in W^{m,p}(B)$ because $|\alpha|+m\leq\ell$. Therefore,
\[
 \|D^\alpha z\|_{L^q(B)}
 \leq C\|D^\alpha z\|_{W^{m,p}(B)}
 \leq C\|z\|_{W^{\ell,p}(B)},
 \qquad |\alpha|\leq k.
\]
Summing the $q$-th powers over these finitely many multi-indices gives $z\in W^{k,q}(B)$ and $\|z\|_{W^{k,q}(B)}\leq C_B\|z\|_{W^{\ell,p}(B)}$. Since $D^\alpha(z|_{B^+})=(D^\alpha z)|_{B^+}$, restriction now yields
\[
 \|v\|_{W^{k,q}(B^+)}\leq\|z\|_{W^{k,q}(B)}\leq C_B\|z\|_{W^{\ell,p}(B)}
 \leq C_BC_{\mathcal E}\|v\|_{W^{\ell,p}(B^+)}.
\]

Now assume the strict inequality and let $(v_j)_{j\in\mathbb N}$ be bounded in $W^{\ell,p}(B^+)$, with every $v_j$ vanishing almost everywhere outside the same compact set $K$. Using the same extension operator for every $j$, set $z_j:=(\mathcal E_\ell\widetilde v_j)|_B$. The preceding estimate gives
\[
 \sup_{j\in\mathbb N}\|z_j\|_{W^{\ell,p}(B)}
 \leq C_{\mathcal E}\sup_{j\in\mathbb N}\|v_j\|_{W^{\ell,p}(B^+)}<\infty.
\]
Consequently, for each multi-index $\alpha$ with $|\alpha|\leq k$, the sequence $(D^\alpha z_j)$ is bounded in $W^{m,p}(B)$. The strict condition $\frac1q>\frac1p-\frac mn$ permits the application of \cite[Theorem~1.2.28, p.~40]{DrabekMilota2007}. If $mp<n$, the strict condition is precisely $q<np/(n-mp)$; if $mp=n$, every finite $q$ is allowed. If $mp>n$, part~(iii) gives compactness in $C(\overline B)$ and hence in $L^q(B)$, since $B$ is bounded and $q<\infty$.

Applying this scalar compactness result to the finitely many derivative sequences and successively extracting subsequences, we obtain a common subsequence, still indexed by $j$, for which every $(D^\alpha z_j)$ is Cauchy in $L^q(B)$. Thus
\[
 \|z_j-z_h\|_{W^{k,q}(B)}^q
 =\sum_{|\alpha|\leq k}
   \|D^\alpha z_j-D^\alpha z_h\|_{L^q(B)}^q
 \longrightarrow0
 \qquad\text{as }j,h\to\infty.
\]
By completeness, $z_j\to z_\infty$ in $W^{k,q}(B)$ for some $z_\infty\in W^{k,q}(B)$. Since $z_j|_{B^+}=v_j$, one final restriction gives
\[
 \|v_j-z_\infty|_{B^+}\|_{W^{k,q}(B^+)}
 \leq\|z_j-z_\infty\|_{W^{k,q}(B)}
 \longrightarrow0.
\]
This proves the compactness assertion.
\end{proof}

For a boundary chart in the proofs of Theorems~\ref{thm:sobolev_embedding_bundle} and~\ref{thm:rellich_kondrashov_bundle}, apply this lemma to $v=(\psi_\alpha F_\alpha^a)\circ\phi_\alpha^{-1}$, with $K=\phi_\alpha(\operatorname{supp}\psi_\alpha)$, $\ell=j+m$, and $k=j$. The set $K$ is independent of the section or sequence being considered. Lemma~\ref{lem:meyers-serrin-local} then transfers the component estimates back to the bundle.

\section{Estimates and the Leibniz Rule for Tensor Contractions}\label{app:star-properties}

We prove the two properties of the $*$ notation used in Section~\ref{sec:applications}.

\begin{starproposition}[Estimate on Compact Subsets]\label{prop:star-compact-bound}
Let $(M,g)$ be a Riemannian manifold with or without boundary, and let $G_1,\dots,G_r,H\to M$ be smooth vector bundles of finite rank, equipped with smooth fiber metrics $h_1,\dots,h_r,h_H$. Let
\[
 \Phi:G_1\otimes\cdots\otimes G_r\longrightarrow H
\]
be a smooth bundle homomorphism obtained from the tensor operations in Notation~\ref{notacion estrella}. For every compact subset $K\subseteq M$, there is a constant $C_{K,\Phi}>0$ such that
\[
 |\Phi_x(v_1\otimes\cdots\otimes v_r)|_{h_H}
 \leq C_{K,\Phi}|v_1|_{h_1}\cdots|v_r|_{h_r},
 \qquad x\in K,\quad v_i\in(G_i)_x.
\]
In particular, if $u_i\in\Gamma(G_i)$ and $T=u_1*\cdots*u_r$ is defined by this homomorphism, then
\[
 |T(x)|_{h_H}
 \leq C_{K,\Phi}|u_1(x)|_{h_1}\cdots|u_r(x)|_{h_r},
 \qquad x\in K.
\]
\end{starproposition}

\begin{proof}
Fix $x\in M$ and consider
\[
 \widetilde\Phi_x:(G_1)_x\times\cdots\times(G_r)_x\longrightarrow H_x,
 \qquad
 \widetilde\Phi_x(v_1,\dots,v_r)=\Phi_x(v_1\otimes\cdots\otimes v_r).
\]
The canonical tensor-product map $\theta_x(v_1,\dots,v_r)=v_1\otimes\cdots\otimes v_r$ is multilinear, and $\Phi_x$ is linear. Thus $\widetilde\Phi_x=\Phi_x\circ\theta_x$ is multilinear. Since the fibers are finite dimensional, it is continuous, so the estimate holds with a constant $C(x)$. It remains to choose these constants uniformly on $K$. If $K$ is empty or one of the bundles has rank zero, the conclusion is immediate; assume otherwise.

For each $p\in K$, choose open sets $U_p,V_p\subseteq M$ such that $p\in U_p$, $\overline{U_p}$ is compact, and $\overline{U_p}\subseteq V_p$, with smooth frames
\[
 e^{(i)}_1,\dots,e^{(i)}_{m_i}\quad\text{for }G_i,
 \qquad
 \varepsilon_1,\dots,\varepsilon_\ell\quad\text{for }H
\]
on $V_p$. Write
\[
 \Phi_x\bigl(e^{(1)}_{a_1}(x)\otimes\cdots\otimes e^{(r)}_{a_r}(x)\bigr)
 =\sum_{b=1}^\ell\Phi^b_{a_1\cdots a_r}(x)\varepsilon_b(x).
\]
The coefficient functions and the frames are smooth on a neighborhood of $\overline{U_p}$. Hence there are constants $A_p,B_p>0$ such that
\[
 |\Phi^b_{a_1\cdots a_r}(x)|\leq A_p,
 \qquad
 |\varepsilon_b(x)|_{h_H}\leq B_p
\]
for all $x\in\overline{U_p}$, $1\leq b\leq\ell$, and $1\leq a_i\leq m_i$ for each $1\leq i\leq r$.

For $x\in\overline{U_p}$, write $v_i=\displaystyle\sum_{a_i=1}^{m_i}v_i^{a_i}e^{(i)}_{a_i}(x)$. Multilinearity gives
\[
 \Phi_x(v_1\otimes\cdots\otimes v_r)
 =\sum_{b=1}^\ell\sum_{a_1=1}^{m_1}\cdots\sum_{a_r=1}^{m_r}
   \Phi^b_{a_1\cdots a_r}(x)v_1^{a_1}\cdots v_r^{a_r}\varepsilon_b(x).
\]
By the triangle inequality,
\begin{align*}
 |\Phi_x(v_1\otimes\cdots\otimes v_r)|_{h_H}
 &\leq A_pB_p\sum_{b=1}^\ell
          \sum_{a_1=1}^{m_1}\cdots\sum_{a_r=1}^{m_r}
          |v_1^{a_1}|\cdots|v_r^{a_r}|\\
 &=A_pB_p\,\ell\prod_{i=1}^r
       \left(\sum_{a_i=1}^{m_i}|v_i^{a_i}|\right).
\end{align*}
The factor $\ell$ comes from the sum over $b$, and the remaining sums factor by distributivity.

The local frame of $G_i$ defines the component norm $\|v_i\|_{(i),x}:=\displaystyle\sum_{a_i=1}^{m_i}|v_i^{a_i}|$. Since the matrix of $h_i$ is smooth and positive definite on a neighborhood of $\overline{U_p}$, compactness gives constants $0<d_{p,i}\leq D_{p,i}<\infty$, independent of $x$, such that
\[
 d_{p,i}|v_i|_{h_i}
 \leq\|v_i\|_{(i),x}
 \leq D_{p,i}|v_i|_{h_i},
 \qquad x\in\overline{U_p}.
\]
Substituting the upper bounds yields
\[
 |\Phi_x(v_1\otimes\cdots\otimes v_r)|_{h_H}
 \leq C_p|v_1|_{h_1}\cdots|v_r|_{h_r},
 \qquad
 C_p:=A_pB_p\,\ell D_{p,1}\cdots D_{p,r}.
\]
Choose a finite subcover $U_{p_1},\dots,U_{p_N}$ of $K$ and set $C_{K,\Phi}=\max\{C_{p_1},\dots,C_{p_N}\}$. This proves the estimate on $K$. The assertion for sections follows by taking $v_i=u_i(x)$.
\end{proof}

\begin{starproposition}[Leibniz Rule for $*$]\label{prop:star-leibniz}
Let $(M,g)$ be a Riemannian manifold with or without boundary, and let $E\to M$ have a fiber metric $h_E$ and a compatible connection. Use the Levi--Civita connection on $TM$ and the induced connections on all tensor products, duals, and homomorphism bundles. Let
\[
 \Phi:G_1\otimes\cdots\otimes G_r\longrightarrow H
\]
be a homomorphism as in Notation~\ref{notacion estrella}. For $u_i\in\Gamma(G_i)$, set $T=\Phi(u_1\otimes\cdots\otimes u_r)=u_1*\cdots*u_r$. Then, for every $X\in\mathfrak X(M)$,
\[
 \nabla_XT
 =\sum_{i=1}^r u_1*\cdots*(\nabla_Xu_i)*\cdots*u_r,
\]
where every summand uses the same fixed homomorphism $\Phi$.
\end{starproposition}

\begin{proof}
Regard $\Phi$ as a section of $\operatorname{Hom}(G_1\otimes\cdots\otimes G_r,H)$ with its induced connection. By definition, for every $S\in\Gamma(G_1\otimes\cdots\otimes G_r)$,
\begin{equation}\label{eq:hom-induced-connection}
 (\nabla_X\Phi)(S)
 =\nabla_X\bigl(\Phi(S)\bigr)-\Phi(\nabla_XS).
\end{equation}
A homomorphism is called parallel if $\nabla_X\Phi=0$ for every $X$; equivalently, it commutes with covariant differentiation in this identity.

Metric compatibility makes the metric identifications and their inverses parallel. Evaluation pairings, permutations, and tensor products also commute with the induced connections. Compositions and finite linear combinations with constant coefficients preserve this property. Since $\Phi$ is formed from these operations, it is parallel.

Taking $S=u_1\otimes\cdots\otimes u_r$ in \eqref{eq:hom-induced-connection} and applying the tensor-product rule gives
\begin{align*}
 \nabla_XT
 &=(\nabla_X\Phi)(S)+\Phi(\nabla_XS)\\
 &=\Phi\left(\sum_{i=1}^r
    u_1\otimes\cdots\otimes\nabla_Xu_i\otimes\cdots\otimes u_r\right)\\
 &=\sum_{i=1}^r
    \Phi(u_1\otimes\cdots\otimes\nabla_Xu_i\otimes\cdots\otimes u_r).
\end{align*}
This is the asserted Leibniz rule.
\end{proof}

\section*{Declarations}

\begin{itemize}
\item Funding: The authors did not receive support from any organization for the submitted work.
\item Conflict of interest/Competing interests: There is no conflict of interest related to the work submitted for publication.
\item Ethics approval and consent to participate: All authors have approved and consented to participate in the work submitted for publication.
\item Consent for publication: All authors have approved and consented to submit the work for publication.
\item Data availability: Not applicable.
\item Materials availability: Not applicable.
\item Code availability: Not applicable.
\item Author contribution: All authors contributed to the study conception and design of the work submitted for publication. The first draft of the manuscript was written by Carlos D. Velázquez-Mendoza and all authors commented on previous versions of the manuscript. All authors read and approved the final manuscript.
\end{itemize}

\bibliography{sn-bibliography}

\end{document}